\documentclass[12pt,parskip=full,twoside]{article}

\usepackage[margin=1in]{geometry}  
\usepackage{graphicx}              
\usepackage{amsmath}               
\usepackage{amsfonts}              
\usepackage{amsthm}                
\usepackage{amssymb}
\usepackage{verbatim}
\usepackage{adjustbox}
\usepackage{parskip} 
\usepackage{mathabx} 
\usepackage{hyperref} 
\usepackage{fancyhdr}
\usepackage{float}

\usepackage{sansmathfonts}
\usepackage[T1]{fontenc}

\fancypagestyle{plain}{\fancyhead{}}

\long\def\symbolfootnote[#1]#2{\begingroup%
\def\thefootnote{\fnsymbol{footnote}}\footnote[#1]{#2}\endgroup}

\theoremstyle{plain}
\newtheorem{thm}{Theorem}[section]
\newtheorem{lem}[thm]{Lemma}
\newtheorem{prop}[thm]{Proposition}
\newtheorem{cor}[thm]{Corollary}

\theoremstyle{definition}
\newtheorem{rmk}[thm]{Remark}
\newtheorem{conv}[thm]{Convention}
\newtheorem{dfn}[thm]{Definition}
\newtheorem{qst}[thm]{Question}

\newcommand{\RR}{\mathbb{R}}      
\newcommand{\ZZ}{\mathbb{Z}}      
\newcommand{\NN}{\mathbb{N}}      
\newcommand{\PP}{\mathbb{P}} 
\newcommand{\PPP}{\mathcal{P}} 

\newcommand{\nc}[1]{\langle\langle#1\rangle\rangle} 
\newcommand{\pres}[1]{\langle#1\rangle} 
\newcommand{\cz}{\text{CAT}(0)}
\newcommand{\cl}[1]{\overline{#1}} 
\newcommand{\intr}[1]{\text{int}(#1)} 
\newcommand{\abs}[1]{\lvert#1\rvert}

\newcommand{\aux}[1]{\widecheck{#1}} 
\newcommand{\ucc}[1]{\bar{#1}} 
\newcommand{\pl}[1]{#1_{\#}} 
\newcommand{\os}[1]{#1^{(1)}} 
\newcommand{\zs}[1]{#1^{(0)}} 
\newcommand{\gos}[1]{G(#1)} 
\newcommand{\rl}{\ell_r} 
\newcommand{\bsl}[1]{\ell(#1)} 
\newcommand{\frgsl}[1]{L(#1)} 
\newcommand{\da}[1]{A(#1)} 
\newcommand{\stab}[1]{\text{stab}(#1)} 
\newcommand{\nbhd}[2]{\mathcal{N}_{#2}(#1)}
\newcommand{\diam}[1]{\text{diam}(#1)} 
\newcommand{\smcan}{C'(\frac{1}{6})} 
\newcommand{\horo}{\mathcal{H}} 
\newcommand{\aug}[1]{A(#1)} 
\newcommand{\coll}[1]{{#1}_c} 

\begin{document}


\title{Cubulating one-relator products with torsion}

\author{Ben Stucky}

\date{}

\maketitle


\begin{abstract}

We generalize results of Lauer and Wise to show that a one-relator product of locally indicable groups whose defining relator has exponent at least $4$ admits a proper and cocompact action on a $\cz$ cube complex if the factors do.

\end{abstract}

\section{Introduction}
\label{sect:intro}

Much effort has been devoted to studying groups which act properly and cocompactly on $\cz$ cube complexes, henceforth referred to as \emph{cubulable groups}, in recent years. Their most famous appearance is in the resolution of the Virtual Haken Conjecture by Agol and Wise, building on work of Bergeron-Wise, Kahn-Markovic, Perelman, Thurston, and others, in which the cubulation of hyperbolic $3$-manifold groups is featured prominently \cite{bw, km, perel1, perel2, thurs}. Simply knowing that a group is cubulable is sufficient to conclude a good deal of structural information about it. For instance, these groups satisfy a Tits alternative \cite{sw}, admit a quadratic-time solution to the word problem \cite{brid}, and satisfy the Novikov and Baum-Connes conjectures \cite{hp, ccjjv}. Cubulable groups which have the stronger property of being \emph{virtually special}, i.e., possess a finite index subgroup which embeds into a right-angled Artin group, enjoy stronger properties still, including separability of quasiconvex subgroups and linearity \cite{wisebook, hsuwise}.

Aside from hyperbolic 3-manifold groups, many classes of groups have been shown to be cubulable, including $\smcan$ small cancellation groups \cite{w1}. One-relator groups with torsion of exponent $n\geq 4$, groups which admit a presentation of the form $\pres{a_1,\ldots,a_m\mid w^n}$ with $n\geq 4$, were cubulated by Lauer and Wise in 2013 \cite{lw}. These groups are $\smcan$ when $n\geq 6$. An extension of Wise's result for $\smcan$ groups was pursued by Martin and Steenbock in 2014 when they successfully cubulated $\smcan$ small cancellation free products of cubulable groups \cite{ms2}. In 2017, Jankiewicz and Wise gave an alternative proof of Martin and Steenbock's result relying on Wise's cubical small cancellation theory developed in \cite{w2}, though they only proved it for $C'(\frac{1}{20})$ small cancellation free products \cite{jw}. In the present article, we generalize Lauer and Wise's cubulation results for one-relator groups with torsion to the free product setting.

A group is \emph{locally indicable} if every finitely generated subgroup admits $\ZZ$ as a homomorphic image. For an element $w$ of a group $G$, let $\nc{w}$ denote the normal closure of $w$ in $G$. The following is our main theorem.

\noindent
\begin{thm}
\label{main}
Let $A$ and $B$ be locally indicable, cubulable groups, $w$ a word in $A*B$ which is not conjugate into $A$ or $B$, and $n\geq 4$. Then $G=A*B/\nc{w^n}$ is cubulable.
\end{thm}

We remark that this is implied by the results of \cite{ms2} when $n\geq 6$ and \cite{jw} when $n\geq 20$.

To prove Theorem \ref{main}, we are motivated to pass to a broader class of groups; namely, we consider ``staggered'' quotients of a free product of finitely many locally indicable, cubulable groups. The topological models for these groups are \emph{staggered generalized $2$-complexes}. See Section \ref{sect:def} for the definition of such a complex $X$ and its \emph{minimal exponent $n(X)$}. Theorem \ref{main} follows from the more general statement below by taking $X$ to be a dumbell space for the free product $A*B$ with a $2$-cell corresponding to $w^n$ glued to it. 

\noindent
\begin{thm}
\label{mainstag}
Let $X$ be a staggered generalized $2$-complex. Suppose that $X$ has locally indicable, cubulable vertex groups and that $n(X)\geq 4$. Then $\pi_1(X)$ is cubulable.
\end{thm}

Wise uses his theory of quasiconvex heirarchies to directly prove a strong generalization of the main result in \cite{lw}, namely that all one-relator groups with torsion are virtually special \cite[Corollary 18.2]{w2}. One-relator groups with torsion are Gromov hyperbolic, so when the exponent of the defining relator in a one-relator group is at least $4$, this result also follows from \cite{lw} and Agol's theorem that a hyperbolic, cubulable group is virtually special \cite[Theorem 1.1]{agol}.

Local indicability of $A$ and $B$ also implies that $G=A*B/\nc{w^n}$ is hyperbolic relative to $\{A,B\}$, a fact we will recover in the present article. Thus if $A$ and $B$ are hyperbolic themselves, then so is $G$ \cite[Corollary 2.41]{o1}, and \cite[Theorem 1.1]{agol} gives the following as a corollary to Theorem \ref{main}:

\noindent
\begin{cor}
Suppose that $A$ and $B$ are locally indicable, hyperbolic, and cubulable. Let $w$ be a word in $A*B$ which is not conjugate into $A$ or $B$, and $n\geq 4$. Then $G=A*B/\nc{w^n}$ is virtually special.
\end{cor}

Though we suspect that Theorem \ref{mainstag} is true when $n(X)\geq 2$, we unfortunately find it necessary to impose the restriction that $n(X)\geq 4$, just as Lauer and Wise do, when seeking to prove properness of the action. In contrast to Lauer and Wise's setting, it also appears that the condition that $n(X)\geq 4$ is necessary for the cocompactness argument.

\noindent
\begin{qst}
Do Theorems \ref{main} and \ref{mainstag} hold when $n(X)\in\{2,3\}$?
\end{qst}

In view of the fact that one-relator groups with torsion are virtually special, the following question is intriguing but well beyond the scope of the present article.

\noindent
\begin{qst}
Let $A$ and $B$ be locally indicable, virtually special groups, $w$ a word in $A*B$ which is not conjugate into $A$ or $B$, and $n\geq 2$. Is $G=A*B/\nc{w^n}$ virtually special?
\end{qst}

\subsection{Methods}

Our methods are topological, and we follow \cite{lw} whenever possible. Briefly, the argument for proving Theorem \ref{main} is as follows. We first build a model space $X$ for $G=A*B/\nc{w^n}$  by starting with a dumbell space $X_A\vee X_B$ of non-positively curved cube complexes with $\pi_1(X_A)=A$ and $\pi_1(X_B)=B$, and then attaching a $2$-cell to a path corresponding to the word $w^n$, so that $\pi_1(X)=G$. See figures \ref{fig1} and \ref{fig2}. The task, then, is to build a $G$-invariant collection of walls in the universal cover, invoke a construction of a dual cube complex with a $G$-action due to Sageev \cite{ms}, and prove that the walls are geometrically nice enough to conclude properness and cocompactness of the action.

 \begin{minipage}{\linewidth}
      \centering
      \begin{minipage}{0.2\linewidth}
          \begin{figure}[H]
              \includegraphics[width=\linewidth]{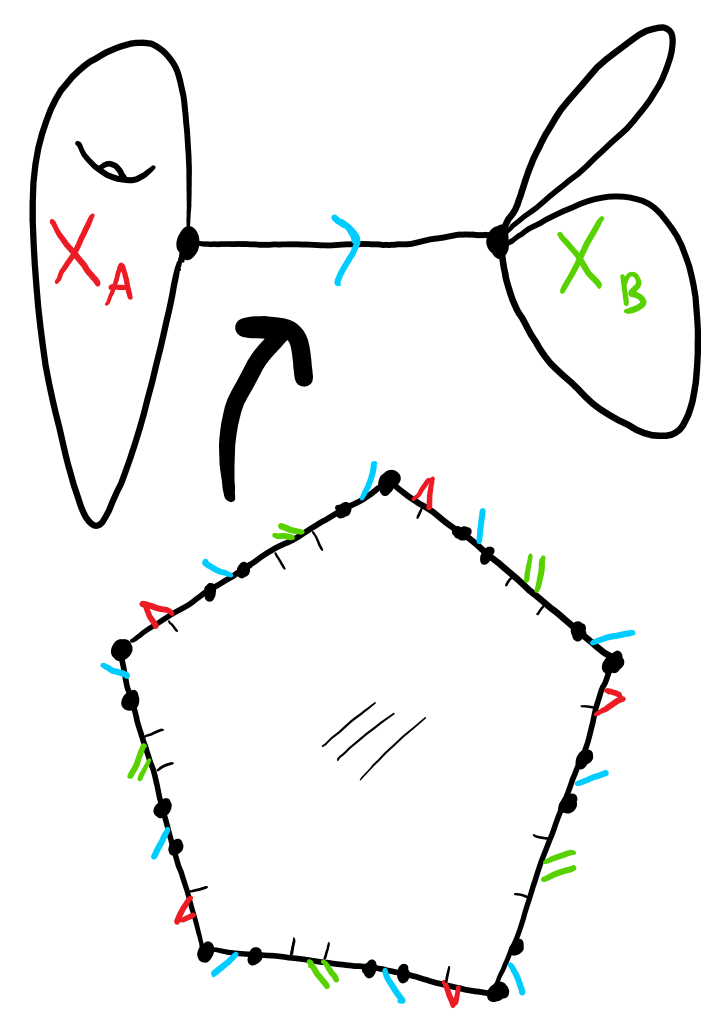}
							\caption{\footnotesize A presentation complex for $G$. The boundary path of the pentagonal cell corresponds to a word of the form $w^5$.}
					\label{fig1}
          \end{figure}
      \end{minipage}
      \hspace{0.2\linewidth}
      \begin{minipage}{0.4\linewidth}
          \begin{figure}[H]
              \includegraphics[width=\linewidth]{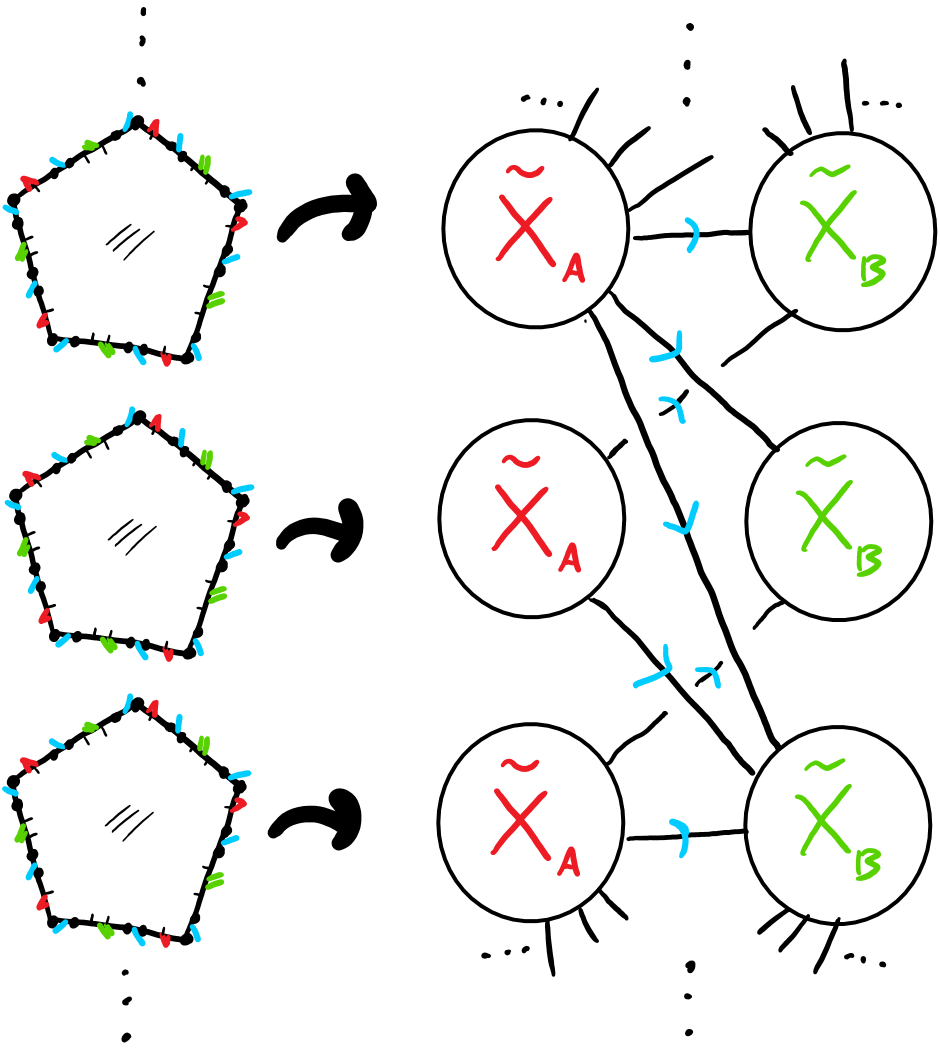}
					\caption{\footnotesize The universal cover of this presentation complex. We build our walls in this space by combining the Lauer-Wise walls considered in \cite{lw} (in the pentagonal cells) with the natural hyperplanes in the $\cz$ cube complex factors $\tilde{X}_A$ and $\tilde{X}_B$.}	
					\label{fig2}
          \end{figure}
      \end{minipage}
  \end{minipage}
	
\subsection{Outline}

We define staggered generalized $2$-complexes in Section \ref{sect:def}. We also define the notion of a \emph{tower} in this section, a fundamental tool for studying these complexes.

Let $G$ be the fundamental group of a staggered generalized $2$-complex $X$ with locally indicable, cubulable vertex groups and minimal exponent $n(X)\geq 4$. We prove geometric small cancellation results about exposed and extreme $2$-cells in generalized van Kampen diagrams over $G$ in Sections \ref{sect:ext} and \ref{sect:ext2}. These are strong statements about the local geometry of staggered generalized $2$-complexes on which the rest of this work depends. These sections are direct generalizations of the work of \cite{lw}. Here the importance of the hypothesis of local indicability will be made clear. The work in this section relies heavily on work of James Howie \cite{h1, h2, h3}.

In Section \ref{sect:rl}, we prove statements about the local geometry of a space $\ucc{X}$ which is essentially the universal cover of $X$, and we develop a tool called \emph{patchings} for producing the kinds of diagrams we can work with to prove results in later sections.

In Section \ref{sect:rh}, we recover relative hyperbolicity of $G$ using Osin's idea of \emph{linear relative Dehn functions} \cite{o1}, which will be important for later arguments. The results up to this point in the outline do not depend on the fact that $X$ has cubulable vertex groups.

We define the walls in $\ucc{X}$ in Section \ref{sect:wl}, combining the Lauer-Wise walls of \cite{lw} with the natural walls in the portions of the universal cover which are already $\cz$ cube complexes. \emph{Ladders} are defined as well -- these are a convenient way to focus our study of the walls on the $2$-skeleton of $\ucc{X}$. We prove that walls embed and separate in Section \ref{sect:eandc}.

We establish necessary conditions for the action on the dual cube complex to be cocompact in Section \ref{sect:relqc}. Here the present work diverges from \cite{lw} significantly in order to deal with the fact that $G$ is not a Gromov hyperbolic group, in general. The fact that $\smcan$ and one-relator groups with torsion are hyperbolic was used critically in \cite{w1} and \cite{lw} to get that the action of $G$ on the dual cube complex is cocompact, in part because quasiconvexity is much easier to characterize in hyperbolic groups. This was also a concern for Martin and Steenbock \cite{ms2}. We prove that wall stabilizers satisfy a property called \emph{relative quasiconvexity}; this turns out to be the key to cocompactness of the action. Importantly, this argument involves attaching \emph{combinatorial horoballs} (defined in \cite{gm}) to $\ucc{X}$ to obtain a hyperbolic space.

In Section \ref{sect:ls}, we show that the walls in $\ucc{X}$ satisfy a properness criterion called \emph{linear separation}, which roughly means that the number of walls separating two points grows linearly in the distance between them.

We put everything together in Section \ref{sect:action}. We use the Sageev construction to produce a dual cube complex with a $G$-action. Since our group is hyperbolic relative to the factors and our walls are relatively quasiconvex, a little more work allows us apply a theorem of Hruska and Wise and prove cocompactness in this more general setting \cite[Theorem 7.12]{hw2}. Linear separation is used to show that the action is proper. Theorem \ref{mainstag} is proved in Theorem \ref{main2} and Theorem \ref{main} is Corollary \ref{main3}.

\subsection{Acknowledgments}

The author wishes to thank Max Forester for his invaluable guidance throughout the duration of this project and without whom this work would not have been possible. He also wishes to thank Paul Plummer and Jing Tao for helpful discussions, and Noel Brady for helpful comments and questions during the post-production phase. Finally, he wishes to thank the faculty and graduate students of Temple University for their hospitality and generosity in providing a place for him to work and discuss mathematics during the 2018 -- 2019 academic year.

\section{Preliminaries}
\label{sect:def}

\noindent
\begin{dfn} (\textbf{Regular map}). Let $X$ be a CW complex. A continuous map $S^1\to X$ is called \emph{regular} if there is a decomposition of $S^1$ such that the map takes vertices to vertices and edges to edges.
\end{dfn}

\noindent
\begin{dfn} (\textbf{Cyclically reduced edge path}). Let $X$ be the total space of a graph of spaces where the vertex spaces are CW complexes and the edge spaces are trivial. A \emph{cyclically reduced edge path} is a regular edge path in $\os{X}$ with no backtracking and with the property that if it contains a path of the form $e\gamma e^{-1}$ where $e$ is an oriented edge between two vertex spaces and $\gamma$ maps to a single vertex space, then $\gamma$ represents a nontrivial element of $\pi_1$ of that vertex space.
\end{dfn}

The following is a more topological definition of a staggered generalized $2$-complex than that given in \cite{hp}.

\noindent 
\begin{dfn} (\textbf{Staggered generalized $2$-complex}). A \emph{staggered generalized $2$-complex} $X$ consists of:
\begin{itemize}
\item The \emph{total space} $\gos{X}$: A graph of spaces where the vertex spaces are CW complexes and the edge spaces $E(X)$ are trivial;
\item A set of $2$-cells $C(X)$ attached to $\os{\gos{X}}$ whose attaching maps are regular, map to cyclically reduced edge paths, and contain an edge of $E(X)$ in their image.
\item A \emph{staggering}:
\begin{itemize}
	\item[$\bullet$] A linear order on $C(X)$,
	\item[$\bullet$] A linear order on $E(X)$,
	\item[$\bullet$] For $c, c'\in C(X)$, if $c<c'$ then $\max(c)<\max(c')$ and $\min(c)<\min(c')$, where $\min(c)$ is defined to be the least edge from $E(X)$ occurring in the attaching map for $c$, and similarly for $\max(c)$.
\end{itemize}
\end{itemize}

We call $C(X)$ the \emph{essential} $2$-cells of $X$ and $E(X)$ the \emph{essential} edges. When comparing cells of $X$ we will sometimes use the notation $<_X$ to refer to the linear orders in the staggering. We will also sometimes write $\max_X(c)$ instead of $\max(c)$ to emphasize the staggering to which we are referring.
\end{dfn}

\noindent
\begin{dfn}
\label{reviseattach} (\textbf{Exponent/proper power/minimal exponent $n(X)$}). For an essential $2$-cell $\alpha$ of $C(X)$, the assumptions on the attaching map of $\alpha$ imply that $R=\partial\alpha$, viewed as an element of $\pi_1(\gos{X})$ for some choice of base-point, is not conjugate into the fundamental group of any vertex space. This implies that $R$ acts loxodromically on the Bass-Serre tree corresponding to $\gos{X}$, i.e., it has positive translation length. This implies that $R$ is not infinitely divisible in $\pi_1(\gos{X})$. Thus there is a well-defined \emph{exponent} $m=m(\alpha)=\max\{k \mid R=w^k \text{ for some } w \in\pi_1(\gos{X}) \}$. If $m\geq 2$ we say that $\alpha$ is attached by a \emph{proper power}. We define the \emph{minimal exponent} $n(X)=\min_{\alpha}m(\alpha)$.

For any cell $\alpha\in C(X)$, we are free to adjust the attaching map by free homotopy in $X$ without affecting $\pi_1(X)$. If the exponent of $\alpha$ is $m$, then the attaching map of $\alpha$ is freely homotopic to an edge path of the form $p^m$. We thus adopt the convention that the attaching map of $\alpha$ is periodic with period $m(\alpha)$.

\end{dfn}

\noindent
\begin{dfn} (\textbf{Indicable/locally indicable}). A group is called \emph{indicable} if it has $\ZZ$ as a quotient, and \emph{locally indicable} if every nontrivial finitely generated subgroup is indicable.
\end{dfn}

\noindent
\begin{dfn} (\textbf{Tower/tower lift/height/maximal}). A \emph{tower} is a map $f:Y\to X$ between connected CW complexes such that $f=i_0\circ p_1\circ i_1\circ\cdots\circ p_n \circ i_n$ where each $i_i$ is an inclusion of a finite subcomplex and each $p_i$ is an infinite cyclic cover. The number $n$ is called the \emph{height} of $f$. Let $K$ and $X$ be connected CW complexes and $\psi:K\to X$ be a map. A \emph{tower lift} is a map $\phi:K\to Y$ such that there is a tower $f:Y\to X$ and $\psi=f\circ \phi$. The map $\phi$ is called \emph{maximal} if any tower lift $\phi':K\to Y'$ of $\phi$ has the property that the associated tower $f':Y'\to Y$ is a homeomorphism.
\end{dfn}

Let $K$ be compact and $\psi:K\to X$ be a combinatorial map between connected CW complexes, that is, the restriction of $\psi$ to the interior of each cell is a homeomorphism. Howie shows \cite[Lemma 3.1]{h1} that $\psi$ has a maximal tower lift $\phi:K\to Y$. Note that a tower lift $\phi:K\to Y$ is not maximal if $\pi_1(K)$ is not indicable and $\pi_1(Y)$ is. Otherwise, $Y$ admits an infinite cyclic cover $Y'\to Y$ corresponding to the kernel of a nontrivial map $\pi_1(Y)\to\ZZ$, and $\phi$ will lift since $\phi_*(\pi_1(K))$ must lie in this kernel.

The following remark is straightforward, since it is easily verified for infinite cyclic covers and inclusions of finite subcomplexes (even with the free homotopy considerations of Definition \ref{reviseattach}).

\noindent
\begin{rmk}
\label{proppow}
If the attaching map of a $2$-cell $\alpha$ in $X$ is a proper power of exponent $n$, then for any $2$-cell $\beta$ in $Y$ with $f(\beta)=\alpha$ under a tower $f:Y\to X$, the attaching map of $\beta$ will be a proper power of the same exponent.
\end{rmk}

The following lemma connects staggered generalized $2$-complexes and towers.

\noindent
\begin{lem} 
\label{towerstag} (cf \cite[Lemma 2]{h3}). If $f:Y\to X$ is a tower and $X$ is a staggered generalized $2$-complex, then so is $Y$.
\end{lem}

\begin{proof} We induct on the number of maps $f$ comprises, so it suffices to assume that $f$ is an inclusion of a connected subcomplex or an infinite cyclic cover. In the first case, note that the staggering of $X$ restricts to a staggering of any subcomplex of $X$. In the second case, let $\rho$ be a generator of the deck group of the cover, and define a staggering on both the $1$-cells and $2$-cells of $Y$ by the prescription that $\alpha <\beta$ if $f(\alpha)<f(\beta)$ (if $f(\alpha)\neq f(\beta)$), or $\rho^n(\alpha)=\beta$ for some positive integer $n$ (if $f(\alpha) =f(\beta)$). This gives a ``lexicographic'' staggering for $Y$.\end{proof}

There may be multiple ways to stagger $Y$. Whenever $Y\to X$ is a tower, we make the convention that the staggering on $Y$ arises in the manner just described.

\section{Some extreme $2$-cells}
\label{sect:ext}

In this section let $X$ be a staggered generalized $2$-complex.

\noindent
\begin{conv} In what follows, when we refer to an $n$-cell $\alpha$ of a CW complex, it should be understood that $\alpha$ refers to the interior of that $n$-cell. When we need to explicitly refer to the closure of a cell $\alpha$, we will use the notation $\cl{\alpha}$.
\end{conv}

\noindent
\begin{lem}
\label{howiecollapse} (cf \cite[Lemma 3]{h3}; \cite[Lemma 2.6]{hw1}). Suppose $X$ is compact, has locally indicable vertex groups, and has at least one essential $2$-cell and no infinite cyclic cover. If the greatest essential $2$-cell $\alpha$ of $X$ is not attached along a proper power in $\pi_1(\gos{X})$, then $X$ collapses across $\alpha$ with free edge $\max\alpha$, i.e., $X$ is homotopy equivalent to the complex obtained after removing $\alpha$ and $\max\alpha$ from $X$ through a homotopy supported on $\cl{\alpha}$. \end{lem}

\begin{proof} We follow Howie's proof in \cite{h3} -- only minor changes are necessary.

Note that if some essential $2$-cell $\beta$ is attached by a proper power $p^n$ in $\gos{X}$, then replacing $\beta$ with the $2$-cell $\beta'$ attached by $p$ will not affect $H^1(X)$, and giving $\beta'$ the same position as $\beta$ in the ordering of the $2$-cells will not affect the staggering of $X$. So we may assume no essential $2$-cell is attached by a proper power.

We induct on the number of essential $2$-cells in $X$. If there is only one, then the rank of $H^1(\gos{X})$ is at most one, since $H^1(X)=0$. If $\gos{X}$ is a tree of spaces, then at most one vertex space can have nontrivial first homology by the Mayer-Vietoris theorem. Also, since the attaching map of $\alpha$ is reduced, cyclically reduced and has positive length, there exists a closed subpath $p'$ of the attaching map $p$ of $\alpha$ which lies in a vertex space $V$ of $\gos{X}$ for which $H^1(V)=0$. Since $p$ is reduced and cyclically reduced, $p'$ represents a nontrivial element $g$ of $\pi_1(V)$. Since $\pi_1(V)$ is locally indicable and finitely generated since $X$ is compact, we obtain a surjective map from $\pi_1(V)$ to $\ZZ$, giving us an infinite cyclic cover of $V$ and contradicting that $H^1(V)=0$. On the other hand, if $\gos{X}$ is not a tree of spaces, then we must have $H^1(V)=0$ for each vertex space and there is a unique simple cycle in $\gos{X}$. The attaching map of $\alpha$ must travel exactly once around this cycle, so that it uses $\max\alpha$ exactly once, and we can see that $X$ collapses across $\alpha$ with free edge $\max\alpha$.

For the inductive step, consider the Mayer-Vietoris sequence 

\[\cdots\to H^1(X)\to H^1(X\setminus \alpha)\oplus H^1(D^2)\to H^1(S^1)\to\cdots\]

associated to attaching $\alpha$ to the rest of $X$. Exactness shows that the rank of $H^1(X\setminus\alpha)$ is at most one. Let $X'$ be the subcomplex of $X$ formed by removing $\alpha$ and $\max\alpha$ from $X$. If $X'$ is connected, then $H^1(X\setminus\alpha)=H^1(X')\oplus\ZZ$, so $H^1(X')=0$. Otherwise $X'$ has two components $X_1$ and $X_2$ (say), and $H^1(X\setminus\alpha)=H^1(X_1)\oplus H^1(X_2)$; assume without loss of generality that $H^1(X_1)=0$. In this case, note that $X_1$ must contain at least one essential $2$-cell whose attaching map lies entirely inside it. If not, then $H^1(X_1)=0$ would imply that $X_1$ were a tree of spaces, with each vertex space having trivial first cohomology. Then since the attaching map $p$ of $\alpha$ uses $X_1$ and is reduced/cyclically reduced, we could find a closed subpath $p'$ of $p$ lying in some vertex space $V$ of $X_1$ such that $p'$ represents a nontrivial element $g$ of $\pi_1(V)$. As before (using compactness of $X$), indicability of $\pi_1(V)$ would lead to an infinite cyclic cover of $V$, contradicting that $H^1(V)=0$.

Thus we may apply the inductive hypothesis either to $X'$ (in case $X'$ is connected) or $X_1$ (in case $X'$ is not connected), but using the staggering \emph{opposite} to that inherited from $X$ (i.e., the orderings of the $1$-cells and $2$-cells are reversed). Then the complex in question collapses across its least essential $2$-cell $\beta$ (in the original ordering) with free edge $\min\beta$. But $\alpha$ does not involve $\min\beta$ since $\beta<\alpha$, so $X$ also collapses across $\beta$ with free edge $\min\beta$. Let $X''=X\setminus\{\beta,\min\beta\}$ be the result of this collapse.

Now $X''$ has fewer essential $2$-cells than $X$, so again apply the inductive hypothesis to $X''$ (using the original ordering) to see that $X''$ collapses across $\alpha$ with free edge $\max\alpha$. But $\beta$ does not involve $\max\alpha$ since $\beta<\alpha$. Thus $X=X''\cup\{\beta,\min\beta\}$ also collapses across $\alpha$ with free edge $\max\alpha$.
\end{proof}

\noindent
\begin{lem}
\label{310}
(cf \cite[Lemma 3.10]{lw}; \cite[Lemma 2.7]{hw1}). Suppose $X$ is compact, has locally indicable vertex groups, and has no infinite cyclic cover. Let $\alpha$ be the greatest essential $2$-cell of $X$. Then $\alpha$ is attached along a path $p^n$ where $p$ is a closed path in $\gos{X}$ passing through $\max(\alpha)$ exactly once. Moreover, no other 2-cell is attached along $\max(\alpha)$. \end{lem}

\begin{proof} The proof is identical to the proof of \cite[Lemma 2.7]{hw1}, except that we appeal to Lemma \ref{howiecollapse} rather than \cite[Lemma 2.6]{hw1}.
\end{proof}

We will now prove some helpful results about van Kampen diagrams in $X$. For our purposes it will be useful to allow diagrams which are not planar. In what follows, the \emph{boundary} of a $2$-complex $E$, denoted $\partial E$, is the closure of the set of $1$-cells in $E$ which occur in the attaching map of at most one $2$-cell of $E$. 

\noindent
\begin{dfn} (\textbf{Cancelable pair/reduced/diagram}). Let $Y$ be a CW complex and $E$ a compact $2$-complex. Let $\phi:E\to Y$ be a combinatorial map. Let $\alpha$ and $\beta$ be a pair of $2$-cells of $E$ with attaching maps $\Phi_\alpha$ and $\Phi_\beta$. We say that $\alpha$ and $\beta$ form a \emph{cancelable pair} if there is a decomposition of $\partial\alpha$ as a loop $e_1\sigma_1$ for some edge $e_1$ and a decomposition of $\partial\beta$ as a loop $e_2\sigma_2$ for some edge $e_2$ such that $\Phi_\alpha(e_1)=\Phi_\beta(e_2)$ and $\phi\circ\Phi_\alpha(\sigma_1)=\phi\circ\Phi_\beta(\sigma_2)$. The map $\phi$ is called \emph{reduced} if $E$ does not contain a cancelable pair. It is called a \emph{diagram} if $E$ is simply connected.

\end{dfn}

The following remark is straightforward.

\noindent
\begin{rmk}
\label{covreduce}
Let $Y$ be a CW complex, $\psi: D\to Y$ a diagram, and $\phi: D\to Z$ a lift of $\psi$ to a cover. Then $\phi$ is reduced if and only if $\psi$ is reduced.
\end{rmk}

Thus we have the following.

\noindent
\begin{rmk}
\label{towreduce} Let $Y$ be a CW complex, $\psi: D \to Y$ a reduced diagram, and $\phi: D\to T$ a maximal tower lift. Then $\phi$ is reduced if and only if $\psi$ is reduced.
\end{rmk}

The following fundamental result is due to van Kampen:

\noindent
\begin{thm}
\label{vkl}
Let $Y$ be a CW complex and let $u$ be a closed path in $Y^{(1)}$. Then $u$ is null-homotopic if and only if there exists a diagram $D\to Y$ with $D$ a planar $2$-complex such that there is a parametrization of $\partial D$ mapping to $u$.
\end{thm}

In the above theorem, we may assume $D$ is reduced if $u$ is a cyclically reduced path, as there are standard moves that we can do to make $D$ reduced \emph{without} affecting $\partial D$.

\noindent
\begin{dfn} (\textbf{Position}). Two $1$-cells $e_1$ and $e_2$ on the boundary of an essential $2$-cell $\alpha$ in $X$ are in the same \emph{position} in $\alpha$ if they are attached to the same $1$-cell of $X$, and a path in $\partial\alpha$ from the terminal $0$-cell of $e_1$ to the terminal $0$-cell of $e_2$ is a cyclic conjugate of $p^j$ for some $j\in\ZZ$. For a $1$-cell $e$ in $\partial\alpha$ we let $[e]_\alpha$ denote the collection of the $n$ $1$-cells in the same position as $e$ in $\alpha$. If $\phi:E\to X$ is a combinatorial map, we extend these definitions to $1$-cells and $2$-cells of $E$ by considering their images under $\phi$.
\end{dfn}

\noindent
\begin{dfn} (\textbf{External/internal/exposed}). Let $\phi:E\to X$ be a combinatorial map. An essential $2$-cell $\alpha$ in $E$ is \emph{external} if there is an essential $1$-cell in $\partial\alpha\cap\partial E$; otherwise it is called \emph{internal}. An essential $2$-cell $\alpha$ in $E$ is \emph{exposed} if there is an essential $1$-cell $e$ in $\partial\alpha$ such that every $1$-cell in $[e]_\alpha$ lies in $\partial E$. We also say $e$ is an exposed edge. By definition, only essential edges can be exposed.
\end{dfn}

Note that if $\phi:E\to X$ is a combinatorial map, then a total order $<_X$ of some cells of $X$ induces an order of the preimages of those essential cells of $X$ in $E$, which we will also denote by $<_X$. Since two cells of $E$ may map to the same cell of $X$, it may be the case that $\alpha =_X \beta$ for cells $\alpha$ and $\beta$ of $E$. In this sense, $<_X$ is a \emph{quasi-order}. Note that by our convention for staggerings associated to towers, if $E\to T$ is a tower lift of $\phi$ and $\alpha<_X\beta$, then $\alpha<_T\beta$ for essential cells $\alpha$ and $\beta$ of $E$.

\noindent
\begin{lem}
\label{4.7} 
(cf \cite[Lemma 4.7]{lw}; \cite[Lemma 4.1]{hw1}). Suppose $X$ has locally indicable vertex groups. Let $\phi:D\to T$ be a maximal tower lift of a reduced diagram $\psi:D\to X$. If $\alpha$ is a greatest (resp. least) $2$-cell of $D$ (under $<_T$), then $\alpha$ is exposed with exposed edge $\max_T\alpha$ (resp. $\min_T\alpha$). In particular, every reduced diagram $D\to X$ with at least one essential $2$-cell has an exposed essential $2$-cell.
\end{lem}

\begin{proof}
Note that $T$ is compact by definition. Let $\alpha'$ be the unique greatest $2$-cell of $T$. By Lemma \ref{310}, $\alpha'$ is the unique $2$-cell whose attaching map uses the edge $\max\alpha'$, and it uses it exactly $n$ times if $n$ is the exponent of $\alpha'$. Let $e$ be an essential $1$-cell of $\alpha$ mapping to $\max\alpha'$. If $\alpha$ is not exposed in $D$, then there is a $2$-cell $\beta$ of $D$ adjacent to $\alpha$ along some essential $1$-cell $e'$ which also maps to $\max\alpha'$. Since $\alpha'$ is the unique $2$-cell using $\max\alpha'$, we must have $\phi(\beta)=\alpha'$. Since the attaching map of $\alpha'$ uses $\max\alpha'$ exactly $n$ times and is a proper power of exponent $n$, we must have that $\sigma_\alpha$, the longer path from the terminal to the initial vertex of $e'$ in $\partial\alpha$, and $\sigma_{\beta}$, the analogous path in $\partial\beta$, must map to the same path in $T$. This shows that $\alpha$ and $\beta$ form a cancelable pair and contradicts that the map $\phi$ is reduced (by Remark \ref{towreduce}).
\end{proof}

\noindent
\begin{dfn} (\textbf{Auxiliary diagram/extreme}). Let $\phi:E\to X$ be a combinatorial map. The \emph{auxiliary diagram} $\aux{E}$ associated to $E$ is obtained from $E$ by collapsing all regions of $E$ which map to vertex spaces of $X$ to points. For any set $S$ of $E$, denote the image of $S$ in $\aux{E}$ by $\aux{S}$. We say that an essential $2$-cell $\alpha$ of $E$ is \emph{extreme} if there is a subpath $\gamma$ of $\partial\alpha=p^n$ (also called \emph{extreme}) such that $\gamma$ contains every $1$-cell in $[e]_\alpha$ for some exposed edge $e$ in $\alpha$, and $\aux{\gamma}$ does not intersect the closure of a $2$-cell in $\aux{E}$ other than the closure of $\aux{\alpha}$, except possibly at its endpoints.
\end{dfn}

\noindent
\begin{rmk}
All extreme $2$-cells are exposed. When $n=1$ the definitions of exposed and extreme coincide.
\end{rmk}

The following basic topological fact will be quite useful throughout this paper. The proof is straightforward.

\noindent
\begin{lem}
\label{snip} (Snipping Lemma) Let $E$ be a simply connected $2$-complex. Let $\gamma$ be an embedded, locally separating arc in $E$ between two points $x$ and $y$ in $\partial E$, and suppose that the interior of $\gamma$ does not intersect $\partial E$.  We call $\gamma$ a \emph{snipping arc}. Then $E\setminus\gamma$ is disconnected (i.e, $\gamma$ is separating). In particular, suppose $\intr{\gamma}\cap E$ is contained in a single $2$-cell $\alpha$, and fix a parametrization $p:S^1\to\partial\alpha$. Let $v$ and $w$ be two points of $S^1$ which lie in distinct components of $S^1\setminus p^{-1}(\gamma)$. Then there is no path from $p(v)$ to $p(w)$ in $E\setminus\gamma$.
\end{lem}

\noindent
\begin{lem}
\label{4.9} (cf \cite[Lemma 4.9]{lw}). Suppose $\phi:E\to X$ is a combinatorial map, $E$ is simply connected, and a $2$-cell $\alpha$ of $E$ is external. Let $B$ be a component of $\cl{E\setminus\cl{\alpha}}$. Then $B\cap\cl{\alpha}$ is connected, $B$ is simply connected, and $\cl{\alpha}$ is simply connected.
\end{lem}

\begin{proof}

Suppose $B\cap\cl{\alpha}$ is disconnected and pick points $v$ and $w$ in distinct components therein. Let $\Gamma$ be the component containing $v$. Fix a parametrization $p:S^1\to\partial\alpha$ and subdivide $S^1$ so that $p$ is a combinatorial map. Let $\lambda$ be a maximal arc of $S^1$ (under inclusion) such that $p(\lambda)=\Gamma$. Let $e$ be the last edge of $S^1$ before $\lambda$ and $f$ be the first edge after $\lambda$. It follows that $p(e)$ and $p(f)$ lie in $\partial E$. Connect two points on the interior of $p(e)$ and $p(f)$ by a snipping arc $\gamma$ through the interior of $\alpha$. The fact that there is a path from $v$ to $w$ in $B$ (thus avoiding $\gamma$) contradicts the Snipping Lemma. Thus $B\cap\cl{\alpha}$ is connected.

Note that $E$ is the union of $B$ and $\cl{E\setminus B}$, and that $B\cap\cl{E\setminus B}=B\cap\cl{\alpha}$. Since $E$ is simply connected, so is $B$ by van Kampen's theorem. This proves the second statement of the lemma.

Note that $\cl{E\setminus B}$ is also simply connected by van Kampen's theorem. Proceeding inductively, let $B_1,\ldots,B_k$ be components of $\cl{E\setminus\cl{\alpha}}$ and observe that $\cl{E\setminus (B_1\cup\ldots\cup B_k)}$ decomposes as the union of $\cl{E\setminus(B_1\cup\ldots\cup B_{k-1})}$ and $B_k$ with connected intersection $B_k\cap\cl{E\setminus (B_1\cup\ldots\cup B_{k-1})}=B_k\cap\cl{\alpha}$. By inductive hypothesis and van Kampen's theorem again, $\cl{E\setminus (B_1\cup\ldots\cup B_k)}$ is simply connected. After finitely many steps we obtain that $\cl{\alpha}$ is simply connected, proving the lemma.
\end{proof}

\noindent
\begin{dfn} (\textbf{Branch}). Let $D\to X$ be a reduced diagram. If $\alpha$ is an exposed $2$-cell of $D$ with exposed edge $e$, then the components of $\cl{D\setminus\cl{\alpha}}$ which contain at least one essential $2$-cell are called the \emph{branches} of $D$ at $(\alpha,e)$.
\end{dfn}

The following is immediate by Lemma \ref{4.9} and van Kampen's Theorem:

\noindent
\begin{lem} \label{nicebranch}
Let $D\to X$ be a reduced diagram, and suppose $\alpha$ is an exposed $2$-cell of $D$ with exposed edge $e$. Let $B$ be a branch of $D$ at $(\alpha,e)$. Then $B\cup\cl{\alpha}$ is simply connected.
\end{lem}

We can now prove our first diagram result:

\noindent
\begin{prop} 
\label{4.11} (cf \cite[Theorem 4.11]{lw}). Let $\psi:D\to X$ be a reduced diagram where $X$ has locally indicable vertex groups, and suppose that $D$ contains at least two essential $2$-cells. Then $D$ contains at least two extreme essential $2$-cells.
\end{prop}

\begin{proof}
The proof is quite similar to that of \cite[Theorem 4.11]{lw}.

We induct on the number of essential $2$-cells in $D$. Let $\phi:D\to T$ be a maximal tower lift of $\psi$, and note that $T$ is compact by definition.

First suppose there are exactly two essential $2$-cells in $D$, $\alpha$ and $\beta$. Then $\alpha$ and $\beta$ are both either greatest or least essential $2$-cells, and so Lemma \ref{4.7} implies that they are both exposed. We claim that $\alpha$ and $\beta$ are both extreme. To see $\alpha$ is extreme, let $e$ be an exposed essential edge of $\alpha$. Let $B$ be the branch of $D$ at $(\alpha,e)$ which contains $\beta$. By Lemma \ref{4.9}, $B\cap\cl{\alpha}$ is contained in an arc of $\partial\alpha$ between two consecutive elements of $[e]_{\alpha}$, $e_1$ and $e_2$. Let $\gamma$ be the arc of $\partial\alpha$ containing $e_1$ and $e_2$ which does not intersect $B$. Note that $\gamma$ contains $[e]_\alpha$. Collapse $D$ to the auxiliary diagram $\aux{D}$, which will have exactly two $2$-cells, $\aux{\alpha}$ and $\aux{\beta}$. Note that $\aux{B}=\cl{\aux{\beta}}$. Since $\gamma$ does not intersect $B$ except possibly at its endpoints, $\aux{\gamma}$ does not intersect the closure of $\aux{\beta}$ except possibly at its endpoints. Thus $\alpha$ is extreme. An identical argument shows $\beta$ is extreme.

For the inductive step, note first that we can find two exposed $2$-cells $\alpha$ and $\beta$ in $D$. Indeed, if $T$ has only one essential $2$-cell, then every essential $2$-cell of $D$ is a greatest $2$-cell and so is exposed by Lemma \ref{4.7}, so choose $\alpha$ and $\beta$ arbitrarily. On the other hand if $T$ has two or more essential $2$-cells, and since $\phi$ is surjective, we can find a $2$-cell in $D$ ($\alpha$, say) mapping to the greatest $2$-cell of $T$, and a $2$-cell in $D$ ($\beta$, say) mapping to the least $2$-cell of $T$; Lemma \ref{4.7} will imply that $\alpha$ and $\beta$ are exposed. If $\alpha$ and $\beta$ are extreme we are done, otherwise assume without loss that $\alpha$ is not extreme. Then for an exposed edge $e$ of $\alpha$, there are at least two branches of $D$ at $(\alpha,e)$ (by Lemma \ref{4.9}). Call them $B_1$ and $B_2$. Now $B_1'=B_1\cup \cl{\alpha}$ and $B_2'=B_2\cup\cl{\alpha}$ are simply connected by Lemma \ref{nicebranch}, and thus $\phi_{|B_i'}$ is a reduced diagram for $i=1,2$ with fewer essential $2$-cells than $\psi$. By the inductive hypothesis there is an extreme essential $2$-cell $\alpha_1\neq\alpha$ in $B_1'$. Observe that $\alpha_1$ is also extreme in $D$ since $\alpha$ separates $B_1$ from all other branches of $D$ at $(\alpha,e)$. Similarly, we can find an extreme cell $\alpha_2\neq\alpha$ in $D$ which lies in $B_2'$. They are distinct since $\alpha_1$ lies in $B_1$ and $\alpha_2$ lies in $B_2$.
\end{proof}

Note: This generalizes part of the Spelling Theorem of Howie and Pride \cite[Theorem 3.1(iii)]{hp}, since the diagrams considered in that paper are planar.

The following is a simple criterion for identifying when an essential $2$-cell in a diagram is \textit{not} extreme. It is straightforward to verify. We will not use it until later.

\noindent
\begin{lem}
\label{notextreme}
Let $\phi:E\to X$ be a combinatorial map and let $\alpha$ be an essential $2$-cell of $E$ with boundary path $p^n$, where the loop $p$ is not a proper power. Suppose that there are two vertices $x$ and $y$ lying in $\partial\alpha$ with the following properties:
\begin{itemize}
\item[(i)] Both paths from $x$ to $y$ in $\partial\alpha$ contain at least as many edges as $p$.
\item[(ii)] Each of the vertices $\aux{x}$ and $\aux{y}$ lies in the closure of at least two essential $2$-cells in $\aux{E}$.
\end{itemize}
Then $\alpha$ is not extreme in $E$.
\end{lem}

\begin{proof}
Let $\gamma$ be a subpath of $\partial\alpha$ such that $\gamma$ contains every $1$-cell in $[e]_\alpha$ for some essential edge $e$ in $\alpha$. Condition (i) implies that either $x$ or $y$ lies in the interior of $\gamma$, and condition (ii) implies that the interior of $\aux{\gamma}$ touches the closures of some $2$-cell of $\aux{E}$ other than the closure of $\aux{\alpha}$. Thus $\alpha$ is not extreme.
\end{proof}

\section{Many extreme $2$-cells}
\label{sect:ext2}

In this section let $X$ be a staggered generalized $2$-complex with locally indicable vertex groups.

\noindent
\begin{dfn} (\textbf{Magnus subcomplex}) (cf \cite[Definition 3.6]{lw}). A \emph{Magnus subcomplex} $Z\subset X$ is a subcomplex with the following properties:
\begin{itemize}
\item[(i)] The subcomplex $Z$ contains the disjoint union of all vertex spaces.
\item[(ii)] If $\alpha$ is an essential $2$-cell of $X$ with the property that all essential boundary $1$-cells of $\alpha$ lie in $Z$, then $\alpha$ lies in $Z$.
\item[(iii)] The essential $1$-cells of $X$ contained in $Z$ form an interval.
\end{itemize}
\end{dfn}

The following lemma is equivalent to Howie's ``locally indicable'' Freiheitssatz \cite[Theorem 4.3]{h1}. We will reprove it for completeness.

\noindent
\begin{lem}
\label{frei} (cf \cite[Theorem 6.1]{hw1}). If $Z$ is a Magnus subcomplex of $X$, then the inclusion $i:Z\to X$ is $\pi_1$-injective for any choice of base-point in $Z$.
\end{lem}

\begin{proof}
We follow the proof in \cite{hw1} -- minimal modifications are necessary.

Let $g\in\ker i_*$. Then any loop $u$ representing $i_*(g)$ is nullhomotopic in $X$, so we may apply Theorem \ref{vkl} to construct a reduced diagram $\psi:D\to X$ where $D$ is a disk and $\psi(\partial D)=u$. We will show that every $2$-cell of $D$ maps to $Z$; this will imply $u$ is nullhomotopic in $Z$ and so $g=1$ in $\pi_1(Z)$.

If every essential $1$-cell in $D$ maps to $Z$ (or no essential $1$-cells appear in $D$), then conditions (i) and (ii) imply that every $2$-cell in $D$ maps to $Z$ and we are done. So suppose there is an essential $1$-cell in $D$ not mapping to $Z$ (for brevity, say ``$D$ has a $1$-cell not in $Z$''). Reversing the staggering of $X$ if necessary, we may assume by condition (iii) that $D$ has a $1$-cell not in $Z$ which is greater than any essential $1$-cell in $Z$. Let $\phi:D\to T$ be a maximal tower lift of $\psi$. Note that for any edge $e\in D$ with the property that $e$ is greater (under $<_X$) than any essential $1$-cell in $Z$, $e$ is greater (under $<_T$) than any essential $1$-cell of $T$ mapping to $Z$ by the tower $T\to X$. Thus the greatest essential $1$-cell of $T$, which we call $e'$, does not map to $Z$. Therefore no edge in $\phi^{-1}(e')$ lies in $\partial D$.

Since $e'$ is in the image of the surjective map $\phi$, this last fact implies that $e'$ must lie on the boundary of some essential $2$-cell in $T$. Thus $e'$ is $\max_T\alpha$ for the greatest essential $2$-cell $\alpha$ of $T$. Applying Lemma \ref{4.7}, $\alpha'$ is exposed in $D$ with exposed edge $e'$. This contradicts that no edge in $\phi^{-1}(e')$ lies in $\partial D$.
\end{proof}

Recall the following fact, the proof of which is technical but requires only Bass-Serre theory and Howie's Freiheitssatz (see \cite{h2}):

\noindent
\begin{lem}
\label{3.4} \cite[Corollary 3.4]{h2} Let $(\mathcal{G}, Y)$ be a graph of groups with trivial edge groups and locally indicable vertex groups. Let $w$ be a cyclically reduced closed word of positive length in $(\mathcal{G}, Y)$, and let $N$ be the normal closure of the subgroup generated by $w$. Then no proper closed subword of $w$ represents an element of $N$.
\end{lem}

A topological interpretation of this gives the following:

\noindent
\begin{lem}
\label{3.9} (cf \cite[Corollary 3.9]{lw}). In $X$, let $p$ be a nontrivial proper subpath of the attaching map of an essential $2$-cell $\alpha$, and suppose that $p$ is a closed path in $X$. Then $p$ is not nullhomotopic in $X$.
\end{lem}

\begin{proof}Let $Z$ be the Magnus subcomplex of $X$ consisting of all vertex spaces and the $2$-cell $\alpha$. Let $Z'$ be the component of $Z$ containing $\alpha$. Then $\pi_1(Z'\setminus\alpha)$ decomposes as a graph of groups satisfying the hypotheses of Lemma \ref{3.4}. Let $w=[\partial\alpha]$. Since $\partial\alpha$ is cyclically reduced, we realize $[p]$ as a proper closed subword of $w$. Applying Lemma \ref{3.4}, $p$ is not nullhomotopic in $Z'$. But $\pi_1(Z)=\pi_1(Z')$ for appropriate choice of base-point, and $\pi_1(Z')$ injects into $\pi_1(X)$ by Lemma \ref{frei}. Thus $p$ is not nullhomotopic in $X$. \end{proof}

Also recall the main theorem from \cite{h2}:

\noindent
\begin{lem}
\label{howiemain} \cite[Theorem 4.2]{h2} Let $A$ and $B$ be locally indicable groups, and let $G$ be the quotient of $A*B$ by the normal closure of a cyclically reduced word $w$ of positive length. Then the following are equivalent:
\begin{itemize}
\item[(i)] $G$ is locally indicable;
\item[(ii)] $G$ is torsion free;
\item[(iii)] $w$ is not a proper power in $A*B$.
\end{itemize}
\end{lem}

Howie mentions the following corollary \cite{h2}:

\noindent
\begin{cor}
\label{locind} (cf \cite[Corollary 4.5]{h2}). Suppose $X$ is such that the attaching map of each essential $2$-cell is not a proper power. Then $\pi_1(X)$ is locally indicable.
\end{cor}

\begin{proof}

Consider the set of all staggered generalized $2$-complexes $X'$ which have all of the same data as $X$, except that $C(X')$ is a finite subset of $C(X)$. Then the set of the groups $\pi_1(X')$ forms a directed system for which $\pi_1(X)$ is the direct limit. Since a direct limit of locally indicable groups is locally indicable, it suffices to assume $C(X)$ is finite.

Induct on the number of essential $2$-cells in $X$.

If there is only one essential $2$-cell, then there are two cases. If $\alpha$ uses some essential edge which separates $\gos{X}$, then let $X_A$ and $X_B$ be the two components. Let $A=\pi_1(X_A)$, $B=\pi_1(X_B)$, and $w=[\partial\alpha]$. Note that $A$ and $B$ decompose as free products of locally indicable groups and are thus locally indicable (by, e.g., the Kurosh subgroup theorem). Now apply Lemma \ref{howiemain} to get the result. Otherwise let $e$ be an essential edge used by $\alpha$. We can see that $\pi_1(\gos{X})$ decomposes as a free product $A*\langle t\rangle$, where $A=\pi_1(\gos{X}\setminus e)$ and $t$ corresponds to a loop with winding number $1$ over $e$. Let $A=\pi_1(X_A)$, $B=\langle t\rangle$, and $w=[\partial\alpha]$. Again observe that $A$ is locally indicable. Lemma \ref{howiemain} again applies to give the result.

For the inductive step, let $\alpha$ be the greatest essential $2$-cell of $X$ and let $e=\max{\alpha}$. Then no other essential $2$-cell uses $e$. If $e$ separates $X\setminus \alpha$, then let $X_A$ and $X_B$ be the two components. Let $A=\pi_1(X_A)$, $B=\pi_1(X_B)$, and $w=[\partial\alpha]$. Now $X_A$ and $X_B$ are staggered generalized $2$-complexes with locally indicable vertex groups and fewer essential $2$-cells, and so $A$ and $B$ are locally indicable by induction. Now apply Lemma \ref{howiemain}. If $e$ does not separate $X\setminus \alpha$, we can see that $\pi_1(X\setminus \alpha)$ decomposes as a free product $A*\langle t\rangle$, where $A=\pi_1(X\setminus\{\alpha,e\}$) and $t$ corresponds to a loop with winding number $1$ over $e$, since no essential $2$-cell uses $e$ except $\alpha$. Let $A=\pi_1(X_A)$, $B=\langle t\rangle$, and $w=[\partial\alpha]$. Again observe that $A$ is locally indicable by the inductive hypothesis. Lemma \ref{howiemain} again applies to give the result.
\end{proof}

We can put these results together and get a strong amplification of Remark \ref{towreduce}:

\noindent
\begin{lem}
\label{4.6} (cf \cite[Lemma 4.6]{lw}). Let $\psi:D\to X$ be a reduced diagram. Let $\phi:D\to T$ be a maximal tower lift of $\psi$. If $\alpha$ and $\beta$ are adjacent essential $2$-cells of $D$ then $\phi(\alpha)\neq\phi(\beta)$.
\end{lem}

\begin{proof}
The proof is in the same spirit as that of \cite[Lemma 4.6]{lw}.

Suppose that $\phi(\alpha)=\phi(\beta)$ and let $e$ be a $1$-cell in $\cl{\alpha}\cap\cl{\beta}$ (essential or not). Observe that $\psi(\alpha)=\psi(\beta)$. Let $p^n$ be the boundary path of $\psi(\alpha)=\psi(\beta)$, where $p$ is not a proper power. By Remark \ref{proppow}, the boundary path of $\phi(\alpha)=\phi(\beta)$ is of the form $\hat{p}^n$ where $\hat{p}$ is a lift of $p$ to $T$. Let $\tau$ be the path of length $\abs{\hat{p}}$ in $\partial\alpha$ which begins at the initial point of $e$ and traverses $e$ in the positive direction. The path $\phi(\tau)$ is a closed loop, and we claim that there is a proper closed subpath of $\phi(\tau)$ in $T$. If the statement ``the path $\tau$ is embedded except possibly at its endpoints'' is false, then this is obvious, so in order to prove the claim, we may assume that $\tau$ is embedded in $D$ except possibly at its endpoints. Consider the set $S$ of edges in $\phi^{-1}(\phi(e))\cap\partial\alpha$ which belong to $\tau$, which is nonempty since it contains $e$. If this set has exactly one element, then $[e]_\alpha$ is the only orbit of edges in $\partial\alpha$ mapping to the edge $\psi(e)$. Since $\psi(\alpha)=\psi(\beta)$, this implies that $\psi^{-1}(\psi([e]_\alpha))\cap\partial\beta =[e]_\beta$ so that $\alpha$ and $\beta$ form a cancelable pair, which contradicts that $D$ is reduced. Thus $S$ contains two distinct elements, and so there are two distinct edges of $\tau$ which become identified under $\phi$. This proves the claim. Thus there is a proper closed subpath $\gamma$ of $\hat{p}$ in $T$. See figure \ref{fig:lem46}.

\begin{figure}[htbp]
	\centering
		\includegraphics{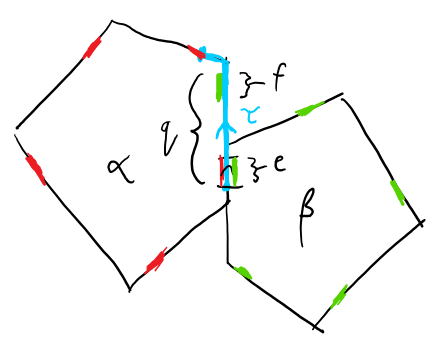}
	\caption{\footnotesize Proving the claim: the fact that $S$ contains two distinct edges $e$ and $f$ implies that the path $q$ contains the desired path $\gamma$, since $e$ and $f$ (in fact, all red and green edges) become identified under $\phi$.}
	\label{fig:lem46}
\end{figure}

Let $X'$ be the $2$-complex associated with $X$ having nonperiodic attaching maps, and consider the map $X\to X'$ which is the identity on the $1$-skeleton of $X$, and an $m$-fold branched cover on each essential $2$-cell if $m$ is the exponent of that $2$-cell. Let $\gamma'$ be the image of $\gamma$ in $X'$. By Lemma \ref{3.9}, $\gamma'$ represents a nontrivial element of $\pi_1(X')$. Thus $\pi_1(T)$ maps to a nontrivial subgroup of $\pi_1(X')$, and that subgroup is finitely generated since $T$ is compact. Since $\pi_1(X')$ is locally indicable by Corollary \ref{locind}, $\pi_1(T)$ is indicable. Thus $T$ has an infinite cylic cover and the tower lift $D\to T$ is not maximal, a contradiction.
\end{proof}

Now we can study connected subdiagrams of a reduced diagram:

\noindent
\begin{lem}
\label{5.1} (cf \cite[Lemma 5.1]{lw}). Let $D\to T$ be a maximal tower lift of a reduced diagram $D\to X$. Let $D'$ be a connected subcomplex of $D$, and let $\alpha$ be a greatest $2$-cell of $D'$. Then $\alpha$ is exposed in $D'$.
\end{lem}

Note: The proof below is slightly more complicated than Lauer and Wise's proof of \cite[Lemma 5.1]{lw}. There, the authors seem to assume that the subcomplex $B$ defined in the proof below is simply connected without justification.

\begin{proof} By Lemma \ref{4.6} applied to the map $D\to T$, each essential $2$-cell adjacent to $\alpha$ in $D'$ is strictly below $\alpha$ (under $<_T$). Let $B$ be the smallest subcomplex of $D'$ containing $\alpha$ and all $2$-cells adjacent to $\alpha$. Let $B'$ be a minimal simply connected subcomplex of $D$ containing $B$ (under inclusion). Let $B'\to T'$ be a maximal tower lift of the composition $B'\hookrightarrow D\to T$, and let $\alpha'$ be a greatest essential $2$-cell of $B'$ under $<_{T'}$. Now Lemma \ref{4.7} implies $\alpha'$ is exposed in $B'$. Note that since all essential $2$-cells in $B\setminus\alpha$ are below $\alpha$ under $<_{T}$, they are also below $\alpha$ under $<_{T'}$. Thus $\alpha'\notin B\setminus\alpha$.  If $\alpha'\neq \alpha$, then consider the component of $\cl{B'\setminus\cl{\alpha'}}$ containing $\alpha$. This subcomplex of $D$ contains $B$, is simply connected (by Lemma \ref{4.9}), and it is strictly contained in $B'$. This violates minimality of $B'$. Thus $\alpha'=\alpha$, so $\alpha$ is exposed in $B'$. But $B'$ contains all $2$-cells in $D'$ adjacent to $\alpha$, so $\alpha$ is also exposed in $D'$.  \end{proof}

For an essential $2$-cell $\alpha$ in a reduced diagram $D\to X$, let $V$ be the preimage in $D$ of the disjoint union of the vertex spaces of $X$, and define the following subcomplexes of $D$:
\[ \widehat{G_\alpha} = \{\cl{\beta}\in D |\beta \geq_X\alpha\}\cup V \]
\[ \widehat{L_\alpha} = \{\cl{\beta}\in D |\beta <_X\alpha\}\cup\{\cl{\alpha}\}\cup V \]

Let $G_\alpha$ and $L_\alpha$ be the components of $\widehat{G_\alpha}$ and $\widehat{L_\alpha}$, respectively, containing $\alpha$.

\noindent
\begin{lem} 
\label{5.3} (cf \cite[Lemma 5.3]{lw}). The components of $\widehat{G_\alpha}$ and $\widehat{L_\alpha}$ are simply connected.
\end{lem}

\begin{proof} The proof is nearly identical to that of \cite[Lemma 5.3]{lw}. We obtain $\widehat{G_\alpha}$ by successively removing the closure of a least essential $2$-cell from $D$ and passing to components of the closure of what remains. Reversing the staggering, Lemma \ref{5.1} ensures that each successive essential $2$-cell will be exposed, and Lemma \ref{4.9} implies that removing each successive cell leaves simply connected components. In finitely many steps we obtain $\widehat{G_\alpha}$, and the argument is essentially the same for $\widehat{L_\alpha}$.
\end{proof}

We are ready to prove our second main diagram theorem:

\noindent
\begin{prop}
\label{5.4} (cf \cite[Theorem 5.4]{lw}). Let $D\to X$ be a reduced diagram. If $D$ has an internal essential $2$-cell that maps to an exponent $n$ $2$-cell of $X$, then $D$ contains at least $2n$ extreme $2$-cells.
\end{prop}

\begin{proof}
The proof is essentially the same as that of \cite[Theorem 5.4]{lw}.

Let $D\to T$ be a maximal tower lift of $D\to X$, and let $\alpha$ be an internal essential $2$-cell of $D$ of exponent $n$. Define $\widehat{G_\alpha}$ and $\widehat{L_\alpha}$ with respect to $<_T$. Now Lemma \ref{5.1} implies that $\alpha$ is exposed in both $G_\alpha$ and $L_\alpha$, so there exist essential $1$-cells $e_G$ and $e_L$ in $\alpha$ such that each $1$-cell in $[e_G]_\alpha$ lies in $\partial G_\alpha$ and each $1$-cell in $[e_L]_\alpha$ lies in $\partial L_\alpha$. Since $\alpha$ is internal, this last statement implies that $[e_G]_\alpha$ and $[e_L]_\alpha$ must be distinct. Since the $n$ elements of $[e_L]_\alpha$ are internal in $G_\alpha$, and because each branch of $G_\alpha$ at $(\alpha,e_G)$ intersects $\partial\alpha$ in an arc (Lemma \ref{4.9}), there are exactly $n$ branches of $G_\alpha$ at $(\alpha,e_G)$. Call them $B_1,\ldots,B_n$. Let $G_i$ be the component of $\widehat{L_\alpha}\cup{B_i}$ containing $\alpha$. Note that $G_i$ contains at least one essential $2$-cell strictly greater than $\alpha$ since $B_i$ contains an essential $2$-cell adjacent to $\alpha$ (applying Lemma \ref{4.6} to $D\to T$). So any greatest $2$-cell of $G_i$ lies in $B_i$. Now Lemma \ref{5.1} implies that there exists an essential $2$-cell $\alpha'$ in $B_i$ which is exposed in $G_i$. Note that $\alpha'$ is exposed in $D$ since if $\beta$ is a $2$-cell of $D$ adjacent to $\alpha'$ and $\beta$ doesn't lie in $\widehat{L_\alpha}$, then $\beta$ is essential and $\beta\geq\alpha$, so $\beta$ lies in $G_i$. Thus we obtain $n$ distinct exposed $2$-cells in $D$, one in each $B_i$, and all strictly greater than $\alpha$. 

We repeat almost the same argument for $L_\alpha$ to obtain $n$ more distinct exposed $2$-cells in $D$, all strictly less than $\alpha$ (in this case, the argument is actually simpler, as we don't need to apply Lemma \ref{4.6}). Thus we obtain $2n$ exposed $2$-cells in $D$. This completes the proof in the case $n=1$, as the definitions of exposed and extreme coincide.

Thus assume $n\geq 2$, and let $\alpha_1,\ldots,\alpha_{2n}$ be exposed $2$-cells of $D$. If $\alpha_i$ is not extreme, then $D$ has at least two branches at $(\alpha_i, e_i)$ for some $e_i$ by Lemma \ref{4.9}. Let $B$ be a branch not containing $\alpha$, and note that $B\cup\cl{\alpha_i}$ is simply connected by Lemma \ref{nicebranch}. By Proposition \ref{4.11}, there are at least two extreme essential $2$-cells in $B\cup\cl{\alpha_i}$; any one of these not equal to $\alpha_i$ is extreme in $D$. Repeating for each $i$, we obtain $2n$ extreme $2$-cells. They are distinct since for $j\neq i$, $\alpha_j$ lies in the branch of $D$ at $(\alpha_i,e_i)$ containing $\alpha$.
\end{proof}

\section{Geometry of the universal cover}
\label{sect:rl}

From now on, we assume that each essential $2$-cell of $X$ is attached by a proper power, that is, $n(X)\geq 2$.

Let $X$ be a staggered generalized $2$-complex with locally indicable vertex groups and such that $n(X)\geq 2$. We will soon be assuming that the vertex groups of $X$ are cubulated. This section contains a collection of results about the geometry of $X$ which do not depend on this assumption.

In what follows, we will be working in the universal cover of $X$ (denoted by $\tilde{X}$), or at least a space with the same one skeleton.

By Lemma \ref{frei}, $\pi_1(V)$ embeds naturally in $\pi_1(X)$ for each vertex space $V$ of $X$, and thus $\gos{\tilde{X}}$ (the preimage of $\gos{X}$ in $\tilde{X}$) decomposes as a graph of spaces with trivial edge spaces, where each vertex space is $\tilde{V}$ for some vertex space $V$ of $X$. Let $\ucc{X}$ be the space obtained from $\tilde{X}$ by identifying elevations of essential $2$-cells of $X$ which have the same boundary; it may be viewed as a subcomplex of $\tilde{X}$ which contains $\gos{\tilde{X}}$. Give $\os{\gos{\tilde{X}}}$ the combinatorial metric in which every edge has length $1$. All of the metric statements in this section are really about $\os{\gos{\tilde{X}}}=\os{\ucc{X}}$, and all paths of interest are edge paths. From now on, let $d$ be the graph metric on $\os{\ucc{X}}$.

Once and for all, for each essential $2$-cell $\alpha$, arrange that lifts of maximal subpaths of $\partial\alpha$ mapping to a vertex space $V$ are geodesics in each $\os{\tilde{V}}$ as follows: Suppose that the exponent of $\alpha$ is $n$, so the boundary $\partial\alpha$ is a path of the form $p^n$, where $p$ is a loop in $\os{\gos{X}}$. For each maximal subpath $p_V$ of $p$ mapping entirely to a vertex space $V$ of $X$, note that $p_V$ is a loop. We modify $p$ by replacing $p_V$ by a loop $p'_V$ in $\os{V}$ with the properties that $p'_V$ has the same basepoint as $p_V$, $p'_V$ and $p_V$ represent the same element of $\pi_1(X)$, and $p'_V$ uses a minimal number of edges. Let $p'$ be the result of modifying $p$ in this way. Replace $\alpha$ by a $2$-cell $\alpha'$ with attaching map $(p')^n$. Doing this for all essential $2$-cells does not affect $\pi_1(X)$, and the resulting staggered generalized $2$-complex has the desired property. Thus we may assume that $X$ has the property that lifts of maximal subpaths of $\partial\alpha$ mapping to a vertex space $V$ are geodesics in each $\os{\tilde{V}}$ for each essential $2$-cell $\alpha$.

In what follows, we refer to cells in $\ucc{X}$ as essential or not according to whether their images in $X$ are essential or not.

\subsection{Admissible pseudometrics and relative geodesics}

We will need to work with paths in $\ucc{X}$ which generalize geodesics. The idea of relative geodesics as defined below is that they allow for the possibility that paths can be ``shorter than they look,'' but only in vertex spaces. At certain times in what follows, we will be ``augmenting'' $\ucc{X}$ and allowing for this sort of behavior.

\noindent
\begin{dfn}
\label{rg} (\textbf{Admissible pseudometrics/relative length/relative geodesic}). Let $d$ denote the metric on $\os{\ucc{X}}$ where every edge has length one. For each vertex space $\tilde{V}$, choose a pseudometric $d_{\tilde{V}}$ on $\tilde{V}^{(0)}$. We require that this choice of pseudometrics is invariant with respect to the action of $G$ on $\ucc{X}$. If this holds we say the choice of pseudometrics is \emph{admissible}.

Let $\gamma:I\to\ucc{X}$ be a path whose endpoints are $0$-cells $x$ and $y$ of $\ucc{X}$. Decompose $\gamma$ as a concatenation $\gamma_{v_1}e_1\ldots\gamma_{v_k}e_k\gamma_{v_{k+1}}$, where each $\gamma_{v_i}$ is a (possibly degenerate) maximal edge path mapping to a vertex space $\tilde{V}_i$ of $\ucc{X}$, and the $e_i$ are essential edges. We define the \emph{relative length} of $\gamma$, $\rl(\gamma)$, by the following formula:

\[\rl(\gamma)=k+\sum_{i=1}^{k+1}d_{\tilde{V}_i}(i(\gamma_{v_i}),t(\gamma_{v_i})),\]

where $i(\lambda)$ and $t(\lambda)$ denote the initial and terminal vertices, respectively, of a path or edge $\lambda$. We say $\gamma$ is a \emph{relative geodesic} if the restriction of $\gamma$ to each vertex space is a geodesic in the one-skeleton of that vertex space, and $\rl(\gamma)$ is minimal among all paths from $x$ to $y$. If we have not made an explicit choice of admissible pseudometrics on vertex spaces, the statement that $\gamma$ is a relative geodesic should be taken to mean that there is a choice of admissible pseudometrics which makes $\gamma$ a relative geodesic.

\end{dfn}

Some examples of admissible choices of pseudometrics are as follows (provided that the choices are made in a $G$-invariant manner):

\begin{itemize}

\item Make no change: For some/all $\tilde{V}$, define $d_{\tilde{V}}(x,y)=d(x,y)$ for some/all $x,y\in\tilde{V}^{(0)}$. Thus geodesics are relative geodesics.

\item ``Electrify'' some/all $\tilde{V}$ by defining $d_{\tilde{V}}(x,y)=0$ for all $x,y\in\tilde{V}$.

\item ``Cone off'' some/all $\tilde{V}$ by adding a new vertex and connecting all vertices of $\tilde{V}$ to it by an edge of length 1/2, and define $d_{\tilde{V}}$ by the metric this procedure induces, so that $d_{\tilde{V}}(x,y)=1$ for all distinct $x,y\in\tilde{V}$.

\item For some/all $\tilde{V}$, choose $d_{\tilde{V}}$ so that there is a constant $C$ such that \[\abs{d_{\tilde{V}}(x,y)-2\log(d(x,y)+1)}<C\] for all $x,y\in\tilde{V}$. This is the choice we will end up making later on.

\end{itemize}

\subsection{Local geometry of essential $2$-cells}

The following fact is a crucially important statement about the boundaries of essential $2$-cells in $\ucc{X}$. 

\noindent
\begin{lem}
\label{noorbits}
Suppose $X$ is a staggered generalized $2$-complex with locally indicable vertex groups and $n(X)\geq 2$. Let $\gamma$ a relative geodesic in $\ucc{X}$. Let $e$ be an essential edge of an essential $2$-cell $\alpha$. Then there exists an element of $[e]_\alpha$ not contained in $\gamma$.
\end{lem}

\begin{proof}
Suppose that the lemma is false. Among all triples $(\alpha,e,\gamma)$ with the property that all members of $[e]_\alpha$ lie in the relative geodesic $\gamma$, choose one for which the number of edges in $\gamma$ is minimal. Note that $\gamma$ will contain at least two edges.

Label the elements of $[e]_\alpha$, $e_1,\ldots,e_m$ (where $m\geq 2$ is the exponent of $\alpha$) in the order that they occur along $\gamma$, and orient them consistently with $\gamma$. Let $i(e_i)$ and $t(e_i)$ be the initial and terminal vertices, respectively, of $e_i$ for $i\in\{1,\ldots,m\}$. By minimality, the initial point of $\gamma$ is $i(e_1)$ and the terminal point is $t(e_m)$. Let $\sigma_i$ be the subpath of $\gamma$ between $t(e_i)$ and $i(e_{i+1})$, for $i\in\{1,\ldots,m-1\}$. Choose $\sigma\in\{\sigma_i\}$ such that $\rl(\sigma)$ is minimal. See figure \ref{fig4}. Decompose the image of $\partial\alpha$ in $X$ as a path $p^m$ where $p$ is not a proper power. The closed path $p$ corresponds to an order $m$ element $w$ of $\pi_1(X)$ which acts on $\ucc{X}$ by ``rotation'' through a point in the interior of $\alpha$. Consider the paths $\{w^i\sigma\}$ for $i\in\{1,\ldots,m\}$. Each path will connect two elements of $[e]_\alpha$ and the orbits will chain together to form an $m$-pointed star shape with corners on members of $[e]_\alpha$ (there are two cases according to whether the $\{w^i\sigma\}$ meet at their endpoints or have endpoints separated by the elements of $[e]_\alpha$).

\begin{minipage}{\linewidth}
      \centering
      \begin{minipage}{0.4\linewidth}
          \begin{figure}[H]
              \includegraphics[width=\linewidth]{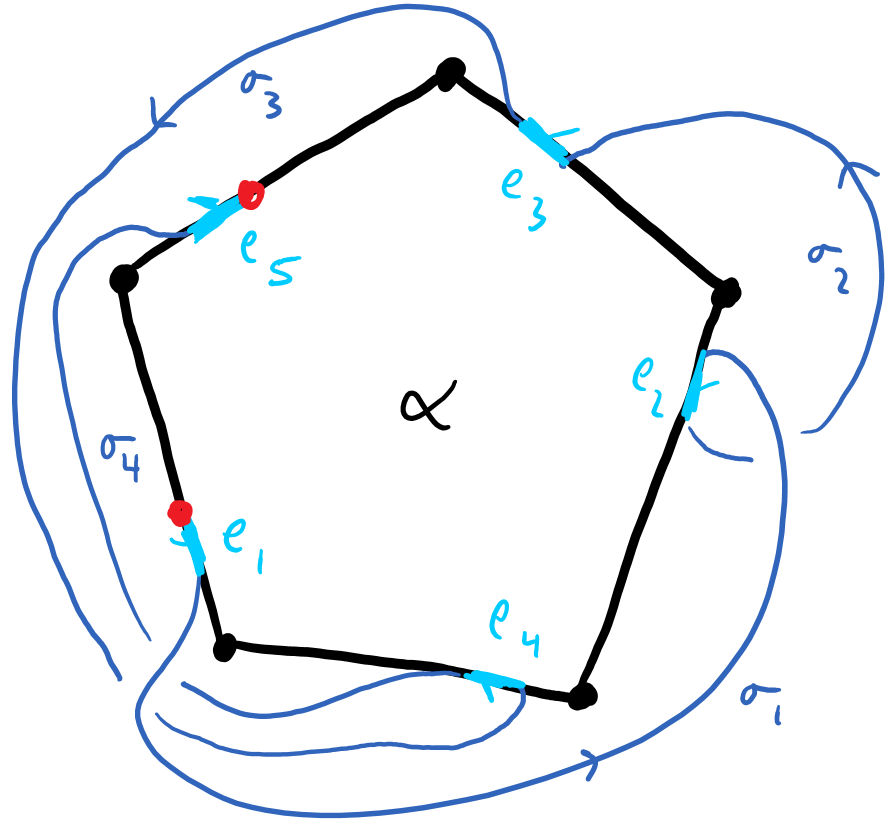}
							\caption{\footnotesize Decomposition of $\gamma$ into the $\sigma_i$. Suppose that $\sigma=\sigma_4$.}
					\label{fig4}
          \end{figure}
      \end{minipage}
      \hspace{0.1\linewidth}
      \begin{minipage}{0.32\linewidth}
          \begin{figure}[H]
              \includegraphics[width=\linewidth]{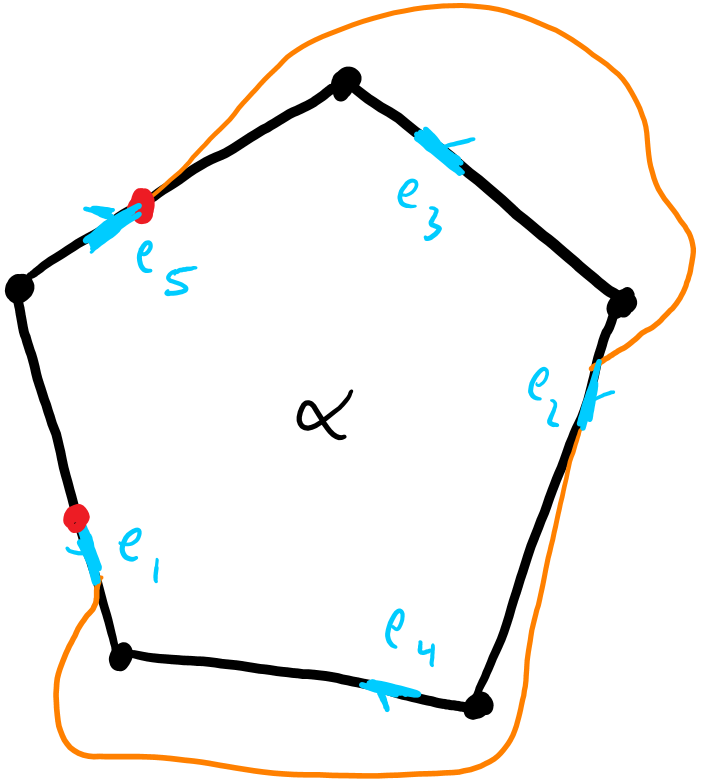}
					\caption{\footnotesize In this example, $\lambda$ is made up of two orbits of $\sigma$ and the edges $e_1$ and $e_2$.}
					\label{fig5}
          \end{figure}
      \end{minipage}
  \end{minipage}

Now, find a shortest relative path $\lambda$ in $\ucc{X}$ connecting $i(e_1)$ to $t(e_m)$ using only $w$-orbits of $\sigma$ and members of $[e]_\alpha$.  See figure \ref{fig5}. It is clear that $\rl(\lambda)\leq \frac{m}{2}\rl(\sigma)+\frac{m}{2}+1$. On the other hand, since $\gamma$ is a relative geodesic with the same endpoints as $\lambda$, we have that $\rl(\lambda)\geq m\rl(\sigma)+m$. Unless $m=2$, this contradicts the inequality

\[\frac{m}{2}L+\frac{m}{2}+1<m(L+1),\]

which holds when $L\geq 0$ and $m\geq 3$.

Thus we have reduced to the case $m=2$. We may also assume that $\sigma$ connects antipodal points of $\partial\alpha$, for otherwise $w\sigma$ connects $i(e_1)$ to $t(e_2)$ and $\rl(w\sigma)<\rl(\gamma)$ since $w\sigma$ avoids $e_1$ and $e_2$.

Observe by Lemma \ref{3.9} that $\partial\alpha$ embeds in $\ucc{X}$, so the two paths $\lambda_1$ and $\lambda_2$ of $\partial\alpha\setminus\{e_1,e_2\}$ do not intersect in $\ucc{X}$ (labeled so that $t(e_1)\in\lambda_1$). Since $\sigma$ starts in $\lambda_1$ and ends in $\lambda_2$, we can find an innermost subpath $\sigma'$ of $\sigma$ whose endpoints lie in $\lambda_1$ and $\lambda_2$, respectively, and whose interior does not intersect $\partial\alpha\setminus\{e_1,e_2\}$. Note that $\sigma'$ does not cross $e_1$ or $e_2$, as this would provide an obvious way to decrease the relative length of $\gamma$.

Consider the compact subcomplex $E=\alpha\cup\sigma'$ of $\ucc{X}$. By choice of $\sigma'$, $\pi_1(E)=\ZZ$. Let $q$ be a reduced path in $\ucc{X}$ which represents a generator of $\pi_1(E)$, and $D'\to\ucc{X}$ a reduced disk diagram with boundary $q$. Let $D=E\cup D'$. If $D$ is not reduced, then there is an essential $2$-cell $\beta$ of $D'$ such that $\alpha$ and $\beta$ form a cancelable pair and share an edge $f$ in their common boundary. If this happens, then ``fold'' $\beta$ over $\alpha$ by identifying the paths $\partial\beta\setminus\{ f\}$ and $\partial\alpha\setminus \{f\}$ and deleting $\beta$ from $D$. This is a homotopy equivalence and has the effect of modifying $q$ and deleting an essential $2$-cell from $D'$. This process terminates after finitely many steps, so we may assume that $D$ is reduced. We may also assume that $\partial D$ is contained in $\partial\alpha\cup\sigma'$, since any $2$-cell contributing an edge to $\partial D$ not in $\partial\alpha\cup\sigma'$ may simply be removed from $D$ without affecting that $D$ is simply connected. Note that at most one of $e_1$ and $e_2$ lies in $\partial D$. Otherwise, connect a point of $e_1$ to a point of $e_2$ by a snipping arc running across the interior of $\alpha$, and observe that the path $\sigma'$ contradicts Lemma \ref{snip}. Without loss of generality, assume that $e_1$ is internal in $D$. Thus $e_1$ lies in the boundary of at least two distinct essential $2$-cells of $D$.

Thus there exist at least two essential $2$-cells in $D$. Consider the natural reduced map $D\to X$. By Proposition \ref{4.11}, there is an extreme essential $2$-cell $\beta$ of $D$ distinct from $\alpha$ with exposed edge $f$, say. Since $\partial D$ is contained in $\partial\alpha\cup\sigma'$, all elements of $[f]_\beta$ are contained in this subcomplex of $\ucc{X}$ as well. In fact, all elements of $[f]_\beta$ are contained in $\sigma'$ since otherwise they could not lie on the boundary of $D$. Now $(\beta,f,\sigma')$ is a counterexample to the lemma. The fact that $\rl(\sigma')<\rl(\gamma)$ contradicts minimality of $(\alpha,e,\gamma)$, and the lemma is proved.
\end{proof}

\subsection{Patchings}

The following construction is of critical importance for later arguments. It shows that certain non-simply connected subcomplexes of $\ucc{X}$ can be made simply connected without introducing extra exposed or extreme $2$-cells, as follows.

\noindent
\begin{dfn} (\textbf{Patching}). Let $\phi:E\to \ucc{X}$ be reduced, where $E$ is compact but not necessarily simply connected. A \emph{patching} for $\phi$ is a simply connected $2$-complex $\pl{E}$ and a reduced diagram $\pl{\phi}:\pl{E}\to\ucc{X}$ such that $\pl{E}$ contains $E$ as a subcomplex, $\pl{\phi}|_E=\phi$, and none of the essential $2$-cells of $\pl{E}\setminus E$ are exposed in $\pl{E}$.
\end{dfn}

\noindent
\begin{rmk} In view of the unique composition $\ucc{X}\to\tilde{X}\to X$, where the first map is any inclusion of $\ucc{X}$ into $\tilde{X}$, reduced diagrams $D\to\ucc{X}$ give rise to reduced diagrams $D\to X$ and vice versa by Remark \ref{covreduce}. Whenever we have a patching $\pl{E}\to \ucc{X}$, we will casually confuse it with the corresponding diagram $\pl{E}\to X$ in order to apply Propositions \ref{4.11} and \ref{5.4}.
\end{rmk}

\noindent
\begin{lem}
\label{patching}
Let $\phi:E\to \ucc{X}$ be an inclusion of a compact connected $2$-complex. Suppose that there is a path $\lambda$ in $E$ with the property that $\lambda$ contains every isolated edge of $E$ and maps to a relative geodesic in $\ucc{X}$. Then a patching for $\phi$ exists.
\end{lem}

\begin{proof}

If $E$ is simply connected, then $\phi$ is a reduced diagram so set $\pl{\phi}=\phi$ and we are done. Otherwise let $g_1,\ldots, g_k$ be generators of $\pi_1(E)$. Let $E_0=E$ and $\phi_0=\phi$. For each $i$, Let $p_i$ be a reduced path in $E^{(1)}$ such that $[p_i]=g_i$. Let $\rho_i:D_i\to\ucc{X}$ be a reduced disk diagram such that $\rho_i(\partial D_i)=\phi(p_i)$. Inductively define $E_i=E_{i-1}\sqcup_{p_i}D_i$, and observe that there is a natural combinatorial map $\phi_i:E_i\to \ucc{X}$. If $\phi_i$ is not reduced, then there is a cancelable pair of $2$-cells in $E_i$, but the cancelable pair cannot both lie in $E$ or in $D_{j}$ for any $j<i$, since $\phi_i$ restricted to $E$ and to $D_j$ is reduced. We can make $\phi_i$ reduced as follows: First suppose that there is a cancelable pair of $2$-cells $\alpha_E$ and $\alpha_{D_i}$ in $E$ and $D_i$, respectively. Let $e_i$ denote the shared edge between $\alpha_E$ and $\alpha_{D_i}$, and let $\sigma_E$ and $\sigma_{D_i}$ be the paths in $\partial\alpha_E\setminus e_i$ and $\partial_{D_i}\setminus e_i$, respectively, from the terminal to the initial vertex of $e_i$, which are identified under $\phi_i$. Modify $E_i$ and $\phi_i$ by removing $\alpha_{D_i}$ from $D_i$ and identifying $\sigma_{D_i}$ with $\sigma_E$. Note that this process preserves $E$ as a subcomplex of $E_i$, and that, although we are modifying $\partial D_i$, $\rho_i(\partial(D_i\setminus\alpha_{D_i}))$ is homotopic to $p_i$ in $E_i$. It preserves homotopy type of $E_i$ because it is a homotopy equivalence. Repeating as many times as necessary, we may assume that there is no cancelable pair between $E$ and $D_j$ for any $j<i$. On the other hand, suppose that there is a cancelable pair of $2$-cells $\alpha_{D_j}$ and $\alpha_{D_i}$ in $D_j$ and $D_i$, respectively, for some $j<i$. Similarly to the first case, let $e_i$ denote the shared edge between $\alpha_{D_j}$ and $\alpha_{D_i}$, and let $\sigma_{D_j}$ and $\sigma_{D_i}$ be the paths in $\partial\alpha_{D_j}\setminus e_i$ and $\partial_{D_i}\setminus e_i$, respectively, from the terminal to the initial vertex of $e_i$, which are identified under $\phi_i$. Modify $E_i$ and $\phi_i$ by removing $\alpha_{D_i}$ from $D_i$ and identifying $\sigma_{D_i}$ with $\sigma_{D_j}$. Again, note that this process preserves $E$ as a subcomplex of $E_i$, and that, although we are modifying $\partial D_i$, $\rho_i(\partial(D_i\setminus\alpha_{D_i}))$ is homotopic to $p_i$ in $E_i$. It preserves homotopy type of $E_i$ because it is a homotopy equivalence. Repeating as many times as necessary, we may assume that there is no cancelable pair between $D_i$ and $D_j$ for any $j<i$, and thus that $\phi_i$ is reduced. Now $\pl{E}=E_k$ contains $E$, and since $\pl{E}$ is simply connected, $(\pl{\phi}=\phi_k):\pl{E}\to\ucc{X}$ is a reduced diagram. By construction, it is also clear that $\pl{\phi}|_E=\phi$.

It remains to prove that any essential $2$-cell $\alpha$ belonging to $\pl{E}\setminus E$ is not exposed in $\pl{E}$. To that end, let $\alpha$ be an essential $2$-cell belonging to $\pl{E}\setminus E$. Then $\alpha$ belongs to the complex $D_i$ for some $i\geq 1$. Consider the complex $E_{i-1}$ to which $D_i$ has been attached by its boundary, and assume that folds have been performed as described in the previous paragraph so that $E_i\to\ucc{X}$ is reduced. Observe that $\lambda$ contains every isolated edge of $E_{i-1}$ and maps to a relative geodesic in $\ucc{X}$, which is true by assumption for $i=1$. Indeed, it is obvious that $\lambda$ maps to a relative geodesic in $\ucc{X}$, and for $1\leq j<i$, every isolated edge of $D_j$ must belong to $\partial D_j$, so attaching $D_j$ to $E_{j-1}$ by its boundary cannot create new isolated edges in $E_j$. Now, if $\alpha$ is exposed in $\pl{E}$, then there is some exposed edge $e$ in $\partial\alpha$ such that $[e]_\alpha$ lies in $\partial E_{i}$. Since each edge of $[e]_\alpha$ also lies in $\partial D_i$, it must be the case that every edge of $[e]_\alpha$ is an isolated edge of $E_{i-1}$. Thus each edge of $[e]_\alpha$ belongs to $\lambda$, contradicting Lemma \ref{noorbits}.
\end{proof}

\subsection{More local geometry of essential $2$-cells}

With patchings as the fundamental tool, we now prove some other statements about the local geometry of essential $2$-cells.

\noindent
\begin{lem}
\label{shareboundary}
Let $\alpha$ and $\beta$ be distinct essential $2$-cells of $\ucc{X}$. Let $e$ be an essential edge of $\alpha$. Then at most one element of $[e]_\alpha$ lies in $\partial\beta$.
\end{lem}

\begin{proof}
Suppose that two elements $e_1$ and $e_2$ of $[e]_\alpha$ lie in $\partial\beta$. Then the complex $E=\cl{\alpha}\cup\cl{\beta}$ satisfies the hypotheses of Lemma \ref{patching}, so let $\pl{E}$ be a patching. By Proposition \ref{4.11}, $\alpha$ is extreme in $\pl{E}$ with exposed edge $f$. Note that $f\notin [e]_\alpha$ since $e_1$ and $e_2$ are internal in $\pl{K}$. Thus there are two elements of $[f]_\alpha$, $f_1$ and $f_2$, lying in distinct components of $\partial{\alpha}\setminus\{e_1,e_2\}$. Connect midpoints of $f_1$ and $f_2$ by a snipping arc running through the interior of $\alpha$, and observe that any path between $e_1$ and $e_2$ through the interior of $\beta$ contradicts Lemma \ref{snip}.
\end{proof}

The following strong statement rules out several more pathologies for a relative geodesic which intersects the boundary of an essential $2$-cell in $\ucc{X}$.

\noindent
\begin{lem}
\label{notweird}
Let $\alpha$ be an essential $2$-cell in $\ucc{X}$ with boundary path $p^n$, and let $\gamma$ be relative geodesic which uses at least $2$ essential edges of $\partial\alpha$. With respect to the orientation of $\gamma$, let $e$ and $e'$ be the first and last essential edges in $\partial\alpha\cap\gamma$ (labeled so that their orientations are consistent with $\gamma$). Index the essential edges of $\gamma$ from $e_1=e$ to $e_m=e'$. The following statements hold:

\begin{itemize}

\item[(i)] Each $e_i$ lies in $\partial\alpha$.
\item[(ii)] There is a path $\lambda_i$ in $\partial\alpha$ connecting $e_i$ to $e_{i+1}$ which does not use any essential edges.
\item[(iii)] The orientations of the $e_i$ are consistent with an orientation of $\partial\alpha$.

\end{itemize}

\end{lem}

\begin{proof}
(i): Assume that some $e_i$ does not lie in $\partial\alpha$. Let $f_1$ be the last essential edge of $\gamma$ before $e_i$ which lies in $\partial\alpha$, and let $f_2$ be the first essential edge of $\gamma$ after $e_i$ which lies in $\partial\alpha$. Let $\sigma$ be the subpath of $\gamma$ whose first edge is $f_1$ and last edge $f_2$. Consider the complex $E=\alpha\cup\sigma$. Then $E$ satisfies the hypothesis of Lemma \ref{patching}, so let $\pl{E}$ be a patching for $E$. The fact that $\pl{E}$ is simply connected implies $e_i$ is contained in an essential $2$-cell $\beta$ of $\pl{E}$ distinct from $\alpha$, since otherwise $e_i$ is isolated and non-separating. Thus $\pl{E}$ contains at least two essential $2$-cells. This contradicts Proposition \ref{4.11}, since $\alpha$ is the only essential $2$-cell of $\pl{E}$ which can be extreme.

(ii): Assume there is no path in $\partial\alpha$ connecting $e_i$ to $e_{i+1}$ which does not use any essential edges. Let $\lambda_1$ and $\lambda_2$ be the two subpaths of $\partial\alpha$ connecting $e_i$ to $e_{i+1}$. The subcomplex $E=\alpha\cup\gamma$ satisfies the hypotheses of Lemma \ref{patching}, so let $\pl{E}$ be a patching. Note that at least one of $\lambda_1$ or $\lambda_2$ has the property that all essential edges therein lie in the interior of $\pl{E}$, otherwise we may join two boundary essential edges of $\lambda_1$ and $\lambda_2$ by a snipping arc running across the interior of $\alpha$, and observe that the portion of $\gamma$ between $e_i$ and $e_{i+1}$ contradicts Lemma \ref{snip}. Without loss of generality, all essential edges of $\lambda_1$ are internal in $\pl{E}$. Also, at least one essential edge exists there by assumption. Thus there is an essential $2$-cell $\beta$ of $\pl{E}$ distinct from $\alpha$. This contradicts Proposition \ref{4.11}, since $\alpha$ is the only essential $2$-cell of $\pl{E}$ which can be extreme.

(iii): If this statement is false, then there is a pair of edges $e_i$ and $e_{i+1}$ which have opposite orientations in $\partial\alpha$. Let $\sigma$ be the subpath of $\gamma$ starting with $e_i$ and ending with $e_{i+1}$, and let $E=\alpha\cup\sigma$. This subcomplex satisfies the hypotheses of Lemma \ref{patching}, so let $\pl{E}$ be a patching. Now, observe that at least one of $e_i$ or $e_{i+1}$ is internal in $\pl{E}$. Indeed, if this is not the case then connect $e_i$ and $e_{i+1}$ together by a snipping arc running across the interior of $\alpha$. The portion of $\sigma$ between $e_i$ and $e_{i+1}$ now contradicts Lemma \ref{snip}. Thus at least one of $e_i$ or $e_{i+1}$ is internal. This shows that there is an essential $2$-cell in the diagram distinct from $\beta$, but this contradicts Proposition \ref{4.11}, since $\alpha$ is the only essential $2$-cell of $\pl{E}$ which can be extreme.
\end{proof}

The following is also useful:

\noindent
\begin{lem}
\label{geodboundary}
Let $\alpha$ be an essential $2$-cell in $\ucc{X}$, and let $\gamma$ be a relative geodesic. Then the number of essential edges in $\partial\alpha\cap\gamma$ is at most half the number of essential edges in $\partial\alpha$. 
\end{lem}

\begin{proof}
Let $e_1$ and $e_m$ be the first and last essential edges of $\alpha\cap\gamma$, if they exist, and labeled so that they are oriented consistently with $\gamma$. By Lemma \ref{notweird}, it makes sense to orient $\partial\alpha$ consistently with $\gamma$. We may assume that $e_1$ and $e_m$ are distinct, for otherwise $\partial\alpha\cap\gamma$ is a single edge and there is nothing to prove. Let $\sigma$ be the (possibly degenerate) arc of $\partial\alpha$ between $t(e_1)$ and $i(e_m)$ but not including either of these edges, and let $\sigma'$ be the other (possibly degenerate) arc of $\partial\alpha\setminus\{e_1,e_m\}$. Lemma \ref{notweird} also implies that $\gamma$ uses every essential edge of $\sigma$, every essential edge of $\gamma$ lies in $\sigma$, and the orientations and order in which these edges are visited are the same in both $\gamma$ and $\sigma$. Suppose the boundary path of the image of $\alpha$ in $X$ is of the form $p^n$, where $p$ is not a proper power. The path $p$ is a loop in $X$ which corresponds to an order $n$ element $w$ in $\pi_1(X)$ which acts by ``rotation'' of $\ucc{X}$ through a point in the interior of $\alpha$.

Let $\gamma_1$ be the portion of $\gamma$ running from $i(e_1)$ to $t(e_m)$. If $\gamma_1$ uses strictly more than half of the essential edges in $\partial\alpha$, then there is some integer $i$ such that $w^i\gamma_1$ properly contains all essential edges of $\sigma'$ as well as $e_m$ and $e_1$. Let $\gamma'$ be the subpath of $w^i\gamma_1$ running from $t(e_m)$ to $i(e_1)$; note $\rl(\gamma')<\rl(w^i\gamma_1)$ since $w^i\gamma_1$ uses $e_m$ and $e_1$ but $\gamma'$ does not. Since $\rl(w^i\gamma_1)=\rl(\gamma_1)$ by $G$-invariance of $\rl$, the path $\gamma'$ is an ``$\rl$-shortcut;'' this contradicts that $\gamma$ is a relative geodesic.
\end{proof}

\subsection{Convexity of vertex spaces}

The following fact will also be useful.

\noindent
\begin{lem}
\label{vertexspacesconvex}
The vertex spaces of $\ucc{X}$ are convex.
\end{lem}

Reminder: We are using the path metric on $\os{\ucc{X}}$.

\begin{proof}
Let $\gamma$ be a geodesic edge path between vertices $x$ and $y$ of a vertex space $\tilde{V}$. By passing to an innermost subpath outside of $\tilde{V}$, we may assume that $\gamma\cap \tilde{V}=\{x,y\}$. Let $\gamma'$ be a shortest path from $x$ to $y$ in $\tilde{V}$. Note that neither $\gamma$ nor $\gamma'$ backtrack. Also, the first edges of $\gamma$ and $\gamma'$ are not identified by the innermost subpath assumption; neither are the last edges. Thus the loop $\gamma(\gamma')^{-1}$ is reduced, so we may fill it with a reduced diagram $D$. If $D$ contains an essential $2$-cell, then by Lemma \ref{4.7}, there as an exposed essential $2$-cell $\alpha$ with exposed edge $e$. Since $\gamma'$ consists only of edges which are not essential, all elements of $[e]_\alpha$ lie on $\gamma$. This contradicts Lemma \ref{noorbits}. Thus $D$ contains no essential $2$-cells and so $\gamma$ also maps to $\tilde{V}$, which is also a contradiction.
\end{proof}

\section{Relative hyperbolicity}
\label{sect:rh}

Let $X$ be a staggered generalized $2$-complex with locally indicable vertex groups and $n(X)\geq 2$. From this point onward, assume that the total space $\gos{X}$ is a finite graph of spaces, i.e., the graph obtained by collapsing each vertex space of $\gos{X}$ to a point is finite. Note that this does not imply that $\gos{X}$ is compact as vertex spaces may not be. However, it does imply that $C(X)$ is finite. A result of crucial importance later on is that $\pi_1(X)$ is relatively hyperbolic with these assumptions. We prove this now.

We will use a definition of relative hyperbolicity in terms of relative Dehn functions, introduced in a more general form by Osin in \cite{o1}, which Hruska shows is well-defined and equivalent to no fewer than five others (\cite{hr1}) in the case that the set of peripheral subgroups is finite.

\noindent
\begin{dfn} (\textbf{Finite relative presentation/finite relative generating set}). Suppose $\PP$ is a finite collection of infinite subgroups of a countable group $G$ (called \emph{peripheral subgroups}) and let $\PPP$ be the union of all $P\in\PP$. We say that $(G,\PP)$ has a \emph{finite relative presentation} with \emph{finite relative generating set} $\mathcal{S}$ if $\mathcal{S}$ is finite and symmetrized ($\mathcal{S}=S\sqcup \overline S$), $\mathcal{S}\cup\PPP$ is a generating set for $G$, and the kernel of the natural map from $F(S)*(*_{P\in \PP}P)\to G$ is finitely normally generated, where $F(S)$ denotes the free group on the set $S$.
\end{dfn}

\noindent
\begin{dfn} (\textbf{Linear relative Dehn function}). Suppose $(G,\PP)$ has a finite relative presentation with finite relative generating set $\mathcal{S}=S\sqcup \overline S$. Let $\PPP$ be the union of all $P\in\PP$. Let $K=F(S)*(*_{P\in \PP}P)$ and $\mathcal{R}$ be a finite normal generating set for the kernel of the natural map $K\to G$.  For any word $W$ over $\mathcal{S}\cup\PPP$ representing the identity of $G$ (called a \emph{trivial word}), we have an equation in $K$ of the form $W=\Pi_{i=1}^lk_i^{-1}R_ik_i$ where $R_i\in\mathcal{R}$ and $k_i\in K$ for each $i$. The smallest such $l$ is called the area of $W$ and denoted by $A(W)$. We say $(G,\PP)$ has a \emph{linear relative Dehn function} for this relative presentation if there is a linear function $f:\NN\to\NN$ such that for each trivial word $W$ of length at most $m$ in $\mathcal{S}\cup\PPP$, $A(W)\leq f(m)$.
\end{dfn}

\noindent
\begin{dfn} (\textbf{Relatively hyperbolic}) \cite[Definition 3.7]{hr1}. Suppose $(G,\PP)$ has a finite relative presentation. If  $(G,\PP)$ has a linear relative Dehn function for some finite relative presentation of $(G,\PP)$, then we say $(G,\PP)$ is \emph{relatively hyperbolic} (or \emph{$G$ is hyperbolic relative to $\PP$}).
\end{dfn}

\noindent
\begin{lem}
\label{rh}
Suppose $X$ is a staggered generalized $2$-complex with locally indicable vertex groups, $n(X)\geq 2$, and the total space $\gos{X}$ is a finite graph of spaces. Let $\PP$ be the collection of vertex groups of $X$. Then $(\pi_1(X),\PP)$ is relatively hyperbolic.
\end{lem}

\begin{proof}

We first construct a finite relative generating set for $G=\pi_1(X)$. Choose a maximal spanning tree $T$ of essential edges in $\gos{X}$. Orient the essential edges of $\gos{X}\setminus T$. Now the finite relative generating set $\mathcal{S}=S\sqcup\overline{S}$ is in one-to-one correspondence with the set of these oriented edges and their formal inverses. Moreover, a normal generating set for the kernel of the natural map from $F(S)*(*_{P\in \PP}P)\to G$ can be identified with the set of boundary paths of each essential $2$-cell of $X$, after choice of base-point in $T$.

Let $p$ be a reduced, cyclically reduced path in $\gos{X}$ such that $[p]$ represents the trivial element of $G$. Let $\PPP$ be the union of all $P\in\PP$, and let $\frgsl{p}$ denote the word length of $p$ in $\mathcal{S}\cup\PPP$. Note that we can compute $\frgsl{p}$ by counting the number of essential edges of $p$ in $\gos{X}\setminus T$, plus the number of maximal subloops of $p$ which lie entirely in a single vertex space. Let $D\to X$ be a reduced diagram for $p$ which uses a minimal number of essential $2$-cells, and call the number of essential $2$-cells in such a diagram $\da{p}$. By Lemma \ref{vkl}, having a linear relative Dehn function with respect to the finite relative generating set above is equivalent to requiring that there exist constants $a,b$ such that $\da{p}\leq am+b$ for each such $p$ with $\frgsl{p}\leq m$.

To find such constants, we will also need to consider the ``Bass-Serre length'' of $p$, denoted by $\bsl{p}$, which is just the number of essential edges occurring in $p$. We claim that:
\begin{itemize}
\item[(1)] $\bsl{p}$ is bounded above by a linear function of $\frgsl{p}$, and
\item[(2)] $\da{p}$ is bounded above by a linear function of $\bsl{p}$.
\end{itemize}

To see the first claim, note that since $T$ is finite, there is a constant $d$ such that any reduced path which stays entirely inside it (using only essential edges) can use at most $d$ essential edges. In particular any reduced path $p'$ in $\gos{X}$ with $\bsl{p'}>d$ will either use an essential edge of $\gos{X}\setminus T$ or contain a subloop representing a nontrivial element of some vertex space. Thus if $p'$ is a subpath of $p$ with $\bsl{p'}=d+1$, $p'$ contributes at least one unit of length to $\frgsl{p}$. This shows that

\[\frac{\bsl{p}}{d+1}-1\leq \frgsl{p},\]

i.e.

\[\bsl{p}\leq (d+1)\frgsl{p}+(d+1).\]

For the second claim, use Dehn's algorithm: Let $D\to X$ be a reduced diagram for $p$ which uses a minimal number of essential $2$-cells. Suppose first that $D$ contains at least two essential $2$-cells. Then $D$ contains an extreme essential $2$-cell $\alpha$ by Proposition \ref{4.11}. Since $n(X)\geq 2$, $\alpha$ has exponent at least two, and thus strictly more than half of the essential edges of $\partial\alpha$ lie on $\partial D$. Let $D'$ be the unique component of $D\setminus\alpha$ which contains essential $2$-cells (it is unique since $\alpha$ is extreme). The path $p'=\partial D'$ has the property that $\bsl{p'}\leq \bsl{p}-1$. Also, $D'$ uses a minimal number of essential $2$-cells since $D$ does. By induction on $\bsl{p}$, we may assume that there exist positive constants $a'$ and $b'$ such that $\da{p'}\leq a'\bsl{p'}+b'$. Assume without loss that $a',b'\geq 1$. We have that

\[\da{p}=\da{p'}+1\leq a'\bsl{p'}+b'+1\leq a'\bsl{p}-a'+b'+1\leq a'\bsl{p}+b'\]

as well. On the other hand, if $D$ contains one or fewer essential $2$-cells, then $\da{p}\leq 1$. In particular, we again have that $\da{p}\leq a'\bsl{p}+b'$.

Stacking the inequalities from claims (1) and (2) gives us our linear relative Dehn function.
\end{proof}

\section{Walls and ladders}
\label{sect:wl}

From now on, assume that the staggered generalized $2$-complex $X$ with $n(X)\geq 2$ has the additional property that each of the vertex groups of $X$ admits a proper and cocompact action on a $\cz$ cube complex. We also continue to assume that $\gos{X}$ is a finite graph of spaces.

Since locally indicable groups are necessarily torsion-free, our assumption that the vertex groups are cubulable in fact allows us to assume that each vertex space $V$ is a compact non-positively curved (NPC) cube complex, and the universal cover $\tilde{V}$ is a $\cz$ cube complex. Note that this implies in particular that each vertex group is finitely presented since $V$ is a finite $K(G,1)$ for its vertex group. Since $C(X)$ is finite, this also implies that the complex $\ucc{X}$ is locally finite. For metric statements in what follows, we will always be using the $\ell_1$ metric in the $1$-skeleton of $\tilde{V}$.

Note that $\pi_1(X)$ acts geometrically (properly and cocompactly) on $\ucc{X}$ (though no longer freely, since there is a fixed point in each elevation of an essential $2$-cell). We will define our walls as codimension-$1$ immersed hyperspaces in $\ucc{X}$ and then prove that they satisfy the necessary properties to apply the Sageev construction.

Similarly to the description in \cite{manningnotes}, we define walls as components of a ``midcube complex,'' $M(\ucc{X})$. The cube complex $M(\ucc{X})$ and its natural map to $\ucc{X}$ are defined as follows.

We first describe the disjoint union of the cubes of $M(\ucc{X})$. Fix $\frac{1}{2}>\epsilon>0$. Each cell of $\ucc{X}$ is either a cube of some dimension or an essential $2$-cell. Each $k$-dimensional cube $C$ of $\ucc{X}$ contains $k$ midcubes of codimension $1$ obtained by setting exactly one coordinate equal to $\frac{1}{2}$. For us, each of these midcubes $C'$ will give rise to exactly two $(k-1)$-dimensional cubes of $M(\ucc{X})$ equipped with homeomorphisms to two parallel copies of $C'$ distance $\epsilon$ from $C'$ on opposite sides of $C'$. On the other hand, each essential $2$-cell $\alpha$ of $\ucc{X}$ contributes edges to $M(\ucc{X})$ as follows. Suppose that $\alpha$ is of exponent $n$. Each edge $e$ in $\partial\alpha$ is either an essential edge or a $1$-dimensional cube in some $\tilde{V}$. In either case, consider two points in the interior of $e$ which are distance $\epsilon$ from the midpoint of $e$. After choosing an orientation of $\partial\alpha$ we may label them $v_e^-$ and $v_e^+$. There are an analogous pair of points in each edge of $[e]_\alpha$, and we add $n$ edges ($1$-dimensional cubes) to $M(\ucc{X})$ where each edge maps to a path in $\cl{\alpha}$ running from the $v_e^+$ in each edge of $[e]_\alpha$ to the $v_e^-$ in the next edge of $[e]_\alpha$ through $\intr{\alpha}$, and such that the images of these $n$ edges are disjoint. Moreover, we require that the image of edges of $M(\ucc{X})$ mapping to essential $2$-cells is invariant with respect to the action of $\pi_1(X)$ on $\ucc{X}$.

Now identify faces of cubes of $M(\ucc{X})$ as follows: Whenever one of the face identifications of $\ucc{X}$ identifies the images of two faces of cubes of $M(\ucc{X})$, we identify those faces in $M(\ucc{X})$. The walls of $\ucc{X}$ are defined as the components of $M(\ucc{X})$. Figure \ref{fig:walls} shows an illustration of some portions of walls in $\ucc{X}$.

\begin{figure}[htbp]
	\centering
		\includegraphics[width=0.8\textwidth]{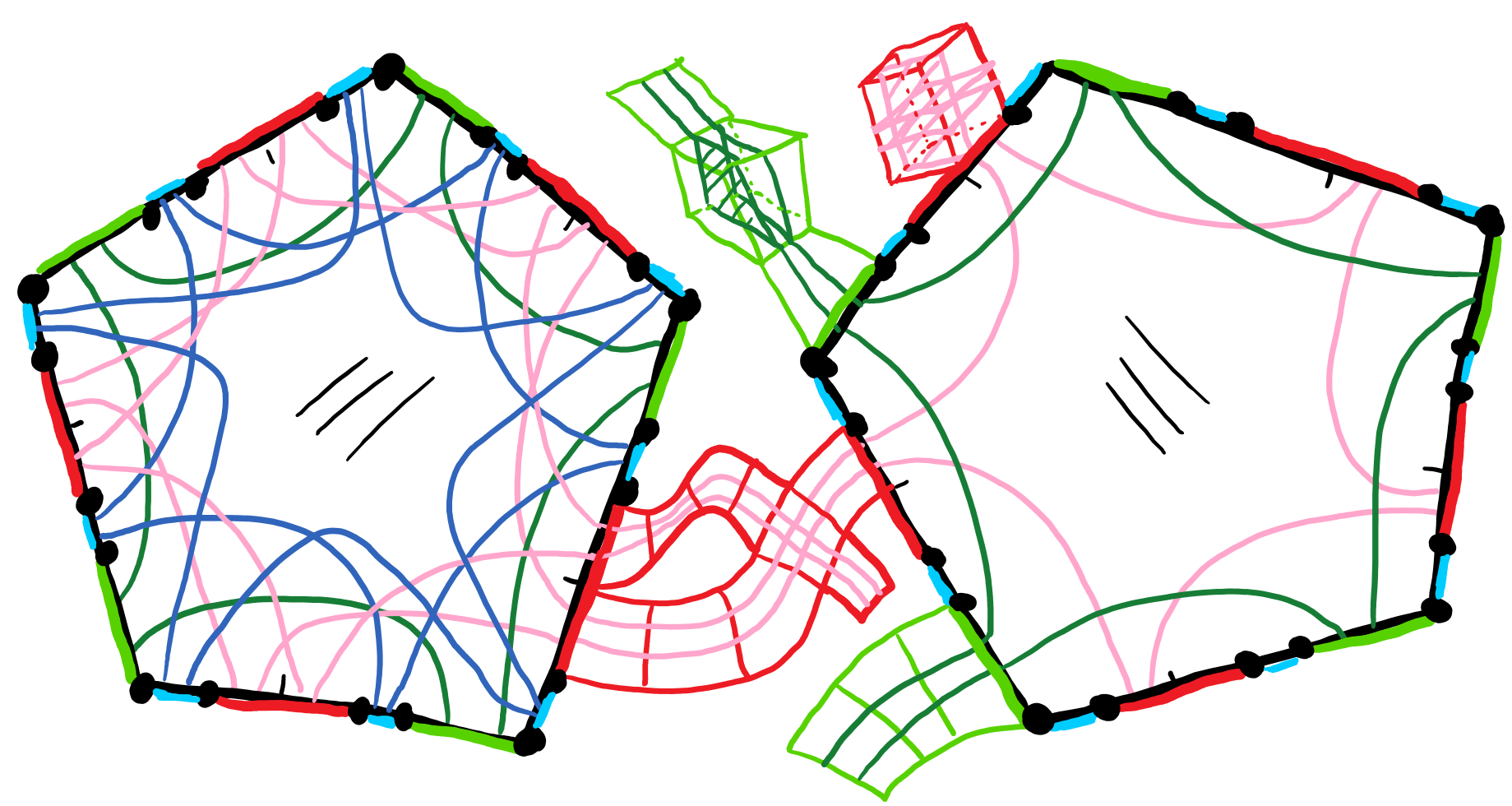}
	\caption{\footnotesize Some portions of walls in $\ucc{X}$. The dark blue segments joining essential (light blue) edges are disjoint from vertex spaces and are part of a walls which are immersed graphs in $\ucc{X}$. The pink and dark green segments joining the non-essential (red and light green) edges join to hyperplanes in vertex spaces of $\ucc{X}$ at their endpoints.}
	\label{fig:walls}
\end{figure}

Note that the action of $\pi_1(X)$ on $\ucc{X}$ preserves the system of walls just defined. Also note that there are two types of walls in $\ucc{X}$:
\begin{itemize}
\item[(i)] The walls which are dual to essential edges and do not intersect any $\tilde{V}$; these walls are graphs.
\item[(ii)] The walls which intersect some $\tilde{V}$. These walls may be higher dimensional. More precisely, these walls are \emph{graphs of hyperplanes}, i.e., they consist of hyperplanes of vertex spaces which are joined to each other by edges crossing essential $2$-cells, with the property that the endpoints of each edge are connected to vertices of hyperplanes.
\end{itemize}

A straightforward observation about walls is that they are locally determined:

\noindent
\begin{lem}
\label{locdet}
For any cell $\omega$ of $\ucc{X}$, if $\text{im}(\Lambda)\cap\omega$ is nonempty and $\text{im}(\Lambda)\cap\omega=\text{im}(\Lambda')\cap\omega$, then $\Lambda=\Lambda'$.
\end{lem}

It is not clear that the walls we have just defined are well-behaved in $\ucc{X}$. For example, a priori, a wall could travel in some vertex space $\tilde{V}$, leave the space through some essential $2$-cell $\alpha$, and later come back to that same vertex space so that its image in $\ucc{X}$ intersects itself. However, note that each wall is an NPC cube complex and so it makes sense to speak of a local geodesic in the $1$-skeleton of a wall.

\noindent
\begin{dfn} (\textbf{Carrier/wall segment/ladder}). For a wall $\Lambda\looparrowright\ucc{X}$, the \emph{carrier} of $\Lambda$ is the smallest subcomplex of $\ucc{X}$ containing the image of $\Lambda$. A \emph{wall segment} $\lambda$ in a wall $\Lambda$ is a local geodesic in $\Lambda^{(1)}$, embedded except possibly at its endpoints. The \emph{ladder} associated to $\lambda$ is the smallest subcomplex of $\ucc{X}$ containing the image of $\lambda$.
\end{dfn}

Note that ladders are necessarily $2$-dimensional.

\section{Walls embed and separate}
\label{sect:eandc}

In Lauer and Wise's setting, ladders turn out to be simply connected. This is not necessarily true in our case, but they can be patched:

\noindent
\begin{lem}
\label{wallpatch}
Let $H$ be the ladder associated to a wall segment. Then $H$ contains at most two extreme essential $2$-cells, and there is a patching $\pl{H}\to\ucc{X}$ for $H$.
\end{lem}
\begin{proof}
Consider the inclusion of $H$ into $\ucc{X}$, which is a reduced map. Note that the first and last essential $2$-cells of $H$ are the only candidates for extreme $2$-cells. Indeed, let $\lambda$ be the wall segment for which $H$ is the associated ladder, and observe that Lemma \ref{notextreme} may be applied to any essential $2$-cell $\alpha$ of $H$ which is not the first or last (taking the points $x$ and $y$ to be respective endpoints of the two edges of $\partial\alpha$ dual to $\lambda$ and on opposite sides of $\lambda$). Note also that $H$ has no isolated $1$-cells, unless $H$ is a single edge. Thus the hypotheses of Lemma \ref{patching} are satisfied and $\pl{H}\to\ucc{X}$ exists.
\end{proof}

The fact that walls embed and separate is a consequence of the following lemma.

\noindent
\begin{lem}
\label{wallcell}
Let $\alpha$ be a $2$-cell of $\ucc{X}$ (essential or not). If $\lambda$ is a wall segment with both endpoints in $\alpha$, then $\lambda$ is contained in $\alpha$.
\end{lem}

\begin{proof}
Let $H$ be the ladder associated to $\lambda$ and let $K=\alpha\cup H$. Note that $\partial\alpha$ embeds in $\ucc{X}$. If $\alpha$ is essential this follows from Lemma \ref{3.9}, and if $\alpha$ is a square then this is a general fact about $\cz$ cube complexes. We will show that $K$ contains no $2$-cells besides $\alpha$, which proves the lemma. If $K$ contains a $2$-cell besides $\alpha$ then we may choose distinct points $u$ and $v$ in $\partial\alpha\cap\lambda$ such that the portion of $\lambda$ (of positive length) between $u$ and $v$ (which we denote by $\lambda'$) does not internally intersect $\alpha$. Let $H'$ be the ladder associated to $\lambda'$, and note that $K'=\alpha\cup H'$ is itself a ladder (by possibly extending $\lambda'$ across $\alpha$ if necessary). By Lemma \ref{wallpatch}, $K'$ has a patching $\pl{K'}\to\ucc{X}$.

Note first that $\alpha$ cannot be a square. Indeed, if it is, then the wall segment $\lambda'$ passes through an essential $2$-cell, for otherwise we have found a wall segment in a single $\cz$ cube complex which leaves and comes back to the same square, and this contradicts the known behavior of hyperplanes in these spaces. Let $u'$ and $v'$ be the first points along $\lambda'$ from $u$ and $v$, respectively, which lie in the boundary of some essential $2$-cells $\alpha_u$ and $\alpha_v$, which may or may not be distinct. Note that $\alpha_u$ and $\alpha_v$ are the only candidates for extreme essential $2$-cells of $\pl{K'}$. On the other hand, $u'$ and $v'$ become identified in the auxiliary diagram, so  in fact neither $\alpha_u$ nor $\alpha_v$ can be extreme by Lemma \ref{notextreme}. The complex $\pl{K'}$ contradicts Proposition \ref{4.11}.

Thus $\alpha$ is an essential $2$-cell. By extending $\lambda'$ through $\alpha$ if necessary, we see that $\alpha$ is both the first and last essential $2$-cell through which $\lambda'$ passes. Since $\alpha$ is the only candidate for an extreme $2$-cell of $\pl{K'}$ by Lemma \ref{wallpatch}, Proposition \ref{4.11} implies that $\alpha$ is the only essential $2$-cell of $\pl{K'}$, and $\alpha$ is exposed by Lemma \ref{4.7}. Thus $H'$ is made entirely of squares. Let $e_u$ and $e_v$ be the edges of $\partial\alpha$ containing $u$ and $v$. Let $\sigma$ and $\sigma'$ be the two arcs of $\partial\alpha\setminus\{u,v\}$. Suppose one of these arcs, say $\sigma$, contains no essential edges. The arc $e_u\cup\sigma\cup e_v$ is a geodesic in a $\cz$ cube complex, and the wall segment $\lambda'$ shows that some wall segment (lying entirely in that $\cz$ cube complex) crosses it twice. This also contradicts behavior of hyperplanes in these spaces. Thus there are essential edges $e$ and $e'$ in $\sigma$ and $\sigma'$ respectively. On the other hand, $e$ and $e'$ lie on $\partial\pl{K'}$ by the fact that $\alpha$ is the only essential $2$-cell of $\pl{K'}$ and Lemma \ref{3.9}. Connect midpoints of $e$ and $e'$ by a snipping arc running through the interior of $\alpha$ and observe that the wall segment $\lambda'$ contradicts Lemma \ref{snip}.

It follows that $K$ contains no $2$-cells besides $\alpha$, and the lemma is proved.
\end{proof}

\noindent
\begin{prop}
\label{7.4}
(cf \cite[Theorem 7.4]{lw}). Each wall is a tree of hyperplanes and embeds in $\ucc{X}$.
\end{prop}

\begin{proof}

If some wall $\Lambda$ is not simply connected, then there exists a wall segment $\lambda$ of positive length in $\os{\Lambda}$ which is a loop. Let $H$ be the ladder associated to $\lambda$. Note that $H$ contains at least two $2$-cells since the boundaries of $2$-cells of $\ucc{X}$ embed. Pick a $2$-cell $\alpha$ in $H$. The previous Lemma implies that every wall segment connecting any pair of points in $\lambda\cap\partial\alpha$ passes through the interior of $\alpha$. This contradicts that $H$ contains at least two $2$-cells.

Thus $\Lambda$ is simply connected. Since it is an NPC cube complex, it is in fact a $\cz$ cube complex. We thus see that $\Lambda$ is a tree if it is a wall of type (i), and a tree of hyperplanes if it is a wall of type (ii).

Now suppose that a wall $\Lambda$ does not embed in $\ucc{X}$. Then $\Lambda$ intersects itself in some essential $2$-cell $\alpha$ or some cube $c$. In the latter case, there is some $2$-dimensional face of $c$ in which we will witness the intersection of $\Lambda$ with itself. Thus we may choose a wall segment $\lambda$ which intersects itself exactly once in a $2$-cell $\alpha$ (essential or not) and let $H$ be the ladder associated to $\lambda$. Note that $H$ contains at least two $2$-cells since the boundaries of $2$-cells of $\ucc{X}$ embed. The previous Lemma implies that every wall segment connecting any pair of points in $\lambda\cap\partial\alpha$ passes through the interior of $\alpha$. This contradicts that $H$ contains at least two $2$-cells.
\end{proof}

This result permits us to casually confuse a wall $\Lambda$ with its image in $\ucc{X}$, a liberty we will take freely in what follows.

\noindent
\begin{cor}
\label{wallsep}
Each wall in $\ucc{X}$ is separating.
\end{cor}

\begin{proof}
For any point $p$ in a wall $\Lambda$, $\Lambda$ separates a neighborhood of $p$ into exactly two components, by Lemma \ref{7.4} and construction. Thus each wall is locally separating and has an $I$-bundle neighborhood. And since each wall is a tree of hyperplanes (also Lemma \ref{7.4}), each wall is contractible. Thus each $I$-bundle neighborhood is actually a product. Thus for each wall, $\ucc{X}$ decomposes as a graph of spaces with a single simply connected edge space. Since $H^{1}(\ucc{X})=0$, this graph of spaces is a dumbell space (not a loop), and each wall is separating.
\end{proof}

Here are some miscellaneous convenient lemmas about the geometry of walls.

\noindent
\begin{lem}
\label{random2}
Let $\gamma$ be a relative geodesic edge path in a vertex space $\tilde{V}$ of $\ucc{X}$. Let $\Lambda$ be a wall. Then $\Lambda\cap\gamma$ is either empty or a single point.
\end{lem}

\begin{proof}
Since $\gamma$ lies in a vertex space, it is in fact a geodesic by definition. Suppose $\Lambda$ intersects $\gamma$ in two distinct points $x$ and $y$. Let $\lambda$ be a wall segment connecting $x$ to $y$ and let $H$ be the associated ladder. The complex $K=H\cup\gamma$ is a subcomplex of $\ucc{X}$ which has a natural reduced map to $X$, and it satisfies the hypotheses of Lemma \ref{patching}, so let $\pl{K}$ be a patching for $K$. Note $\pl{K}$ has a maximum of two extreme $2$-cells by Lemma \ref{wallpatch} applied to $H$. If $\pl{K}$ has an essential $2$-cell, then $H$ contains essential $2$-cells and the first one $\alpha$ through which $\gamma$ passes is extreme in $\pl{K}$ by Proposition \ref{4.11}. Let $e$ be an exposed essential edge lying in the boundary of $\alpha$, and choose two elements $e_1$ and $e_2$ of $[e]_\alpha$ which lie on opposite sides of $\lambda\cap\alpha$. Connect $e_1$ and $e_2$ by a snipping arc across the interior of $\alpha$, and observe that this snipping arc is non-separating, contradicting the snipping lemma. Indeed we can get from one side to the other by following $H$ to $\gamma$, traversing $\gamma$ from $x$ to $y$ (or $y$ to $x$), and then going through the other portion of $H$ until reaching the snipping arc. This works because there are no essential edges in $\gamma$. Thus there are no essential $2$-cells in $\pl{K}$. But this means that a connected component of $\Lambda\cap\tilde{V}$ (which is a hyperplane in $\tilde{V}$ by Proposition \ref{7.4}) crosses the geodesic  $\gamma$ twice, which contradicts the behavior of hyperplanes in $\cz$ cube complexes.
\end{proof}

We record the following immediate corollary.

\noindent
\begin{cor}
\label{wallintersectvs}
For each wall $\Lambda$ and each vertex space $\tilde{V}$, $\Lambda\cap\tilde{V}$ is either empty or consists of a single hyperplane in $\tilde{V}$.
\end{cor}

\noindent
\begin{lem}
\label{random3}
Let $\gamma$ be a relative geodesic in $\ucc{X}$ and suppose $\Lambda\cap\gamma$ consists of at least two distinct points $x$ and $y$. If $\lambda$ is a wall segment in $\Lambda$ connecting $x$ to $y$, then $\lambda$ passes through at least one essential $2$-cell.
\end{lem}

\begin{proof}

Let $H$ be the ladder associated to $\lambda$, and let $K=H\cup\gamma$. Then $K$ satisfies the hypotheses of Lemma \ref{patching}, so let $\pl{K}\to\ucc{X}$ be a patching. If $\lambda$ does not pass through an essential $2$-cell, then $H$ is made entirely of squares, and thus so is $\pl{K}$ by Lemma \ref{4.7}. This implies that there are no essential edges in $\gamma$, because any such edge would be isolated and nonseparating in $\pl{K}$. Thus $\pl{K}$ maps to a single vertex space $\tilde{V}$ of $\ucc{X}$. Since $\gamma$ is a relative geodesic mapping to a single vertex space, it is a geodesic in that vertex space. The fact that $\Lambda\cap V$ crosses $\gamma$ twice is a contradiction.
\end{proof}

\section{Walls are relatively quasiconvex}
\label{sect:relqc}

In Lauer and Wise's setting, walls turn out to be quasi-convex. This is used in conjunction with the fact that one-relator groups with torsion are Gromov hyperbolic to apply a theorem of Sageev and conclude that the action of these groups on their associated dual cube complexes are cocompact.

We will use a relative version of this argument. As we argued in Lemma \ref{rh}, $G=\pi_1(X)$ is hyperbolic relative to the vertex groups. In this secton, this will be an ingredient in a proof that each wall stabilizer is quasiconvex relative to the vertex groups. This result will be used in Section \ref{sect:action} when we apply a generalization of Sageev's theorem by Hruska-Wise to conclude that the action on the dual cube complex is cocompact.

\subsection{Geometric relative quasiconvexity}

We will first prove the following geometric relative quasiconvexity statement about wall carriers and then translate it to the algebraic relative quasiconvexity of wall stabilizers. In this lemma, we only use the metric on $\os{\ucc{X}}$. The $2$-cells are irrelevant for the argument.

\noindent
\begin{lem}
\label{georqc}
Suppose that $n(X)\geq 4$. Let $\Lambda$ be a wall in $\ucc{X}$. There is a uniform constant $W$ such that if $\gamma$ is a relative geodesic in $\os{\ucc{X}}$ between vertices in the carrier $C$ of $\Lambda$, then every vertex of $\gamma$ which lies in an essential edge is within distance $W$ of $C$.
\end{lem}

\begin{proof}

First note that since $\gos{X}$ is a finite graph of spaces, the set $C(X)$ is finite, and there is an upper bound $W_X$ on the number of edges (essential or not) in the attaching map of the elements of $C(X)$.

Let $\gamma$ be a relative geodesic in $\os{\ucc{X}}$ whose endpoints $x$ and $y$ are vertices in $C$. If $\gamma$ is contained in $C$, then we are done. By passing to an innermost subpath of $\gamma$ which lies outside of $C$, we may assume that $\gamma\cap C=\{x,y\}$. Since $x$ and $y$ lie in $C$, there is a ladder $H$ in $C$ containing $x$ and $y$ with associated wall segment $\lambda$, and $\gamma$ does not internally intersect $H$. The subcomplex $K=\gamma\cup H$  satisfies Lemma \ref{patching}, so let $\pl{K}\to\ucc{X}$ be a patching. When choosing generators of $\pi_1(K)$ to perform the patching, choose them so that there is exactly one generator which uses the path $\gamma$. Call the disk associated to this generator $D$ and make the choice that this is $D_1$, the first disk, in the patching construction. With this choice we may assume there is a planar subcomplex $D$ of $\pl{K}$, homeomorphic to a disk, such that $\gamma$ is one arc of $\partial{D}$ and the other arc $\sigma$ lies in $H$. Note also that $\sigma$ has no edges on $\partial\pl{K}$.

Note $\pl{K}$ has a maximum of two extreme $2$-cells since $H$ does (by Lemma \ref{wallpatch}). Thus Proposition \ref{5.4} implies that every essential $2$-cell of $\pl{K}$ is external (since the exponent of each essential $2$-cell is at least two). In particular, this holds for every essential $2$-cell of $D$, and in fact every essential $2$-cell of $D$ has an essential edge lying along $\gamma$.

Let $A$ be the union of essential $2$-cells of $D$ whose closures intersect $H$ (i.e., their boundaries intersect $\sigma$). Let $z$ be a point in an essential edge $e$ of $\gamma$. These are the points we will show are uniformly close to $H$. If $z\in\cl{A}$, then $d(z,H)\leq\frac{W_X}{2}$. If $z\notin\cl{A}$, let $\delta$ be the maximal connected subpath of $\gamma$ containing $z$ such that $\intr{\delta}\cap\cl{A}$ is empty. Since every $2$-cell of $A$ has an edge on $\gamma$, the complex $\cl{D\setminus\cl{A}}$ is a tree of disks. Let $D'$ be the maximal subcomplex of $\cl{D\setminus\cl{A}}$ which contains $z$ and is homeomorphic to a disk. Let $\delta'$ be the path $\partial D'\setminus\intr{\delta}$ (the other boundary arc of $D'$), and label the endpoints of $\delta'$, $x'$ and $y'$ in such a way that $x'$ lies on the subpath of $\gamma$ between $y'$ and $x$.

We claim that at most two essential $2$-cells in $A$ are adjacent to $\delta'$ along essential edges. Indeed, if there are three or more let $\alpha$ be one which is not the first, $\alpha_1$, or the last, $\alpha_2$ (with respect to a chosen orientation of $\delta'$). Since $\alpha$ is external in $\pl{K}$, there is an essential edge $f$ of $\alpha$ on $\partial{\pl{K}}$, and because $\alpha$ lies in $D$, $f$ lies on $\gamma$. Without loss of generality, suppose that $f$ lies in the portion of $\gamma$ between $z$ and $x$. Because $D$ is planar, whichever of $\alpha_1$ or $\alpha_2$ intersects the subpath of $\delta'$ between $\cl{\alpha}\cap\delta'$ and $x'$ cannot also intersect $\sigma$, contradicting that it lies in $A$. This proves the claim.

The above claim shows that $\delta'$ decomposes as a path $\delta_1\delta_2\delta_3$, where $\delta_1$ and $\delta_3$ are (possibly degenerate) paths, each of which lies along the boundary of an essential $2$-cell of $A$ , and $\delta_2$ is a (possibly degenerate) subpath of $\sigma$ which does not use any essential edges and maps to a single vertex space. See figure \ref{fig:fig7} for the general picture.

\begin{figure}[htbp]
	\centering
		\includegraphics[width=1\textwidth]{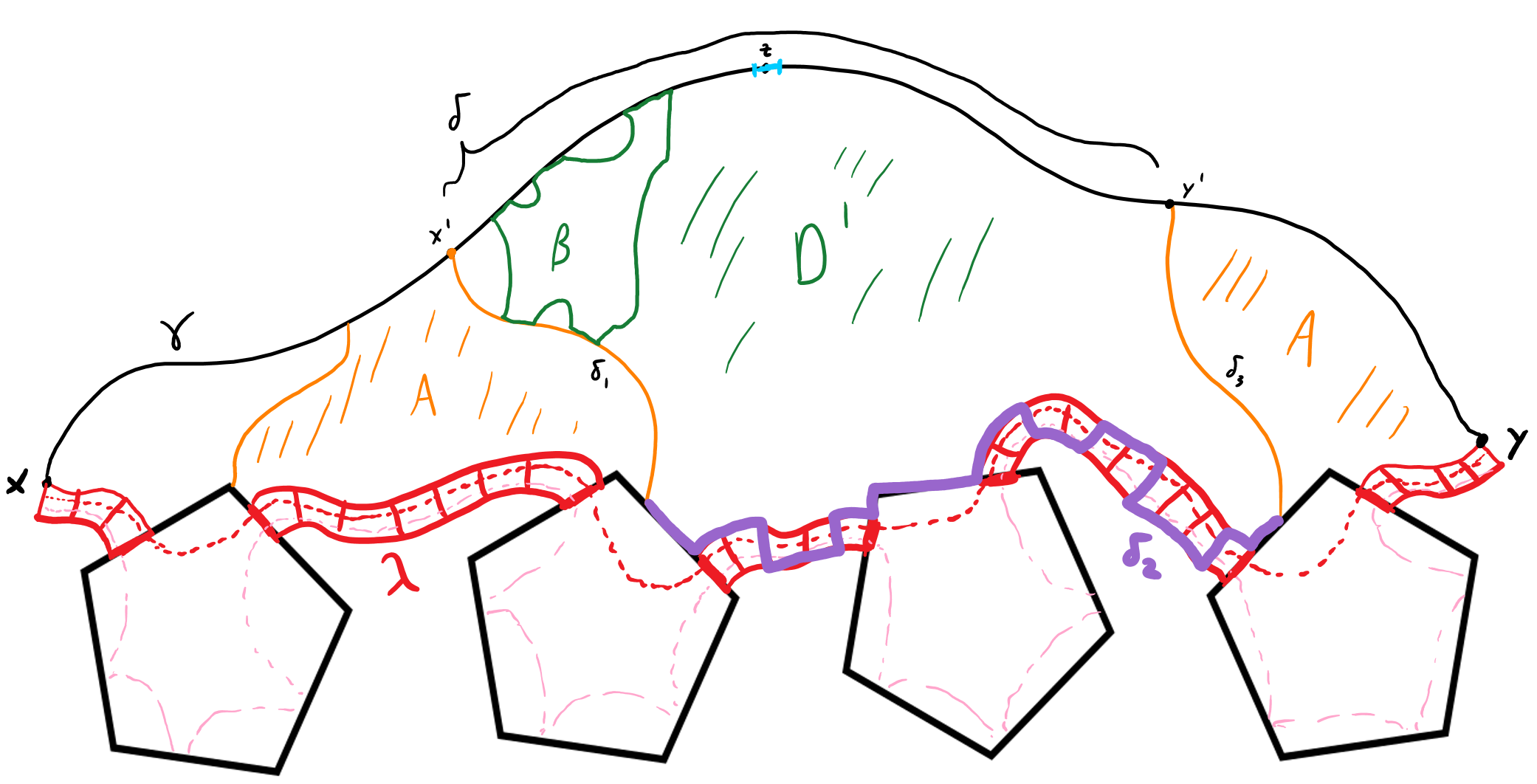}
	\caption{\footnotesize An illustration of the general case. Because $\delta_1$ and $\delta_3$ are so short, $\delta$ is a relative geodesic, $\delta_2$ contains no essential edges, and $n(X)\geq 4$, any candidate $\beta$ for an extreme essential $2$-cell of $D'$ must have exposed edges on all of $\delta_1$, $\delta$, and $\delta_3$. This shows that $D'$ contains a single essential $2$-cell which contains $z$ and intersects $\delta_1\cup\delta_3$, so that $z$ is close to $A$.}
	\label{fig:fig7}
\end{figure}

Next, we claim that $D'$ contains at most one essential $2$-cell. To see this claim, suppose that $D'$ contains two or more essential $2$-cells. Then $D'$ contains at least two extreme $2$-cells $\alpha$ and $\beta$ by Proposition \ref{4.11}, with, say, exposed edges $e$ and $f$, respectively. Note that all elements of $[e]_\alpha$ and $[f]_\beta$ lie along $\delta_1\cup\delta\cup\delta_3$ since $\delta_2$ contains no essential edges. In fact, it must be the case that at least two elements $e_1$ and $e_2$ of $[e]_\alpha$ lie along $\delta_1\cup\delta_3$. Indeed, otherwise $m-1$ elements of $[e]_\alpha$ along $\delta$, where $m$ is the exponent of $\alpha$. Lemma \ref{notweird} implies that $\delta$ visits every essential edge of some subpath of $\partial\alpha$ containing these $m-1$ elements of $[e]_\alpha$. Since $m\geq n(X)\geq 4$, this subpath contains strictly more than half of the essential edges of $\partial\alpha$. This contradicts Lemma \ref{geodboundary} since $\delta$ is a relative geodesic. Similarly, at least two elements $f_1$ and $f_2$ of $[f]_\beta$ lie along $\delta_1\cup\delta_3$. Now consider the following statements:

\begin{itemize}
\item $\{e_1,e_2\}$ lies in $\delta_1$.
\item $\{e_1,e_2\}$ lies in $\delta_3$.
\item $\{f_1,f_2\}$ lies in $\delta_1$.
\item $\{f_1,f_2\}$ lies in $\delta_3$.
\end{itemize}

If \emph{none} of these statements hold then both $\alpha$ and $\beta$ have boundary intersecting both $\delta_1$ and $\delta_2$, so either $\alpha$ or $\beta$ is internal in $\pl{K}$ by planarity of $D'$, which contradicts Proposition \ref{5.4}. On the other hand, if \emph{any} of these statements hold, we immediately obtain a contradiction to Lemma \ref{shareboundary}, since $\delta_1$ and $\delta_3$ both lie in the boundary of a single essential $2$-cell. This contradiction proves the claim.

Since $z\notin A$, $D'$ contains a single essential $2$-cell $\alpha$, and $z\in\partial\alpha$. By Lemma \ref{4.7}, $\alpha$ is exposed in $D'$ with exposed edge $e$, say. By Lemma \ref{noorbits}, some element of $[e]_\alpha$ lies in $\delta_1\cup\delta_3$. This shows that $d(z,A)\leq \frac{W_X}{2}$ and $d(z,H)\leq W_X$, so setting $W=W_X$ proves the lemma.
\end{proof}

Problem: Does Lemma \ref{georqc} hold when $n(X)\in\{2,3\}$? One seems to run into trouble when trying to rule out the case where $D'$ contains a ``fat'' region of squares in its interior. Lauer and Wise do not experience this difficulty in their setting.

To apply the Hruska-Wise cocompactness criterion, we also need to know that wall stabilizers act cocompactly on their associated walls:

\noindent
\begin{lem}
\label{wallsrqc}
Let $\Lambda$ be a wall of $\ucc{X}$. Then $H=\stab{\Lambda}$ acts cocompactly on the carrier of $\Lambda$, and thus on $\Lambda$.
\end{lem}

\begin{proof}

Let $C$ be the carrier of $\Lambda$ in $\ucc{X}$. We claim that there are finitely many $H$-orbits of cells of $C$, which implies the result. Let $\phi:\ucc{X}\to X$ be the natural map. Let $\beta$ be any cell of $X$ which intersects $\phi(C)$. Now $\phi(\Lambda)\cap \beta$ consists of finitely many codimension-1 (in $\beta$) ``subwalls'' of $\beta$. Enumerate these subwalls $\lambda_1,\cdots,\lambda_k$. By Lemma \ref{wallcell}, any cell $\alpha$ of $C$ which maps to $\beta$ has a well-defined type $i\in\{1,\cdots,k\}$, defined to be the unique index for which $\phi^{-1}(\lambda_i)\cap\alpha$ lies in $\Lambda$. Let $\alpha$ and $\alpha'$ be cells of the same type. Since the action of $G=\pi_1(X)$ is essentially the universal covering space action (except on essential $2$-cells where the following is still true), there is an element $g\in G$ which takes $\alpha$ to $\alpha'$. Moreover, because these cells are the same type, $\phi^{-1}(\lambda_i)\cap\alpha'$ lies in both $g\Lambda$ and $\Lambda$ (in case $\alpha$ and $\alpha'$ are essential $2$-cells, we may need to compose with a finite-order ``rotation'' in $\stab{\alpha'}$). Now, since walls are locally determined (Lemma \ref{locdet}), this shows that $g$ in fact stabilizes $\Lambda$, i.e. $g\in H$. Thus the number of $H$-orbits of $\phi^{-1}(\beta)$ is bounded above by $k$. Since $\beta$ was arbitrary, this proves the claim and the lemma.
\end{proof}

\subsection{Algebraic relative quasiconvexity}

To show wall stabilizers are relatively quasiconvex, we will use the following definition of relative quasiconvexity, which we quote from \cite{hr1}. In that paper, Hruska shows that this notion of relative quasiconvexity is well-defined and equivalent to no fewer than four others, at least in the case that the peripheral groups are finitely generated and there are finitely many peripheral groups. See \cite{hr1} for the definitions of cusp-uniform action and truncated space.

\noindent
\begin{dfn}
\label{relqc3} (\textbf{Relatively quasiconvex}) \cite[Definition 6.6]{hr1} (``QC-3'') Suppose $G$ is countable, $\PP=\{P_1,\ldots,P_m\}$ is a finite collection of subgroups, and that $(G,\PP)$ is relatively hyperbolic. A subgroup $H\leq G$ is \emph{relatively quasiconvex} (with respect to $\PP$) if the following holds. Let $(Y,\rho)$ be a proper $\delta$-hyperbolic metric space on which $(G,\PP)$ has a cusp-uniform action. Let $Y-U$ be a truncated space for $G$ acting on $Y$. For some base-point $x\in Y-U$, there is a constant $\mu\geq 0$ such that whenever $\gamma$ is a geodesic in $Y$ with endpoints in the orbit $Hx$, we have \[\gamma\cap(Y-U)\subset N_\mu(Hx),\] where the $\mu$-neighborhood $N_\mu(Hx)$ of $Hx$ is taken with respect to the metric $\rho$ on $Y$.

\end{dfn}

\noindent
\begin{prop}
\label{stabsrqc}
The stabilizer of each wall in $\ucc{X}$ is quasiconvex relative to the collection of vertex groups of $X$ when $n(X)\geq 4$.
\end{prop}

\begin{proof}

We will proceed by ``augmenting'' the space $\ucc{X}$, which is decidedly not $\delta$-hyperbolic, in general, by attaching ``combinatorial horoballs'' to form a space $A(\ucc{X})$ which is $\delta$-hyperbolic and on which $G$ acts in a cusp uniform manner, as follows.

As in Section \ref{sect:rh}, let $\PP=P_1,\ldots, P_m$ be the vertex groups of $X$ and choose a maximal spanning tree $T$ of essential edges of $\gos{X}$. Let $\mathcal{S}=S\sqcup \overline{S}$ be the set of oriented essential edges of $X$ not in $T$ and their formal inverses. Then $\mathcal{S}$ is a finite relative generating set for $(G,\PP)$. The Cayley graph $\Gamma$ of $G$ with respect to $\mathcal{S}$ is disconnected, in general.

Now, attach Groves-Manning ``combinatorial horoballs'' to $\Gamma$ to form the ``augmented space'' $\aug{\Gamma}$ associated to the data $(G,\PP,\mathcal{S})$. See \cite[Definitions 4.1 and 4.3]{hr1} for the precise construction. To each $P_i$ is associated a $\cz$ cube complex which induces a natural left-invariant metric $d_i$ on it. The rough idea is that for each coset $gP_i$, we take countably many copies of $gP_i$ indexed by the naturals, attach ``vertical edges'' between each element of $gP_i$ in every level and the corresponding element above and below it, and ``horizontal edges'' between elements of $gP_i$ in the same level of $d_i$-distance less than or equal to $2^j$, where $j$ is the level. The original coset $gP_i$ sits at level $0$. Let $\horo_\Gamma(g,i)$ be the combinatorial horoball above the coset $gP_i$, which by convention includes the original $gP_i$ at level $0$, as well as any edges added there. By \cite[ Theorem 4.4]{hr1} (originally proved by Groves and Manning) and relative hyperbolicity of $(G,\PP)$, the augmented space $\aug{\Gamma}$ is connected and $\delta$-hyperbolic.

On the other hand, let $\coll{X}$ be the space obtained by collapsing $T$ to a point. This collapse lifts to a $G$-equivariant quotient map $f:\ucc{X}\to\coll{\ucc{X}}$, where the target is obtained by collapsing each copy of $T$ in $\ucc{X}$; this map is a quasi-isometry.

Now, $G$ acts naturally on $\coll{\ucc{X}}$, and each vertex space of $\coll{\ucc{X}}$ is stabilized by some $gP_ig^{-1}$. We label this vertex space $\tilde{V}_g^i$. We now form the augmented space $\aug{\coll{\ucc{X}}}$ by building a combinatorial horoball $\horo_X(g,i)$ above the one-skeleton of $\tilde{V}_g^i$, again with respect to the cube complex metric, for each $(g,i)$ (as before, $\horo_\Gamma(g,i)$ includes the one-skeleton of $\tilde{V}_g^i$ by convention). We can identify the group elements of $gPg^{-1}$ with vertices of $\tilde{V}_g^i$ via the orbit map (choosing the image of $T$ in $\ucc{X}_c$ as a base-point). Thus, $\horo_\Gamma(g,i)$ is a full subgraph of $\horo_X(g,i)$ for each $(g,i)$.

Observe now that the Cayley graph $\Gamma$ includes naturally inside of $\coll{\ucc{X}}$. By the observation of the previous paragraph, there is also a natural inclusion $\aug{\Gamma}\hookrightarrow \aug{\coll{\ucc{X}}}$, which we now claim is a quasi-isometry. Assuming this claim, we have that $\aug{\coll{\ucc{X}}}$ is $\delta$-hyperbolic (after possibly modifying $\delta$).

To see the claim, first choose $K>\max_i(\text{diam}_{d_i}(P_i))$. It is clear that $\aug{\Gamma}$ is $K$-cobounded in $\aug{\coll{\ucc{X}}}$. It remains to show that $\aug{\Gamma}$ is quasi-isometrically embedded. For points $x$ and $y$ of $\zs{\aug{\Gamma}}$, it is also clear that $d_{\aug{\coll{\ucc{X}}}}(x,y)\leq d_{\aug{\Gamma}}(x,y)$. In the other direction, we seek a constant $K'$ such that $d_{\aug{\Gamma}}(x,y)\leq K'd_{\aug{\coll{\ucc{X}}}}(x,y)+K'$. Let $\gamma$ be a geodesic in $\os{\aug{\coll{\ucc{X}}}}$ between $x$ and $y$. Then $\gamma$ decomposes as a path of the form $\gamma_0e_1\gamma_1e_2\ldots e_k\gamma_k$ where each $e_j$ is an essential edge and each $\gamma_j$ is a (possibly empty) edge path in some $\horo_X(g,i)$. By \cite[Lemma 3.10]{gm}, we may assume that each $\gamma_j$ consists of at most two vertical segments and a single horizontal segment of length at most $3$. Moreover, since the endpoints of $\gamma_j$ lie in the image of the orbit map, these vertical segments also lie in $\horo_\Gamma(g,i)$. Now, the horizontal segment $h_j$ may not belong to $\horo_\Gamma(g,i)$, but because its endpoints are connected by a path of length at most $3$, there is a path $h_j'$ of length $5$ in $\horo_\Gamma(g,i)$ between its endpoints, where $h_j'$ consists of two vertical segments of length $2$ and a single horizontal edge two levels above $h_j$. Replacing each $h_j$ by $h_j'$, we obtain a path $\gamma'$ between $x$ and $y$ in $\aug{\Gamma}$, and since $\ell(h_j')\leq \ell(h_j)+4$, we have that $\ell(\gamma')\leq \ell(\gamma)+4(k+1)$. But also $d_{\aug{\Gamma}}(x,y)\leq \ell(\gamma')$ and $k\leq \ell(\gamma)=d_{\aug{\coll{\ucc{X}}}}(x,y)$, so $d_{\aug{\Gamma}}(x,y)\leq 5d_{\aug{\coll{\ucc{X}}}}(x,y)+4$. Setting $K'=5$ proves the claim.

Finally, build the augmented space $\aug{\ucc{X}}$. For each vertex space $\tilde{V}_g^i$ of $\ucc{X}$ which is stabilized by $g P_i g^{-1}$, build a combinatorial horoball above it using the cube complex metric as in the case of $X_c$. In fact, since the map $f$ is the identity on $\tilde{V}_g^i$, the horoball just added will be an isometric copy of $\horo_X(g,i)$. The map $f$ thus extends to a quasi-isometry $\tilde{f}:\aug{\ucc{X}}\to \aug{\coll{\ucc{X}}}$ which is the identity on combinatorial horoballs, so that $\aug{\ucc{X}}$ is $\delta$-hyperbolic (after possibly modifying $\delta$).

Now, we claim that $G$ has a cusp-uniform action on $\aug{\ucc{X}}$ with truncated space the disconnected union of all essential edges of $\ucc{X}$. In other words, the vertex spaces of $\ucc{X}$, along with their combinatorial horoballs, form a collection of disjoint $G$-equivariant horoballs (in the cusp-uniform sense) centered at the parabolic points of $G$. It is clear that $G$ acts coboundedly on this truncated space.

To see this, one can construct explicit horofunctions on these horoballs. For each vertex space $\tilde{V}$ of $\ucc{X}$, let $\horo_{\tilde{V}}$ be the combinatorial horoball above it. Let $d_A$ be the natural metric on $\aug{\ucc{X}}$. Define a function $\tilde{v}:\aug{\ucc{X}}\to \RR$ by

\begin{displaymath}
   \tilde{v}(x) = \left\{
     \begin{array}{lr}
       d_A(x,\tilde{V}) & : x \in \horo_{\tilde{V}} \\
       -d_A(x,\tilde{V}) & : \text{otherwise}
     \end{array}
   \right.
\end{displaymath} 

It is easy to check using elementary hyperbolic geometry that $\tilde{v}$ is a horofunction centered at the parabolic point $\xi$ in the Gromov boundary of $\aug{\ucc{X}}$ which can be identified with any geodesic ray starting in $\tilde{V}$ and using only vertical edges. This proves the claim.

For each vertex space $\tilde{V}$ of $\ucc{X}$, define $d_{\tilde{V}}(x,y)=d_A(x,y)$ for all $x,y\in\tilde{V}^{(0)}$.  The property of $G$-invariance is clear, so this is an admissible choice of pseudometrics.

To complete the proof, pick a basepoint $x$ in the carrier $C$ of $\Lambda$ and let $H=\stab{\Lambda}$, so that $Hx$ lies in $C$. Let $x',y'$ in $Hx$, and let $\gamma'$ be a relative geodesic in $\os{\ucc{X}}$ between $x'$ and $y'$ (with respect to the admissible choice of pseudometrics above). Let $\gamma$ be a geodesic in $A(\os{\ucc{X}})$ which agrees with $\gamma'$ on essential edges (it is clear by the construction of the pseudometrics that such a geodesic exists). Note that the intersection of $\gamma$ with the truncated space is precisely the set of essential edges of $\gamma$. Applying Lemma \ref{georqc} to $\gamma'$, we see that every essential edge of $\gamma'$ lies uniformly close to $C$. Thus the same is true for $\gamma$, and the proposition is proved.
\end{proof}

\section{Walls satisfy linear separation}
\label{sect:ls}

In order to conclude that the action of $G=\pi_1(X)$ on its associated dual cube complex is proper, we will argue that the walls in $\ucc{X}$ satisfy the ``linear separation property,'' which roughly means that the number of walls separating pairs of points in $\ucc{X}$ grows at least linearly with their distance. Hruska and Wise describe how the linear separation property leads to properness of the dual cube complex action in \cite[Theorem 5.2]{hw2}.

The precise statement we will prove is as follows:

\noindent
\begin{prop}
\label{9.1}
Suppose that $n(X)\geq 4$. There are constants $\kappa$ and $\epsilon$ such that for any vertices $x,y \in\ucc{X}$, the number of walls separating $x$ and $y$ is at least $\kappa d(x,y)-\epsilon$.
\end{prop}

We will be assuming for contradiction that walls frequently ``double-cross'' geodesics. We will use the following definition.

\noindent
\begin{dfn} (\textbf{Double-crosses/double-crossed ladder}). Let $\gamma$ be a geodesic in $\os{\ucc{X}}$ between two $0$-cells $x$ and $y$ of $\ucc{X}$. For every edge $e$ of $\gamma$, there are two dual walls to $e$ which intersect $e$ in the points $v_e^x$ and $v_e^y$, labeled so that $d(x,v_e^x)<d(x,v_e^y)$. Call the wall which passes through $v_e^x$, $\Lambda_e^x$, and the wall passing through $v_e^y$, $\Lambda_e^y$. We say that $\Lambda_e^x$ \emph{double-crosses} $\gamma$ if there is a wall segment $\lambda_e^x$ in $\Lambda_e^x$ between $v_e^x$ and another distinct point $u_e^x$ along $\gamma$. If this behavior occurs we will pass to an initial such wall segment emanating from $v_e^x$ and assume that $\Lambda_e^x$ does not cross $\gamma$ between $v_e^x$ and $u_e^x$. There is a unique ladder $H_e^x$ associated to $\lambda_e^x$. Let $\gamma_e^x$ be the subsegment of $\gamma$ connecting the edges containing $v_e^x$ and $u_e^x$. Let $Y=Y_e^x=\gamma_e^x\cup H_e^x$. We call the subcomplex $Y_e^x$ a \emph{double-crossed ladder} of $\gamma$ at $(e,x)$, if it exists. See figure \ref{fig:fig8} for an illustration.

\end{dfn}


\begin{figure}[htbp]
	\centering
		\includegraphics[width=1\textwidth]{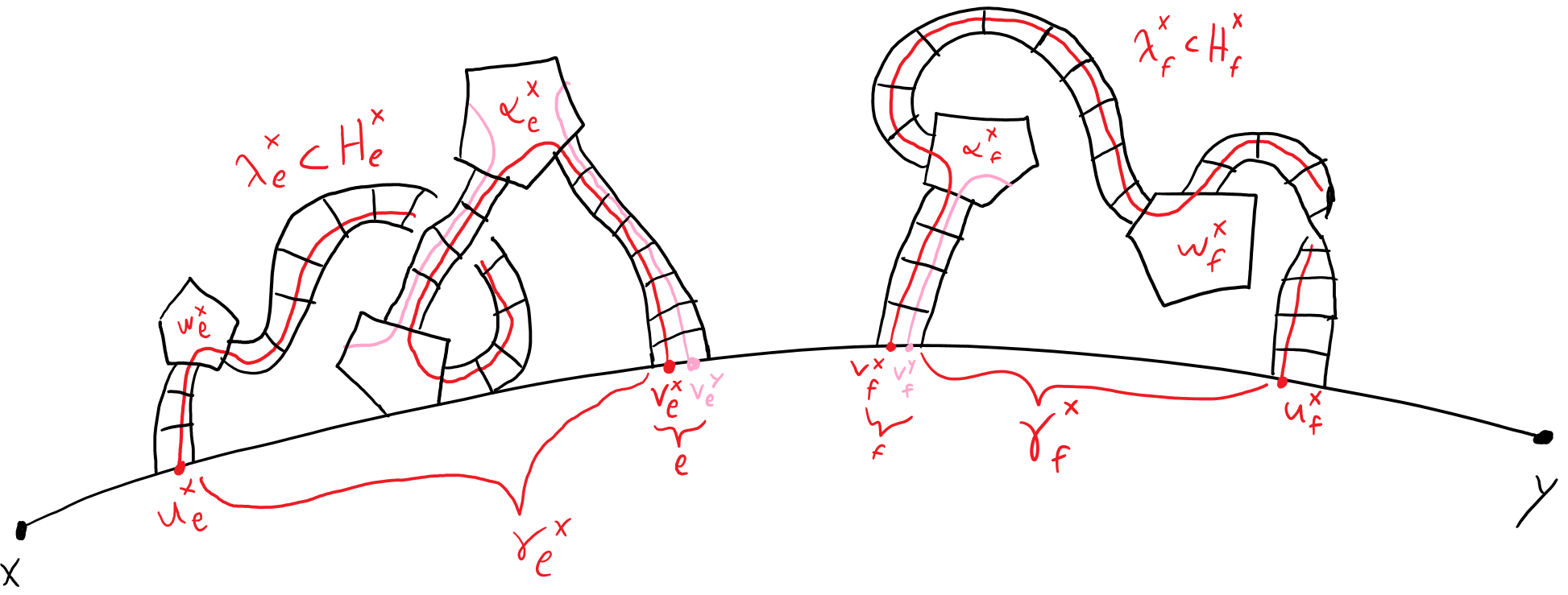}
	\caption{\footnotesize Some double-crossed ladders. The ladder $H_e^x$ bends in the direction of $x$, and $H_f^x$ bends in the direction of $y$. Here the rank of $\pi_1(Y_e^x)$ is $2$. Some pathologies for double-crossed ladders may be ruled out immediately. For example, the depicted twist in $H_f^x$ is ruled out by Corollary \ref{wallsep}.}
	\label{fig:fig8}
\end{figure}

\noindent
\begin{dfn} (\textbf{Returns}). Let $Y_e^x$ be a double-crossed ladder of $\gamma$ at $(e,x)$, with associated ladder $H_e^x$. We say that $Y_e^x$ (or $H_e^x$) \emph{returns} through an essential $2$-cell if that $2$-cell is the first or last essential $2$-cell of $H_e^x$ through which the wall segment $\lambda_e^x$ passes, as we traverse $\lambda_e^x$ starting from $v_e^x$. We use the notation $\alpha_e^x$ for the first $2$-cell through which $Y_e^x$ returns, and $\omega_e^x$ for the last.
\end{dfn}

Lemma \ref{random3} implies that whenever $Y_e^x$ is a double-crossed ladder, $\alpha_e^x$ and $\omega_e^x$ always exist, and they are clearly unique. It is possible that $\alpha_e^x=\omega_e^x$.

\noindent
\begin{dfn} (\textbf{Bends in the direction of}). Let $Y_e^x$ be a double-crossed ladder of $\gamma$ at $(e,x)$ with associated ladder $H_e^x$. We say that $Y_e^x$ (or $H_e^x$) \emph{bends in the direction of} $x$ if $d(u_e^x,x)<d(v_e^x,x)$. Otherwise we say that $Y_e^x$ (or $H_e^x$) bends in the direction of $y$. We make analogous definitions for $Y_e^y$ (or $H_e^y$) with $x$ and $y$ interchanged.
\end{dfn}

The following lemma allows us to determine the direction in which walls bend, but only when $n(X)\geq 4$. The lemma is false for $n(X)\in\{2,3\}$.

\noindent
\begin{lem}
\label{direction}
Suppose that $n(X)\geq 4$. For some edge $e$ of $\gamma$, suppose that a wall $\Lambda_e^x$ double-crosses $\gamma$. Then there is a double-crossed ladder $Y_e^x$ of $\gamma$ at $(e,x)$ with associated ladder $H_e^x$ which bends in the direction of $x$.
\end{lem}

\begin{proof}
Suppose that every double-crossed ladder $Y_e^x$ bends in the direction of $y$. Let $Y=Y_e^x$ be a double-crossed ladder with the property that $\Lambda_e^x$ does not cross $\gamma$ between $v=v_e^x$ and $u=u_e^x$. By Corollary \ref{wallsep}, $\ucc{X}\setminus\Lambda_e^x$ decomposes into two components $\ucc{X}_{\text{in}}$ and $\ucc{X}_{\text{out}}$, labeled so that $\gamma'=\gamma_e^x$ maps to $\ucc{X}_{\text{in}}$.

Let $e_1$ and $e_2$ be the edges of $\partial\alpha=\partial\alpha_e^x$ which are dual to $\lambda=\lambda_e^x$ (they may be essential or not), labeled so that there is a path from $e_1$ to $e$ inside $\lambda$. Suppose $\partial\alpha=p^m$ in $X$, where $p$ is not a proper power. Orient $e_1$ so that it crosses $\lambda$ in the same direction that $e$ crosses it, and extend this orientation to $\partial\alpha$. Let $\sigma_{\text{in}}$ and $\sigma_{\text{out}}$ be the two subpaths of $\partial\alpha\setminus\{e_1,e_2\}$, oriented consistently with $\partial\alpha$, and labeled so that $\sigma_{\text{in}}$ maps to $\ucc{X}_{\text{in}}$ and $\sigma_{\text{out}}$ maps to $\ucc{X}_{\text{out}}$ (we may do this since $\alpha\cap\Lambda_e^x$ consists only of the arc $\alpha\cap\lambda$ by Lemma \ref{wallcell}). Thus no point of $\sigma_{\text{out}}$ lies along $\gamma'$.

Note that $Y$ satisfies the hypotheses of Lemma \ref{patching} and let $\pl{Y}$ be a patching for $Y$. Note that $\alpha$ and $\omega_e^x$ are the only essential $2$-cells of $\pl{Y}$ which can be extreme, and in fact they are extreme by Lemma \ref{4.7} (if they are distinct). We claim that $\sigma_{\text{out}}$ is not internal in $\pl{Y}$. To see this, let $f$ be an exposed essential edge of $\alpha$. Since $\sigma_{\text{out}}$ has length $\abs{p}-1$, either some element of $[f]_\alpha$ lies along $\sigma_{\text{out}}$, in which case we are done, or $e_1,e_2\in[f]_\alpha$. In the latter case, $\alpha=\omega_e^x$ and both $e_1$ and $e_2$ lie along $\gamma'$. Lemma \ref{notweird} implies that every element of $[f]_{\alpha}$ lies along $\gamma'$, which contradicts Lemma \ref{noorbits}. This proves the claim.

Since $e_1$ and $e_2$ do not lie in $[f]_\alpha$, we may choose $f$ to be the element of $[f]_\alpha$ which lies in $\sigma_{\text{out}}$. The other $m-1$ elements of $[f]_\alpha$ lie in $\sigma_{\text{in}}$. Note that every such element must lie along $\gamma'$. Indeed, if this is not the case then given an element $f'\in[f]_\alpha$ which lies in $\sigma_{\text{in}}$ but not along $\gamma'$, we may join $f$ and $f'$ by a snipping arc running through the interior of $\alpha$. The graph $\gamma'\cup\lambda$ now contradicts Lemma \ref{snip}.

Thus the geodesic $\gamma'$ visits $m-1$ elements of $[f]_\alpha$. Lemma \ref{notweird} implies that $\gamma'$ visits each essential edge of $\sigma_{\text{in}}$ in turn. Let $f'$ and $f''$ be the first and last elements of $[f]_\alpha$ along $\sigma_{\text{in}}$. Since $m\geq 4$, the minimal subpath of $\gamma'$ containing these two edges contains strictly more than half of the essential edges of $\partial\alpha$. This contradicts Lemma \ref{geodboundary}.
\end{proof}

The following definition describes an impossible configuration of a pair of double-crossed ladders in $\ucc{X}$. We will show that if linear separation fails we can find such a configuration.

\noindent
\begin{dfn} (\textbf{Double-crossed pair of ladders}). Let $\gamma$ be a geodesic in $\os{\ucc{X}}$ with endpoints $0$-cells $x$ and $y$. Let $e_a$ and $e_b$ be adjacent edges along $\gamma$. Suppose that $Y_a$ and $Y_b$ are double-crossed ladders at $(e_a,z_a)$ and $(e_b,z_b)$, respectively, where $z_a,z_b\in\{x,y\}$. Suppose further that $Y_a$ and $Y_b$ bend in the same direction and that $\alpha_a=\alpha_{e_a}^{z_a}$ and $\alpha_b=\alpha_{e_b}^{z_b}$ are distinct. In this case we call the subcomplex $Y=Y_a\cup Y_b$ of $\ucc{X}$ a \emph{double-crossed pair of ladders}. We denote by $\omega_a$ the last essential $2$-cell through which $Y_a$ returns, $\lambda_a$ the wall segment associated to $Y_a$, $H_a$ its associated ladder, etc. Similarly define $\omega_b$, $\lambda_b$, and $H_b$, etc.
\end{dfn}

\noindent
\begin{lem}
\label{notrouble}
There does not exist a double-crossed pair of ladders in $\ucc{X}$.
\end{lem}

Remark: This lemma is true when $n(X)\in\{2,3\}$. This is what makes the following proof so technical.

\begin{proof}
Let $Y=Y_a\cup Y_b$ be a double-crossed pair of ladders. Suppose without loss of generality that $Y_a$ and $Y_b$ bend in the direction of $x$. Note that $Y$ satisfies the hypotheses of Lemma \ref{patching}, and let $\pl{Y}$ be a patching. The only candidates for extreme $2$-cells of $\pl{Y}$ are $\alpha_a$, $\omega_a$, $\alpha_b$, and $\omega_b$. We know that $\pl{Y}$ contains at least two essential $2$-cells since $\alpha_a$ and $\alpha_b$ are distinct. Observe that $H_a$ and $H_b$ embed in $\pl{Y}$, but they may overlap with each other.

We will prove the following statements:

\begin{itemize}
\item[(i)] If $\alpha_a\neq\omega_a$, then $\alpha_a$ is not extreme.
\item[(ii)] If $\alpha_b\neq\omega_b$, then $\alpha_b$ is not extreme.
\item[(iii)] If $\omega_a\neq\omega_b$, then at most one of $\omega_a$ and $\omega_b$ can be extreme.
\end{itemize}

Taken together, these statements imply that $\pl{Y}$ contains at most one extreme essential $2$-cell. This contradicts Proposition \ref{4.11}.

To see statement (i), temporarily orient $e_a$ and $e_b$ so that their terminal points coincide. Let $f_a$ and $g_a$ be the edges of $\partial\alpha_a$ which are dual to $\lambda_a$ (they may be essential or not), labeled so that there is a path from $f_a$ to $e_a$ inside $\lambda_a$ which does not internally intersect $\alpha_a$. Suppose $\partial\alpha=p^m$ in $X$, where $p$ is not a proper power. Orient $f_a$ so that it crosses $\lambda_a$ in the same direction that $e_a$ crosses it, and extend this orientation to $\partial\alpha_a$. Now the terminal points $t(f_a)$ and $t(g_a)$ of $f_a$ and $g_a$ are the length of $p$ apart in $\partial\alpha_a$. Moreover, in the auxiliary diagram $\aux{Y}$, $\aux{t(f_a)}$ lies in $\aux{\alpha_b}$ and $\aux{t(g_a)}$ lies in $\aux{\beta}$ for some essential $2$-cell of $Y_a$ distinct from $\alpha_a$, since $\alpha_a\neq\omega_a$. Lemma \ref{notextreme} proves the claim. Note that this argument does not depend on the direction in which $\lambda_a$ bends. Switching the symbols $a$ and $b$, an identical argument shows that $\alpha_b$ is not extreme if $\alpha_b\neq\omega_b$, and statement (ii) is proved. See figure \ref{fig:fig9}.

\begin{figure}[htbp]
	\centering
		\includegraphics[width=0.7\textwidth]{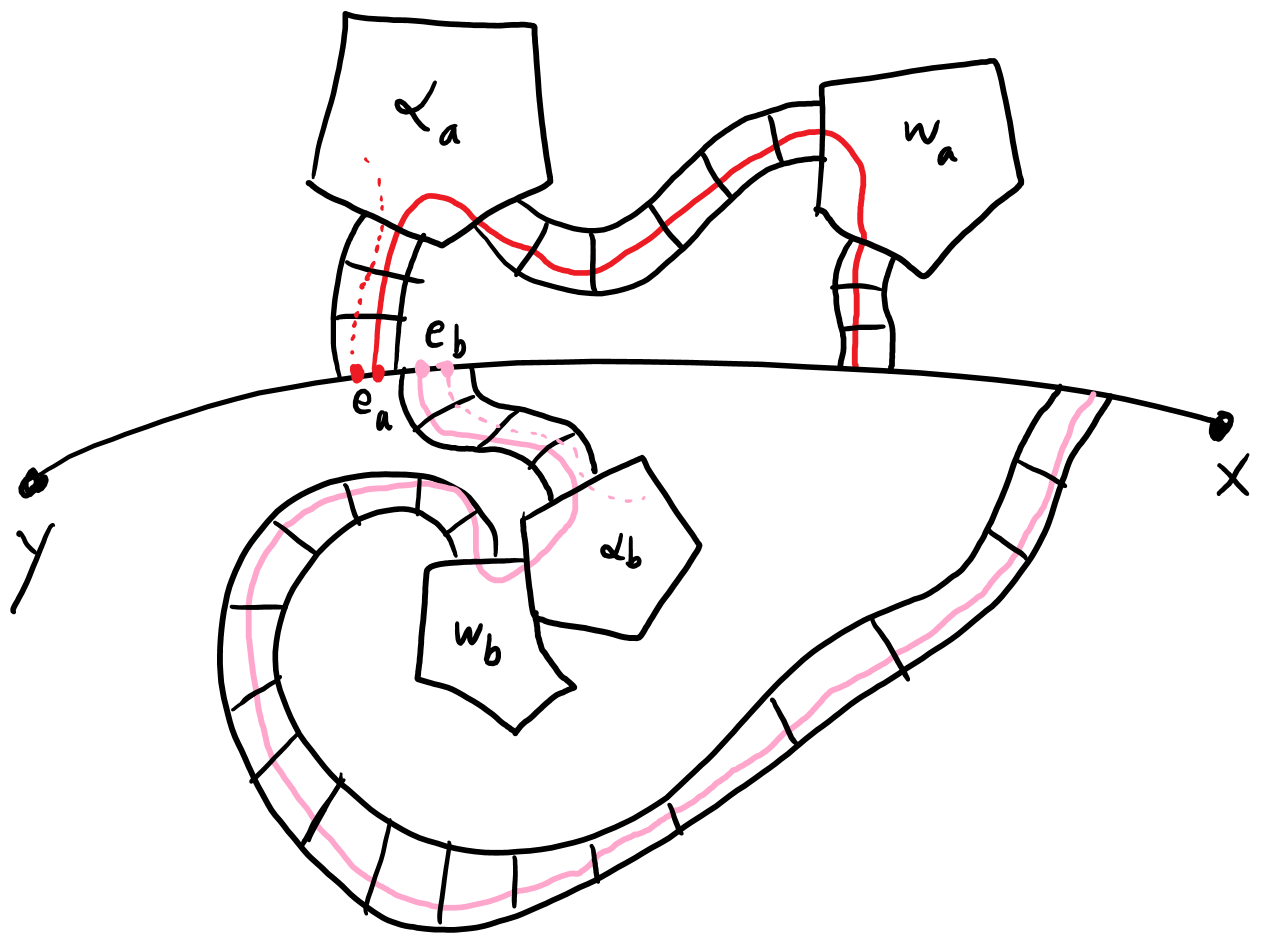}
	\caption{\footnotesize Proving statements (i) and (ii). The point is that $\alpha_a$ and $\alpha_b$ prevent each other from being extreme, provided that $H_a$ and $H_b$ both contain at least two essential $2$-cells.}
	\label{fig:fig9}
\end{figure}

The following fact will be useful in proving statement (iii): Suppose $\omega_a$ is extreme with exposed essential edge $f_a$. Then some element of $[f_a]_{\omega_a}$ lies along $\gamma$. To see this, not that in case some element of $[f_a]_{\omega_a}$ contains the terminal point of $\lambda_a$ along $\gamma$, this is obvious. Otherwise, we may pick two elements from $[f_a]_{\omega_a}$ on opposite sides of $\lambda_a$, neither of which lies along $\gamma$, for contradiction. Connect these two edges by a snipping arc running across $\omega_a$. This arc is non-separating in $\pl{Y}$, since there is a path from one side to the other in the graph $(\gamma\cup\lambda_a)\cap Y_a$; this contradicts Lemma \ref{snip}. Similarly, if $\omega_b$ is extreme with exposed essential edge $f_b$, then some element of $[f_b]_{\omega_b}$ lies along $\gamma$.

Finally, we prove statement (iii). Suppose for contradiction that $\omega_a\neq\omega_b$, but both are extreme. Among all exposed essential edges $e'$ of $\omega_a$ (meaning that all members of $[e']_{\omega_a}$ lie on the boundary of $\pl{Y}$), choose the one which is on $\gamma$ and closest to $x$ along $\gamma$ and call it $f_a$. Define $f_b$ similarly. Note $f_a\neq f_b$ since all elements of both $[f_a]_{\omega_a}$ and $[f_b]_{\omega_b}$ lie in $\partial\pl{Y}$. There are two cases according to whether $f_b$ is closer to $x$ than $f_a$ or vice-versa.

Suppose first that $f_b$ is closer to $x$ than $f_a$. In this case we will show that there are two edges in $\partial\omega_a\cap\partial \pl{Y}$ which can be connected together by a non-separating snipping arc through $\omega_a$, contradicting Lemma \ref{snip}. Orient $f_a$ so that it points towards $x$ along $\gamma$ and extend this orientation to $\partial\omega_a$. Let $g_a$ be the next element of $[f_a]_{\omega_a}$ after $f_a$. Note that $g_a$ does not lie along $\gamma$. Indeed, if it does, then by choice of $f_a$, $g_a$ lies closer to $y$ along $\gamma$ than $f_a$ by Lemma \ref{notweird}. Lemma \ref{notweird} also implies that every element of $[f_a]_{\omega_a}$ lies along $\gamma$, which contradicts Lemma \ref{noorbits}.

Connect midpoints of $f_a$ and $g_a$ together by a snipping arc that runs across $\omega_a$ and let $S$ be a closed neighborhood of this arc which includes the vertices $i(f_a)$, $t(f_a)$, $i(g_a)$, and $t(g_a)$ but is small enough so that $\partial S\cap\partial\omega_a=f_a\cup g_a$. Orient $S$ by declaring that the edge of $S$ running from $t(f_a)$ to $i(g_a)$ is the front edge of $S$, and the edge running from $i(f_a)$ to $t(g_a)$ is the back edge. Let $v_a$ denote the first point (with respect to the orientation of $\lambda_a$) in $\omega_a\cap\lambda_a$. Note that $v_a$ does not lie in $S$, for otherwise $\lambda_a$ runs through the center of $S$ connecting $g_a$ to $f_a$, but because $g_a$ lies on the boundary of $\pl{Y}$ this would mean $g_a=e_a$, contradicting that $g_a$ does not lie on $\gamma$. Note also that $e_a\neq f_a$, as this scenario would imply $\alpha_a=\omega_a$ and either force $g_a$ to lie on $\gamma$ or give rise to another contradiction to Lemma \ref{notweird}.

There are now some cases to consider.

\begin{itemize}

\item Case 1: The vertices $v_a$ and $t(f_a)$ lie in different components of $\omega_a\setminus\intr{S}$. This case is illustrated in figure \ref{fig:fig10}. In this case we find a path from $t(f_a)$ to the back edge of $S$ in $\pl{Y}\setminus\intr{S}$ as follows:

Starting from $t(f_a)$, travel along $\gamma$ until reaching $f_b$. From $i(f_b)$, travel inside the interior of $\omega_b$ to reach $\lambda_b$. Next, travel backwards along $\lambda_b$ all the way through $H_b$ until reaching $e_b$. If at any point we cross $S$, then it means that $\omega_a$ is identified with an essential $2$-cell in the ladder $H_b$ \textit{distinct} from $\omega_b$, but this cannot happen since we already know that none of these $2$-cells are extreme. Once arriving at $e_b$, travel within $e_b\cup e_a$ to $\lambda_a$ -- here we will not touch $S$ because $e_a\neq g_a$ and $e_b\neq g_a$ since $g_a$ does not lie on $\gamma$, $e_b\neq f_a$ since $\alpha_b\neq\omega_a$ but $f_a$ lies on the boundary of $\pl{Y}$, and $e_a\neq f_a$ as previously observed. Finally, continue along $\lambda_a$ all the way through $H_a$ until entering $\omega_a$ through $v_a$ and reaching the back edge of $S$ in $\omega_a$ (we will not touch $S$ in any other essential $2$-cell since $H_a$ is a subcomplex of $\ucc{X}$). The path we have found connects the front and back edges of $S$ in $\pl{Y}\setminus\intr{S}$ and contradicts Lemma \ref{snip}.

\begin{figure}[htbp]
	\centering
		\includegraphics[width=0.7\textwidth]{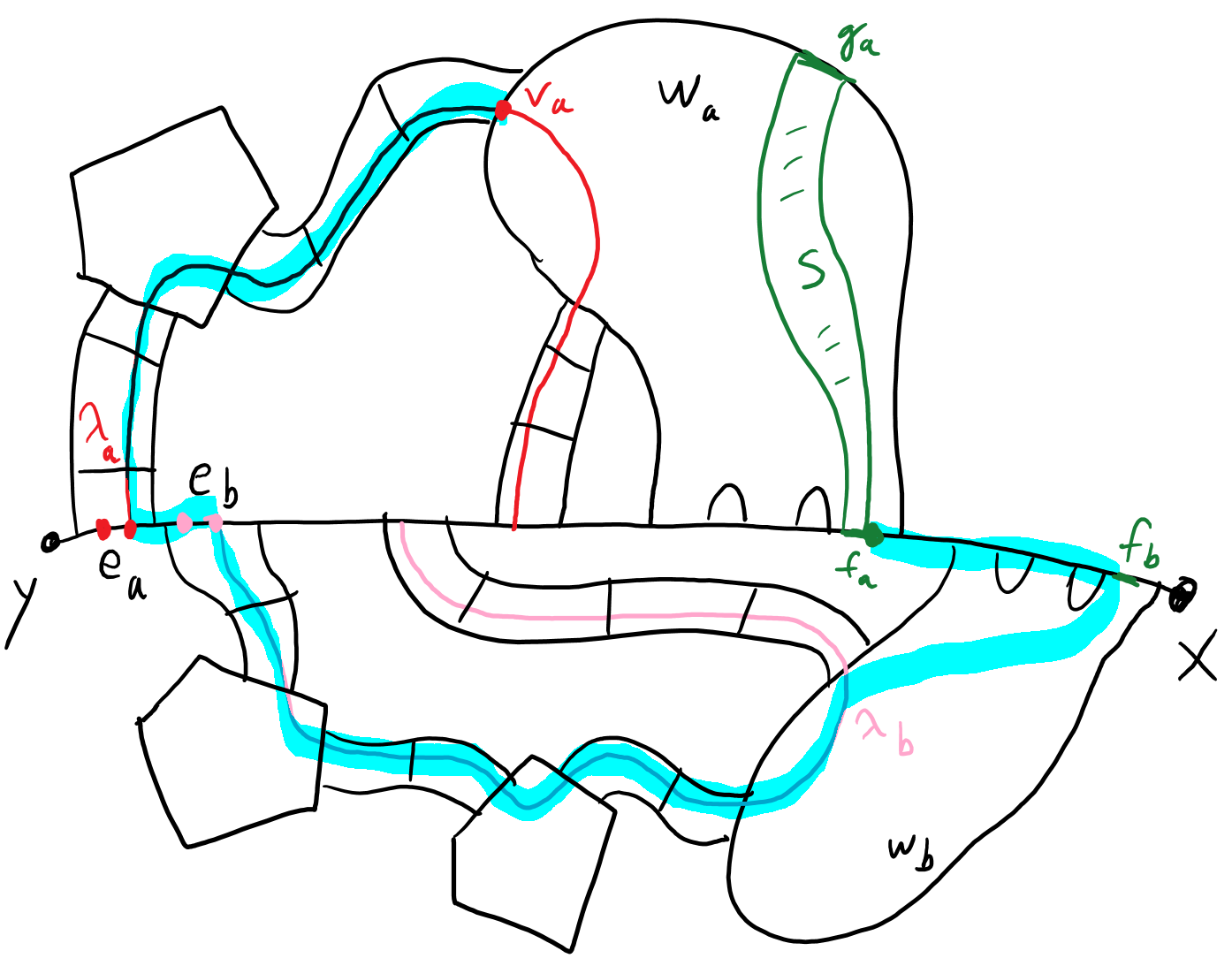}
	\caption{\footnotesize An example of what could happen in case 1. The highlighted blue path gives the contradiction to Lemma \ref{snip}. }
	\label{fig:fig10}
\end{figure}


\item Case 2: The vertices $v_a$ and $t(f_a)$ lie in the same component of $\omega_a\setminus\intr{S}$. This case further breaks into two subcases. Note that $e_a\neq f_a$ as previously observed.

\begin{itemize}

\item Subcase 1: The edge $e_a$ is strictly closer to $y$ along $\gamma$ than $f_a$ is. This subcase is illustrated in figure \ref{fig:fig11}. In this case we find a path from $t(f_a)$ to the back edge of $S$ in $\pl{Y}\setminus\intr{S}$ as follows:

Starting from $t(f_a)$, travel along $\gamma$ until reaching $i(f_b)$, and then through the interior of $\omega_b$ to reach $\lambda_b$. Travel backwards through $\lambda_b$ to reach $e_b$ (for the same reasons as the previous case, this path does not touch the interior of $S$). Since $e_b$ is adjacent to $e_a$ and $e_b\neq f_a$ (as in the previous case), it is the case that $e_b$ is strictly closer to $y$ along $\gamma$ than $f_a$ is. Thus there is a path in $\gamma$ from the initial point of $\lambda_b$ to $i(f_a)$ which avoids $S$. We have again contradicted Lemma \ref{snip}.

\begin{figure}[htbp]
	\centering
		\includegraphics[width=0.7\textwidth]{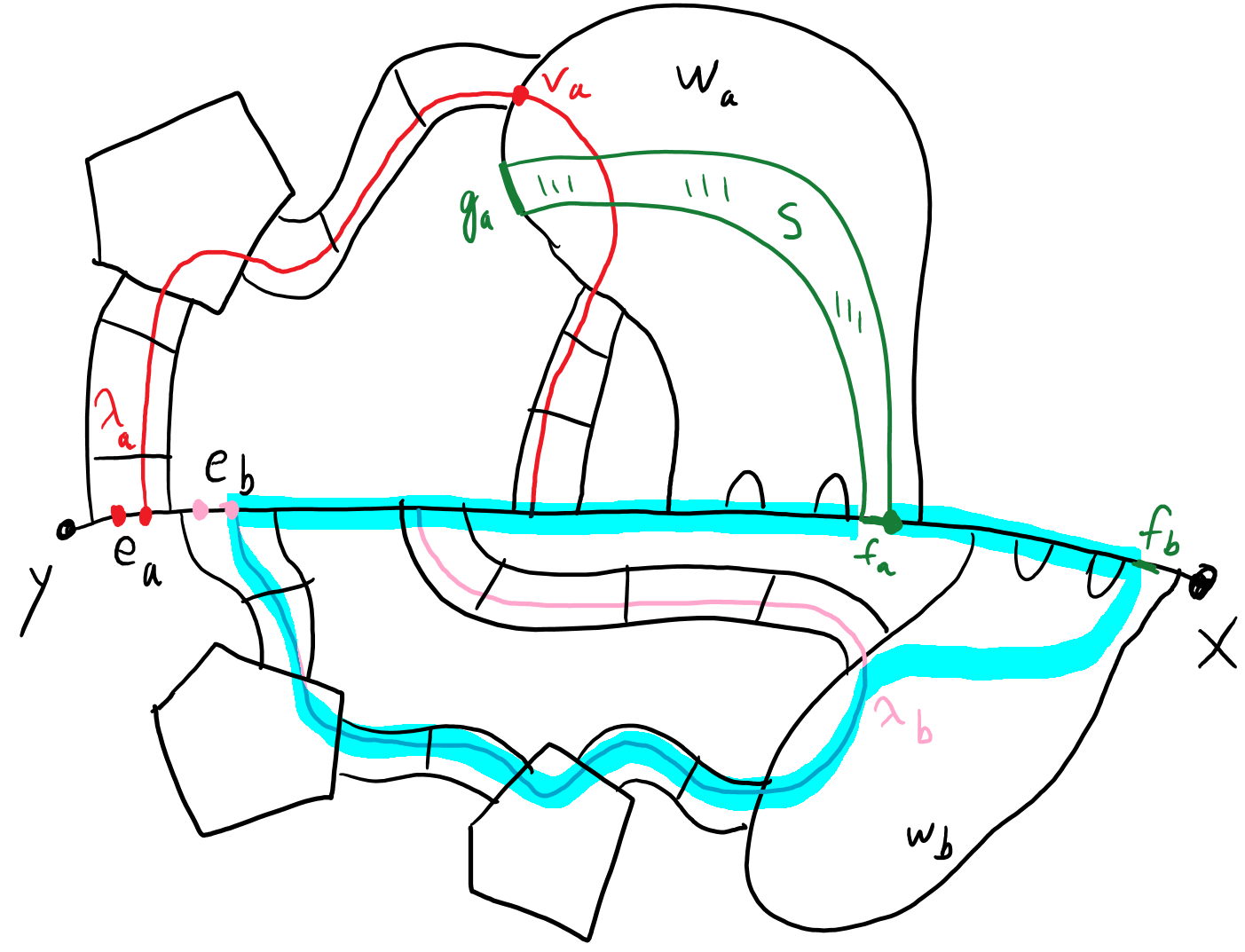}
	\caption{\footnotesize An example of subcase 1.}
	\label{fig:fig11}
\end{figure}

\item Subcase 2: The edge $e_a$ is strictly closer to $x$ along $\gamma$ than $f_a$ is. This subcase is illustrated in figure \ref{fig:fig12}. Let $e'_a$ be the edge of $\gamma$ which is dual to the terminal edge of $\lambda_a$, and oriented so that it points in the direction of $x$. Note that $e_a\neq e'_a$ (for example by Lemma \ref{random2}), and $e'_a$ is strictly closer to $x$ along $\gamma$ than $e_a$. Let $w_a^{\text{front}}$ and $w_a^{\text{back}}$ be the vertices of $S\cap\lambda_a$, labeled according to whether they are on the front or back edge of $S$. In this case we find a path from $w_a^{\text{back}}$ to $w_a^{\text{front}}$ in $\pl{Y}\setminus\intr{S}$ as follows:

Travel from $w_a^{\text{back}}$ to $e'_a$ along $\lambda_a$ in the forward direction, and travel backwards along $\gamma$ from $e_a'$ to $e_a$. Then simply travel forward along $\lambda_a$ through $H_a$ until reaching $w_a^{\text{front}}$. This again contradicts Lemma \ref{snip}.

\begin{figure}[htbp]
	\centering
		\includegraphics[width=0.7\textwidth]{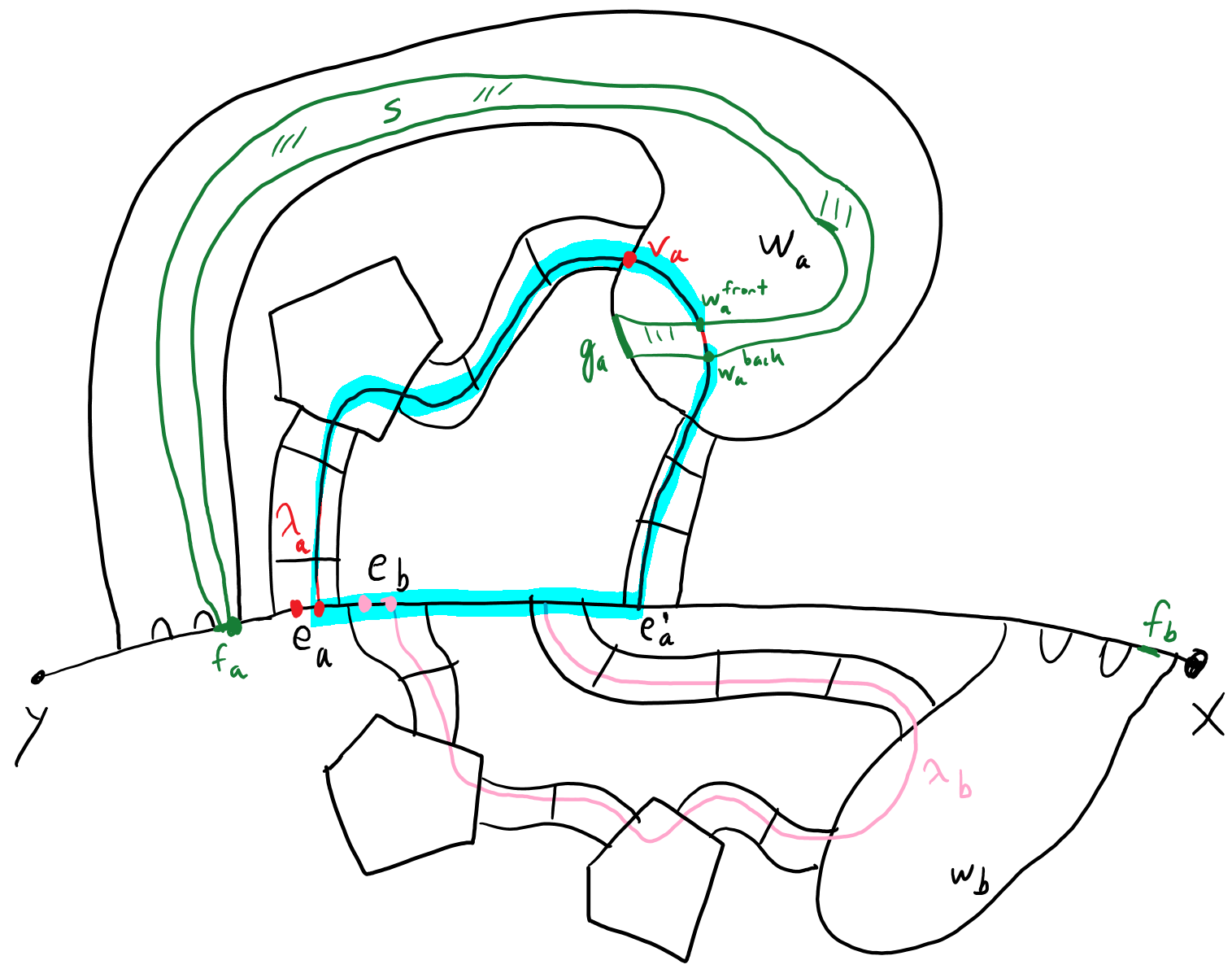}
	\caption{\footnotesize An example of subcase 2. The idea that $e_a$ could be closer to $x$ than $f_a$ seems like a strange pathology, but we have to deal with it separately since we have not ruled it out.}
	\label{fig:fig12}
\end{figure}

\end{itemize}

\end{itemize}

For the case in which $f_a$ is closer to $x$ than $f_b$, the argument is identical, except that we exchange the roles of $a$ and $b$ in the above argument. Note that the above argument does not depend on the order in which $e_a$ and $e_b$ occur along $\gamma$, but only uses that these edges are adjacent in $\gamma$.
\end{proof}

\noindent
\begin{lem}
\label{9.4}
Let $\gamma$ be a geodesic in $\os{\ucc{X}}$ with endpoints $0$-cells $x$ and $y$. Suppose that $n(X)\geq 4$. For any $1$-cell $e$ of $\gamma$, there exists a wall that intersects $\gamma$ exactly once, and the point of intersection is within $W_X+1$ edges of $e$.
\end{lem}

\begin{proof}
As in the proof of Lemma \ref{georqc}, let $W_X$ be an upper bound on the number of edges (essential or not) in the attaching map of the elements of $C(X)$.

If either wall dual to $e$ does not double-cross $\gamma$, then we are done. Thus, assume that $\Lambda_e^x$ double-crosses $\gamma$. Fix a wall segment $\lambda_e^x$ associated to this double crossing and let $Y_e^x$ be the associated double-crossed ladder. By Lemma \ref{direction}, we may assume that $Y_e^x$ bends in the direction of $x$. By Lemma \ref{random3}, the first essential $2$-cell through which $Y_e^x$ returns, $\alpha_a=\alpha_e^x$, exists. Let $\gamma_x$ be the subsegment of $\gamma$ between $e$ and $x$, including $e$. Consider the sequence of successive edges of $\gamma_x$ starting with $e$ and moving towards $x$, $\{e=e_1,e_2,e_3,\ldots\}$. Let $k$ be the largest integer with the property that $\Lambda_{e_k}^x$ double crosses $\gamma$ and such that $\alpha_a$ is the first essential $2$-cell through which some wall segment $\lambda_{e_k}^x$ returns. Since there are at most $W_X$ wall segments passing through $\alpha_a$, $k\leq W_X$. Define $Y_a$ to be the double-crossed ladder associated to $\lambda_{e_k}^x$. By Lemma \ref{direction}, we may assume $Y_a$ bends in the direction of $x$. In particular, $e_{k+1}$ exists.

Now, observe that the wall $\Lambda_{k+1}^x$ crosses $\gamma$ exactly once. Indeed, if not then there is a double-crossed ladder $Y_b=Y_{k+1}^x$ at $(e_{k+1},x)$ which bends in the direction of $x$ by Lemma \ref{direction}, and $\alpha_a\neq\alpha_b$ by definition of $k$. Thus $Y_a\cup Y_b$ is a pair of double-crossed ladders, contradicting Lemma \ref{notrouble}.
\end{proof}

Proposition \ref{9.1} follows easily (assuming of course that $n(X)\geq 4$).

Problem: Just as Lauer and Wise do, we wonder -- Does $\ucc{X}$ satisfy the linear separation property relative to its walls when $n(X)\in\{2,3\}$? It appears difficult to produce a pair of double-crossed ladders in this situation, since one has less control over the direction in which double-crossed ladders bend.

\section{Existence of the action}
\label{sect:action}

In this section we will prove the main theorem, that is that $\pi_1(X)$ acts properly and cocompactly on a $\cz$ cube complex. We first invoke the so-called ``Sageev contruction'' to obtain an action of $\pi_1(X)$ on a $\cz$ cube complex.

\noindent
\begin{dfn} (\textbf{Wallspace/dual cube complex}). Let $Y$ be a metric space and let $\mathcal{W}$ be a collection of closed, connected subspaces of $Y$, each of which separates $Y$ into two components. We call $(Y,\mathcal{W})$ a (geometric) \emph{wallspace}. If a group $G$ acts properly and cocompactly on $Y$ preserving both its metric and wallspace structures, then Sageev shows that $G$ acts on a $\cz$ cube complex $\mathcal{C}(Y)$, called the \emph{dual cube complex} \cite{ms}. A summary can be found in \cite[Construction 3.2, Theorem 3.7, Remark 3.11]{hw2}.
\end{dfn}

Properness of this action in our setting will follow immediately from what we proved in Section \ref{sect:ls}. Cocompactness will follow by an application of \cite[Theorem 7.12]{hw2}. We state a simplified version of this theorem below.

\noindent
\begin{thm} (cf \cite[Theorem 3.1]{jw}).
\label{jw3.1}
Let $(Y,\mathcal{W})$ be a wallspace. Suppose $G$ acts properly and cocompactly on $Y$ preserving both its metric and wallspace structures, and the action on $\mathcal{W}$ has only finitely many $G$-orbits of walls. Suppose $G$ is hyperbolic relative to $\PP$ with $\PP$ finite. Suppose $\stab{\Lambda}$ acts cocompactly on $\Lambda$ and is relatively quasiconvex for each wall $\Lambda\in\mathcal{W}$. For each $P\in\PP$ let $Y_P\subset Y$ be a nonempty $P$-invariant $P$-cocompact subspace. Let $\mathcal{C}(Y)$ be the cube complex dual to $(Y,\mathcal{W})$ and for each $P\in\PP$ let $\mathcal{C}_*(Y_P)$ be the cube complex dual to $(Y_P,\mathcal{W}_P)$, where $\mathcal{W}_P$ consists of all walls $\Lambda$ with the property that $\diam{\Lambda\cap\nbhd{Y_P}{d}}=\infty$ for some $d=d(\Lambda)$.

Then there exists a compact subcomplex $K$ such that $\mathcal{C}(Y)=GK\cup\bigcup_{P\in\PP} G\mathcal{C}_*(Y_P)$. In particular, $G$ acts cocompactly on $\mathcal{C}(Y)$ provided that each $\mathcal{C}_*(Y_P)$ is $P$-cocompact.
\end{thm}

For us, $G=\pi_1(X)$, $Y=\ucc{X}$, $\mathcal{W}$ is the collection of walls we defined in $\ucc{X}$, and $\PP$ is the finite collection of vertex groups of $X$. Each vertex group $P$ has an associated vertex space $V_P$ in $X$ (a compact NPC cube complex). Fix a base-point in $\ucc{X}$ and let $Y_P$ to be the copy of the universal cover of $V_P$ in $\ucc{X}$ (a $\cz$ cube complex) with $\stab{Y_P}=P$. 

In order to apply this theorem, it remains to show that each $\mathcal{C}_*(Y_P)$ is $P$-cocompact, as we will see. The following key lemma says, roughly, that a relative geodesic with large projection to $Y_P$ comes very close to $Y_P$.

\noindent
\begin{lem}
\label{bigproj}
Fix $Y_P$. Suppose $\gamma$ is a relative geodesic in $\os{\ucc{X}}$ with endpoints $0$-cells $x$ and $y$. Let $\pi_x$ and $\pi_y$ be nearest-point projections of $x$ and $y$ to the vertex set of $Y_P$. For all $d\geq 0$, there exists $R\geq0$ such that if $d(x,\pi_x)\leq d$, $d(y,\pi_y)\leq d$, and $d(\pi_x,\pi_y)>R$, then there is an essential edge $e$ of $\gamma$ within $W_X/2$ edges of $Y_P$ (where $W_X$ is an upper bound on the lengths of attaching maps of essential $2$-cells in $X$).
\end{lem}

\begin{proof}
First, note that if any edge of $\gamma$ maps to $Y_P$, then the closest essential edge along $\gamma$ to this edge satisfies the conclusion of the lemma with $R=0$.

Let $d$ be given and assume $d(x,\pi_x)\leq d$ and $d(y,\pi_y)\leq d$. Assume that $x$ and $y$ are far enough apart that $d(\pi_x,\pi_y)>W_X+4d+2$. By the triangle inequality, this will imply in particular that $d(x,y)>2d$.

Form a quadrilateral as follows: Let $\gamma_x$ (resp. $\gamma_y$) be a geodesic edge path from $x$ to $\pi_x$ (resp. $y$ to $\pi_y$), and let $\gamma'$ be a geodesic edge path from $\pi_x$ to $\pi_y$. Orient everything so that $\sigma=\gamma\gamma_y\gamma'\gamma_x$ is a closed loop. Note that $\gamma'$ lies in $Y_P$ by Lemma \ref{vertexspacesconvex}. Also note that there is no backtracking in any of $\gamma$, $\gamma_y$, $\gamma_x$, or $\gamma'$, so there can only be backtracking at the corners. We make $\sigma$ cyclically reduced as follows. First note that there is no backtracking of $\sigma$ at $\pi_x$ or $\pi_y$ by the fact that these points are nearest-point projections of $x$ and $y$ to $Y_P$ and $\gamma'$ lies in $Y_P$. Now, there may be backtracking at $x$, so let $x'$ be the last vertex along $\gamma$ (from $x$) in the image of $\gamma_x$, and similarly define $y'$ to be the last vertex along $\gamma$ (from $y$) in the image of $\gamma_y$. The fact that $d(x,y)>2d$ ensures that there will remain at least one edge of $\gamma$ running from $x'$ to $y'$. Note also that if $x'=\pi_x$ or $y'=\pi_y$, then $\gamma\cap Y_P$ is nonempty and we are done with $R=0$ as before. Let $\gamma_0=\gamma|_{[x',y']}$, $\gamma_{x'}=\gamma_x|_{[\pi_x,x']}$, and $\gamma_{y'}=\gamma_y|_{[y',\pi_y]}$. Redefine $\sigma=\gamma_0\gamma_{y'}\gamma'\gamma_{x'}$. It is clear that there is no folding of $\sigma$ at $x'$ or $y'$ so $\sigma$ is reduced and cyclically reduced.

Fill $\sigma$ with a reduced disk diagram $D\to\ucc{X}(\to X)$ using Lemma \ref{vkl}. If $D$ has no essential $2$-cells then all of $D$ maps to $Y_P$, so set $R=0$ and we are done. Otherwise, Suppose $\alpha$ is an exposed $2$-cell of $D$ with exposed edge $e$. We make the following observations:

\begin{itemize}
\item It is not the case that there exist $e,f\in[e]_\alpha$ with $e$ along $\gamma_{x'}$ and $f$ along $\gamma_{y'}$, otherwise $\partial\alpha$ offers a shortcut between $\gamma_{x'}$ and $\gamma_{y'}$ so that $d(\pi_x,\pi_y)<W_X/2+2d<W_X+4d+2$, a contradiction.
\item It is the case that $\cup[e]_\alpha\nsubset\gamma_{x'}$, $\cup[e]_\alpha\nsubset\gamma_{y'}$, and $\cup[e]_\alpha\nsubset\gamma_0$, since all of these paths are relative geodesics (by Lemma \ref{noorbits}).
\item No element of $[e]_\alpha$ lies along $\gamma'$ (since by Lemma \ref{vertexspacesconvex} no edge of $\gamma'$ is essential).
\end{itemize}

Thus $\alpha$ must ``straddle'' $x'$, i.e. at least one element of $[e]_\alpha$ lies in $\gamma_0$ and at least one in $\gamma_{x'}$, and all elements of $[e]_\alpha$ lie in $\gamma_{x'}\cup\gamma_0$. Alternatively, $\alpha$ could straddle $y'$.

Now we claim that $D$ contains at most $2$ extreme $2$-cells. To see this, first note that there is a natural linear order on the extreme two cells of $D$ induced by the order in which their boundaries are encountered while traversing $\gamma_0$ from $x'$ to $y'$. If there are three or more extreme essential $2$-cells, then we may choose one which is not the first or last with respect to this order. Call this $2$-cell $\alpha$ and suppose that $\alpha$ is exposed with exposed edge $e$. Without loss of generality, we may assume that $\alpha$ straddles $x'$. Let $e_1$ be an element of $[e]_\alpha$ along $\gamma_0$ and $e_2$ an element of $[e]_\alpha$ along $\gamma_{x'}$. Let $\gamma_1$ and $\gamma_2$ be the two minimal paths in $\partial\alpha$ containing $e_1$ and $e_2$, and labeled so that the component of $D\setminus\gamma_2$ which contains $x'$ also contains $\alpha$. Now any candidate for an extreme subpath of $\partial\alpha$ containing all elements of $[e]_\alpha$ must contain $\gamma_1$ or $\gamma_2$. But note that the image of $\gamma_1$ in the auxiliary diagram $\aux{D}$ internally intersects an essential $2$-cell of $\aux{D}$ which lies before $\alpha$ in the order determined by $\gamma_0$. Similarly, the image of $\gamma_2$ in $\aux{D}$ internally intersects an essential $2$-cell of $\aux{D}$ which lies after $\alpha$ in the order determined by $\gamma_0$. Since $e$ was arbitrary, this shows that no extreme subpath of $\partial\alpha$ exists, i.e., $\alpha$ is not extreme.

Using this claim and applying Proposition \ref{5.4} and Lemma \ref{4.7}, we see that every essential $2$-cell of $D$ is external.

Now, let $D'$ be the maximal connected subdiagram of $D$ containing $\gamma'$ and mapping to $Y_P$. Call the other arc of $\partial D'$ from $\pi_y$ to $\pi_x$, $\gamma_1$. Note that no edge of $\gamma_1$ lies in $\gamma_{x'}$ or $\gamma_{y'}$ since $\pi_y$ and $\pi_x$ are nearest-point projections. If any edge of $\gamma_1$ belongs to $\gamma_0$, then some edge of $\gamma$ maps $Y_P$ and we are done. Thus we may assume that every edge of $\gamma_1$ belongs to an essential $2$-cell of $D$ lying in $D\setminus D'$.

\begin{figure}[htbp]
	\centering
		\includegraphics[width=0.8\textwidth]{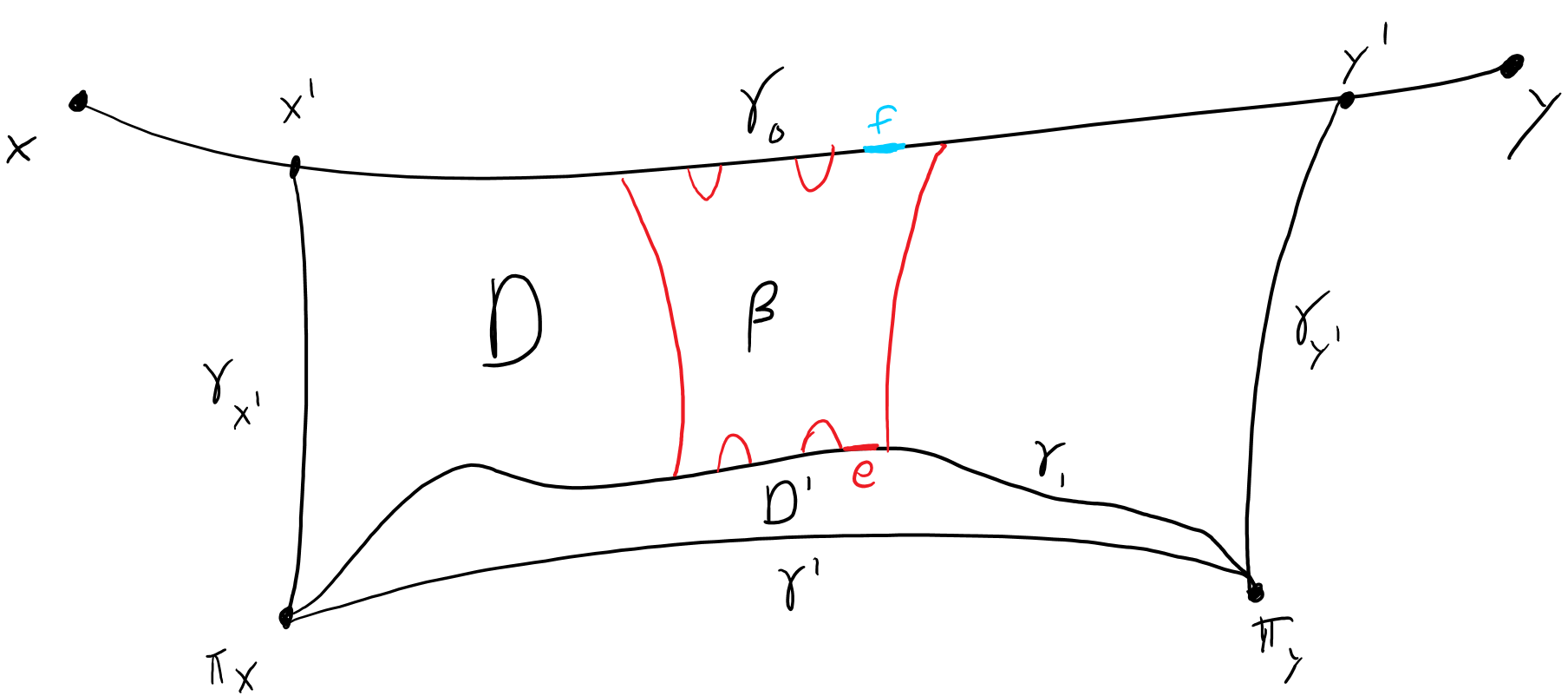}
	\caption{\footnotesize The general case in this lemma. The subdiagram $D'$ maps entirely to $Y_P$. By choosing $\pi_x$ and $\pi_y$ sufficiently far apart, we can find the essential $2$-cell $\beta$ which does not intersect $\gamma_{x'}$ or $\gamma_{y'}$. Since $\beta$ is external in $D'$, we can find the blue essential edge $f$ on $\gamma$, showing that $\gamma$ passes close to $Y_P$.}
	\label{fig:geodclose}
\end{figure}

Since $\ell(\gamma_1)\geq\ell(\gamma')>W_X+4d+2\geq W_X+2d+2$, we may choose an edge $e$ of $\gamma_1$ with the property that $d(e,\pi_x)>W_X/2+d$ and $d(e,\pi_y)>W_X/2+d$. Let $\beta$ be the essential $2$-cell of $D$ with $e$ in its boundary. The observation above implies $\beta$ is external with essential edge $f$ (say) along $\partial D$. Observe that $f$ does not lie along $\gamma_{x'}$, as this would offer a shortcut through $\partial\beta$ from $e$ to $\pi_x$ of length less than or equal to $W_X/2+d$, contradicting the triangle inequality. Similarly, $f$ does not lie along $\gamma_{y'}$. Thus $f$ lies along $\gamma_0$. Now the shorter path along $\partial\beta$ from $e$ to $f$ maps to a path in $\ucc{X}$ from $Y_P$ to an essential edge of $\gamma$ of length less than or equal to $W_X/2$, and we see that $R\geq W_X+4d+2$ satisfies the conclusion of the lemma. See figure \ref{fig:geodclose}.
\end{proof}

\noindent
\begin{lem} \label{periphcc}
Each $\mathcal{C}_*(Y_P)$ is $P$-cocompact.
\end{lem} 

\begin{proof}

Supppose that $\Lambda$ is a wall of $\ucc{X}$ with the property that $\diam{\Lambda\cap\nbhd{Y_P}{d}}=\infty$ for some $d$. Consider points $x$ and $y$ which are very far apart in $\Lambda\cap\nbhd{Y_P}{d}$. Let $\pi_x$ and $\pi_y$ be their projections to $Y_P$, and let $\gamma$ be a relative geodesic between them. By the triangle inequality, $d(\pi_x,\pi_y)$ grows with $d(x,y)$. Choose $x$ and $y$ far enough apart that $d(\pi_x,\pi_y)>R$, where $R(d)$ is chosen according to Lemma \ref{bigproj}. By that lemma, there is a point $z$ in $Y_P$ within distance $W_X/2$ of an essential edge $e$ of $\gamma$. By geometric relative quasiconvexity of wall carriers (Lemma \ref{georqc}), the distance from $e$ to the carrier of $\Lambda$ is uniformly bounded, which also means the distance from $e$ to $\Lambda$ is uniformly bounded since any point in the carrier is within $W_X/2$ of $\Lambda$. So $\Lambda$ passes uniformly close to $Y_P$ independently of $\Lambda$, say within some distance $d'$.

Now, since $P=\stab{Y_P}$ acts cocompactly on $Y_P$ (its action is a covering space action and the vertex space for $P$ is a compact NPC cube complex), $P$ also acts cocompactly on $\nbhd{Y_P}{d'}$ by local finiteness of $\ucc{X}$. Since every wall $\Lambda$ with $\diam{\Lambda\cap\nbhd{Y_P}{d}}=\infty$ for some $d$ meets $\nbhd{Y_P}{d'}$, there are finitely many $P$-orbits of such walls. This is exactly what it means for $\mathcal{C}_*(Y_P)$ to be $P$-cocompact.
\end{proof}

Putting everything together, we have the main theorem for staggered generalized $2$-complexes with locally indicable vertex groups and $n(X)\geq 4$.

\noindent
\begin{thm}
\label{main2}
Let $X$ be a staggered generalized $2$-complex. Suppose that $X$ has locally indicable vertex groups and that $n(X)\geq 4$. Suppose that for each vertex space $V$ of $X$, $\pi_1(V)$ acts properly and cocompactly on a $\cz$ cube complex. Then $\pi_1(X)$ acts properly and cocompactly on a $\cz$ cube complex.
\end{thm}

\begin{proof}

As before, let $G=\pi_1(X)$. Let $\mathcal{W}$ be the collection of walls in $\ucc{X}$ coming from the construction of Section \ref{sect:wl}. Let $\mathcal{C}$ be the cube complex dual to the action of $G$ on the wallspace $(\ucc{X},\mathcal{W})$.

By Proposition \ref{9.1}, the wallspace $(\ucc{X},\mathcal{W})$ satisfies linear separation. By \cite[Theorem 5.2]{hw2}, the action of $G$ on $\mathcal{C}$ is proper.

Let $\PP$ be the finite collection of vertex groups of $X$. Each vertex group $P$ has an associated vertex space $V_P$ in $X$ (a compact NPC cube complex). Fix a base-point in $\ucc{X}$ and let $Y_P$ to be the copy of the universal cover of $V_P$ in $\ucc{X}$ (a $\cz$ cube complex) with $\stab{Y_P}=P$.

Observe that all hypotheses of Theorem \ref{jw3.1} are satisfied. Indeed, it is clear that $G$ acts properly and cocompactly on $\ucc{X}$ preserving both its metric and wallspace structures, and the action on $\mathcal{W}$ has only finitely many $G$-orbits of walls. Relative hyperbolicity of $(G,\PP)$ was shown in Lemma \ref{rh}. For each wall $\Lambda$, Lemma \ref{wallsrqc} implies $\stab{\Lambda}$ acts cocompactly on it, and we showed $\stab{\Lambda}$ is relatively quasiconvex in Proposition \ref{stabsrqc}. Finally,  each $\mathcal{C}_*(Y_P)$ is $P$-cocompact by Lemma \ref{periphcc}.

Applying Theorem \ref{jw3.1}, the action of $G$ on $\mathcal{C}$ is cocompact and the theorem is proved.
\end{proof}

\noindent
\begin{cor}
\label{main3}
Let $A$ and $B$ be locally indicable, cubulable groups, $w$ a word in $A*B$ which is not conjugate into $A$ or $B$, and $n\geq 4$. Then $G=A*B/\nc{w^n}$ is cubulable.
\end{cor}

\begin{proof}
We may assume that $w$ is cyclically reduced. Build a model space $X$ for $G=A*B/\nc{w^n}$ by starting with a dumbell space $X_A\vee X_B$ of non-positively curved cube complexes with $\pi_1(X_A)=A$ and $\pi_1(X_B)=B$, and then attaching a $2$-cell to a path corresponding to the word $w^n$, so that $\pi_1(X)=G$. Observe that $X$ is trivially staggered generalized and Theorem \ref{main2} applies.
\end{proof}


\phantomsection

\addcontentsline{toc}{section}{References}

\bibliographystyle{amsalpha}

\bibliography{paper}

\newcommand{\etalchar}[1]{$^{#1}$}
\providecommand{\bysame}{\leavevmode\hbox to3em{\hrulefill}\thinspace}
\providecommand{\MR}{\relax\ifhmode\unskip\space\fi MR }
\providecommand{\MRhref}[2]{%
  \href{http://www.ams.org/mathscinet-getitem?mr=#1}{#2}
}
\providecommand{\href}[2]{#2}
\begin{thebibliography}{CCJ{\etalchar{+}}01}

\bibitem[Ago13]{agol}
Ian Agol, \emph{The virtual {H}aken conjecture}, Doc. Math. \textbf{18} (2013),
  1045--1087, With an appendix by Agol, Daniel Groves, and Jason Manning.
  \MR{3104553}

\bibitem[Bri02]{brid}
Martin~R. Bridson, \emph{The geometry of the word problem}, Invitations to
  geometry and topology, Oxf. Grad. Texts Math., vol.~7, Oxford Univ. Press,
  Oxford, 2002, pp.~29--91. \MR{1967746}

\bibitem[BW12]{bw}
Nicolas Bergeron and Daniel~T. Wise, \emph{A boundary criterion for
  cubulation}, Amer. J. Math. \textbf{134} (2012), no.~3, 843--859.
  \MR{2931226}

\bibitem[CCJ{\etalchar{+}}01]{ccjjv}
Pierre-Alain Cherix, Michael Cowling, Paul Jolissaint, Pierre Julg, and Alain
  Valette, \emph{Groups with the {H}aagerup property}, Progress in Mathematics,
  vol. 197, Birkh\"{a}user Verlag, Basel, 2001, Gromov's a-T-menability.
  \MR{1852148}

\bibitem[GM08]{gm}
Daniel Groves and Jason~Fox Manning, \emph{Dehn filling in relatively
  hyperbolic groups}, Israel J. Math. \textbf{168} (2008), 317--429.
  \MR{2448064}

\bibitem[How81]{h1}
James Howie, \emph{On pairs of {$2$}-complexes and systems of equations over
  groups}, J. Reine Angew. Math. \textbf{324} (1981), 165--174. \MR{614523}

\bibitem[How82]{h2}
\bysame, \emph{On locally indicable groups}, Math. Z. \textbf{180} (1982),
  no.~4, 445--461. \MR{667000}

\bibitem[How87]{h3}
\bysame, \emph{How to generalize one-relator group theory}, Combinatorial group
  theory and topology ({A}lta, {U}tah, 1984), Ann. of Math. Stud., vol. 111,
  Princeton Univ. Press, Princeton, NJ, 1987, pp.~53--78. \MR{895609}

\bibitem[HP84]{hp}
J.~Howie and S.~J. Pride, \emph{A spelling theorem for staggered generalized
  {$2$}-complexes, with applications}, Invent. Math. \textbf{76} (1984), no.~1,
  55--74. \MR{739624}

\bibitem[Hru10]{hr1}
G.~Christopher Hruska, \emph{Relative hyperbolicity and relative quasiconvexity
  for countable groups}, Algebr. Geom. Topol. \textbf{10} (2010), no.~3,
  1807--1856. \MR{2684983}

\bibitem[HW99]{hsuwise}
Tim Hsu and Daniel~T. Wise, \emph{On linear and residual properties of graph
  products}, Michigan Math. J. \textbf{46} (1999), no.~2, 251--259.
  \MR{1704150}

\bibitem[HW01]{hw1}
G.~Christopher Hruska and Daniel~T. Wise, \emph{Towers, ladders and the {B}.
  {B}. {N}ewman spelling theorem}, J. Aust. Math. Soc. \textbf{71} (2001),
  no.~1, 53--69. \MR{1840493}

\bibitem[HW14]{hw2}
G.~C. Hruska and Daniel~T. Wise, \emph{Finiteness properties of cubulated
  groups}, Compos. Math. \textbf{150} (2014), no.~3, 453--506. \MR{3187627}

\bibitem[JW17]{jw}
K.~Jankiewizc and D.~Wise, \emph{Cubulating small cancellation free products},
  2017.

\bibitem[KM12]{km}
Jeremy Kahn and Vladimir Markovic, \emph{Immersing almost geodesic surfaces in
  a closed hyperbolic three manifold}, Ann. of Math. (2) \textbf{175} (2012),
  no.~3, 1127--1190. \MR{2912704}

\bibitem[LW13]{lw}
Joseph Lauer and Daniel~T. Wise, \emph{Cubulating one-relator groups with
  torsion}, Math. Proc. Cambridge Philos. Soc. \textbf{155} (2013), no.~3,
  411--429. \MR{3118410}

\bibitem[Man16]{manningnotes}
Jason~F. Manning, \emph{Cubulating spaces and groups, lecture notes (working
  draft)}, 2016.

\bibitem[MS17]{ms2}
Alexandre Martin and Markus Steenbock, \emph{A combination theorem for
  cubulation in small cancellation theory over free products}, Ann. Inst.
  Fourier (Grenoble) \textbf{67} (2017), no.~4, 1613--1670. \MR{3711135}

\bibitem[Osi06]{o1}
Denis~V. Osin, \emph{Relatively hyperbolic groups: intrinsic geometry,
  algebraic properties, and algorithmic problems}, Mem. Amer. Math. Soc.
  \textbf{179} (2006), no.~843, vi+100. \MR{2182268}

\bibitem[{Per}02]{perel2}
G.~{Perelman}, \emph{{The entropy formula for the Ricci flow and its geometric
  applications}}, ArXiv Mathematics e-prints (2002).

\bibitem[{Per}03]{perel1}
\bysame, \emph{{Ricci flow with surgery on three-manifolds}}, ArXiv Mathematics
  e-prints (2003).

\bibitem[Sag95]{ms}
Michah Sageev, \emph{Ends of group pairs and non-positively curved cube
  complexes}, Proc. London Math. Soc. (3) \textbf{71} (1995), no.~3, 585--617.
  \MR{1347406}

\bibitem[SW05]{sw}
Michah Sageev and Daniel~T. Wise, \emph{The {T}its alternative for {${\rm
  CAT}(0)$} cubical complexes}, Bull. London Math. Soc. \textbf{37} (2005),
  no.~5, 706--710. \MR{2164832}

\bibitem[Thu82]{thurs}
William~P. Thurston, \emph{Three-dimensional manifolds, {K}leinian groups and
  hyperbolic geometry}, Bull. Amer. Math. Soc. (N.S.) \textbf{6} (1982), no.~3,
  357--381. \MR{648524}

\bibitem[Wis04]{w1}
D.~T. Wise, \emph{Cubulating small cancellation groups}, Geom. Funct. Anal.
  \textbf{14} (2004), no.~1, 150--214. \MR{2053602}

\bibitem[Wis09]{w2}
Daniel~T. Wise, \emph{Research announcement: the structure of groups with a
  quasiconvex hierarchy}, Electron. Res. Announc. Math. Sci. \textbf{16}
  (2009), 44--55. \MR{2558631}

\bibitem[Wis12]{wisebook}
\bysame, \emph{From riches to raags: 3-manifolds, right-angled {A}rtin groups,
  and cubical geometry}, CBMS Regional Conference Series in Mathematics, vol.
  117, Published for the Conference Board of the Mathematical Sciences,
  Washington, DC; by the American Mathematical Society, Providence, RI, 2012.
  \MR{2986461}

\end{thebibliography}

\textsc{Mathematics Department, University of Oklahoma, Norman, OK 73019, USA}

Email: \texttt{\href{mailto:bwstucky@ou.edu}{bwstucky@ou.edu}}

URL: \texttt{\href{http://benstuc.ky}{http://benstuc.ky}}

\end{document}